\documentclass[psamsfonts,reqno]{amsart}
\newtheorem{theorem}{Theorem}[section]
\newtheorem{lemma}{Lemma}[section]
\newtheorem{proposition}{Proposition}[section]
\theoremstyle{remark}

\usepackage{fullpage}
\usepackage[dvips]{graphicx}
\usepackage{subfig}
\usepackage{float}
\theoremstyle{definition}

\newcommand{\ve}{\varepsilon}
\numberwithin{equation}{section}
\numberwithin{table}{section}
\numberwithin{figure}{section}
\newcommand{\abs}[1]{\lvert#1\rvert}
\newcommand{\kp}{\kappa}
\newcommand{\wh}{\widehat}
\renewcommand{\theequation}{\thesection.\arabic{equation}}
\def\ontop#1#2{\setbox0\hbox{#2}\copy0\llap{\raise\ht0\hbox{#1}}}
\begin{document}
\newcommand{\shw}{\ensuremath{\overset{\scriptsize \circ}{S}}}
\title{NOTES ON ERROR ESTIMATES FOR THE STANDARD GALERKIN-FINITE ELEMENT METHOD FOR THE SHALLOW WATER EQUATIONS.}
\author{D.C. Antonopoulos}
%
\author{V.A. Dougalis}
\address{Department of Mathematics, University of Athens, 15784 Zographou, Greece, and
Institute of Applied and Computational Mathematics, FORTH, 70013 Heraklion, Greece}
\email{antonod@math.uoa.gr\,,  doug@math.uoa.gr}
\subjclass[2010]{65M60, 35L60;}
\keywords{Shallow water equations, first-order quasilinear hyperbolic systems, standard Galerkin method, error estimates}
\begin{abstract}
We consider a simple initial-boundary-value problem for the shallow water equations in one space dimension, and also
the analogous problem for a symmetric variant of the system. Assuming smoothness of solutions, we discretize these problems in
space using standard Galerkin-finite element methods and prove $L^{2}$-error estimates for the semidiscrete problems for
quasiuniform and uniform meshes. In particular we show that in the case of spatial discretizations with piecewise linear
continuous functions on a uniform mesh, suitable compatibility conditions at the boundary and superaccuracy properties of the
$L^{2}$ projection on the finite element subspaces lead to an optimal-order $O(h^{2})$ $L^{2}$-error estimate. We also examine
temporal discretizations of the semidiscrete problems by three explicit Runge-Kutta methods (the Euler, improved Euler, and
the Shu-Osher scheme) and prove $L^{2}$-error estimates, which are of optimal order in the temporal variable, under appropriate
stability conditions. In a final section of remarks we prove optimal-order $L^{2}$-error estimates for smooth spline spatial
discretizations of the periodic initial-value problem for the systems. We also prove that small-amplitude, appropriately
transformed solutions of the symmetric system are close to the corresponding solutions of the usual system while they are
both smooth, thus providing a justification of the symmetric system.
\end{abstract}
\maketitle
\section{Introduction}
In this paper we will analyze standard Galerkin approximations to the system of \emph{shallow water equations} (also known as
\emph{Saint-Venant equations})
\begin{equation}
\begin{aligned}
\eta_{t} & + u_{x} + (\eta u)_{x} = 0,\\
u_{t} & + \eta_{x} + uu_{x} = 0,
\end{aligned}
\label{eq11}
\end{equation}
which is an approximation of the two-dimensional Euler equations of water-wave theory that models two-way propagation of long waves
of finite amplitude on the surface of an ideal fluid in a uniform horizontal channel of finite depth, \cite{w}, \cite{p}. The
variables in \eqref{eq11} are nondimensional and unscaled; $x\in \mathbb{R}$ and $t\geq 0$ are proportional to position along the
channel and time, respectively, and $\eta=\eta(x,t)$ and $u=u(x,t)$ are proportional to the elevation of the free surface above a level
of rest corresponding to $\eta=0$, and to the horizontal velocity of the fluid at the free surface, respectively. (In these variables
the bottom of the channel lies at a depth equal to $-1$.) \par
It is well known that, given smooth initial conditions $\eta(x,0)=\eta_{0}(x)$, $u(x,0)=u_{0}(x)$, $x\in \mathbb{R}$, the initial-value
problem for \eqref{eq11} has smooth solutions in general only locally in $t$; the existence of smooth solutions may be studied by
standard methods of the theory of nonlinear hyperbolic systems, cf. e.g.  \cite{m}, Ch. 2, and \cite{t}, Ch. 16. \par
In this paper we shall consider the following initial-boundary-value problem (ibvp) for \eqref{eq11} posed on the spatial interval
$[0,1]$. We seek $\eta=\eta(x,t)$, $u=u(x,t)$, $0\leq x\leq 1$, $0\leq t\leq T$, satisfying
\begin{align}
\begin{aligned}
\eta_{t} & + u_{x} + (\eta u)_{x} = 0, \notag\\
u_{t} & + \eta_{x} + uu_{x} = 0,\notag
\end{aligned}
\quad  0\leq x\leq 1,\,\,\, 0\leq t\leq T,
\tag{SW} \label{eqsw}
\\
\eta(x,0) =\eta_{0}(x), \quad u(x,0)=u_{0}(x), \quad 0\leq x\leq 1,\hspace{-9pt}&  \notag\\
u(0,t) = 0,\quad u(1,t)=0, \quad 0\leq t\leq T. \hspace{32pt} & \notag
\end{align}
In \cite{pt} Petcu and Temam established the existence-uniqueness of $H^{2}$-solutions of \eqref{eqsw} for some
$T=T(\|\eta_{0}\|_{2},\|u_{0}\|_{2})>0$ under the hypothesis that $1+\eta_{0}(x)>0$, $x\in [0,1]$. Moreover, in the temporal interval
$[0,T]$ of existence of solutions there holds that $1+\eta>0$ for $x\in [0,1]$.  (For a precise statement of this result
see section 6.2 below.) \par
We shall also consider the analogous ibvp for a \emph{symmetric} variant of the shallow water equations, posed again on $[0,1]$.
For this purpose we seek $\eta=\eta(x,t)$, $u=u(x,t)$, $0\leq x\leq 1$, $0\leq t\leq T$, satisfying
\begin{align}
\begin{aligned}
\eta_{t} & + u_{x} + \tfrac{1}{2}(\eta u)_{x} = 0, \notag\\
u_{t} & + \eta_{x} + \tfrac{3}{2}uu_{x} + \tfrac{1}{2}\eta\eta_{x} = 0,\notag
\end{aligned}
\quad  0\leq x\leq 1,\,\,\, 0\leq t\leq T,
\tag{SSW} \label{eqssw}
\\
\eta(x,0) =\eta_{0}(x), \quad u(x,0)=u_{0}(x), \quad 0\leq x\leq 1,\hspace{23pt}&  \notag\\
u(0,t) = 0,\quad u(1,t)=0, \quad 0\leq t\leq T. \hspace{65pt} & \notag
\end{align}
Here, the nonlinear hyperbolic system is symmetric; existence-uniqueness of $H^{2}$-solutions of the ibvp \eqref{eqssw}
for $T$ sufficiently small may be established if one follows the argument of \cite{pt}, cf. Section 6.2 below. \par
We chose this symmetric system motivated by the work of Bona, Colin, and Lannes, \cite{bcl}, on completely symmetric Boussinesq-type
dispersive approximations of small-amplitude, long-wave solutions of the Euler equations. In Section 6.3 we derive the
symmetric system in the context of small-amplitude, scaled shallow water equations and study its relation to the usual
shallow water system by analytical and numerical means.\par
In the analysis of the Galerkin approximations that we pursue in this paper we generally prove in parallel error estimates
for both \eqref{eqsw} and \eqref{eqssw}. It will be seen that, as a result of the symmetry of the latter system, the proofs for
\eqref{eqssw} are more straightforward and generally hold under less stringent hypotheses compared to their \eqref{eqsw}
analogs. Let us also mention that it is easy to see that the solution of \eqref{eqssw} satisfies the $L^{2}$-conservation equation
\begin{equation}
\int_{0}^{1}\bigl(\eta^{2}(x,t) + u^{2}(x,t)\bigr)dx = \int_{0}^{1}\bigl(\eta_{0}^{2}(x) + u_{0}^{2}(x)\bigr)dx
\label{eq12}
\end{equation}
for $0\leq t\leq T$. \par
We begin the error analysis in Section 2 considering first the standard Galerkin semidiscretizations of
\eqref{eqsw} and \eqref{eqssw} using for the spatial approximation piecewise polynomial functions of order $r\geq 2$
(i.e. of degree $r-1\geq 1$) with respect to a quasiuniform mesh on $[0,1]$ of maximum meshlength $h$; the spaces consist
of $C^{k}$ functions, where $0\leq k\leq r-2$. We assume throughout that the solutions of \eqref{eqsw} and \eqref{eqssw}
are sufficiently smooth for the purposes of the error estimates. In the case of \eqref{eqssw} the error analysis is
straightforward due to the symmetry of the system and yields, for $r\geq 2$, the expected $O(h^{r-1})$ $L^{2}$-error estimate
for the Galerkin approximations of $\eta$ and $u$. (In this proof and in subsequent error estimates in this paper we compare
the Galerkin approximation with the $L^{2}$ projection of the solution of the p.d.e. problem onto the finite element subspaces
and estimate their difference.) For \eqref{eqsw} the proof is more complicated; we use a symmetrizing choice of test
function in the error equation corresponding to the second p.d.e. of \eqref{eqsw}, a `superapproximation' property of the
finite element subspaces, and the positivity of $1+\eta$ in order to establish the expected $O(h^{r-1})$ $L^{2}$-error estimates
for $\eta$ and $u$ assuming now that $r\geq 3$. This last assumption is needed in the proof for the control of the $W^{1,\infty}$
norm of an intermediate error term. Thus our proof for \eqref{eqsw} and its assumptions resemble those of the analogous proof of
Dupont, \cite{d1}, in the case of a $2\times 2$ nonlinear hyperbolic system which is close relative of \eqref{eqsw}. It is worth
noting that numerical experiments, the results of which are presented at the end of Section 2, suggest that for $r=2$, i.e. for
piecewise linear continuous functions on a quasiuniform mesh, the $L^{2}$- and $L^{\infty}$- errors of the Galerkin approximations
to $\eta$ and $u$ have $O(h)$ bounds, i.e. that the assumption $r\geq 3$ may not be needed. In fact, for special quasiuniform meshes,
e.g. for piecewise uniform or gradually varying meshes, the numerical experiments indicate that the error bounds are of $O(h^{2})$,
resembling those of the uniform mesh case (see below.) \par
In Sections 3 and 4 we examine the error of the standard Galerkin semidiscretization of \eqref{eqsw} and \eqref{eqssw} in the special
case of subspaces of continuous, piecewise linear functions on a uniform mesh on $[0,1]$. It is well known that for \emph{linear},
first-order hyperbolic equations in the uniform mesh case the standard Galerkin approximations may enjoy optimal-order
$L^{2}$-convergence, i.e. of $O(h^{r})$, as a result of superaccuracy due to cancellations in the interior mesh intervals and to
suitable compatibility conditions at the boundary, provided the solutions of the continuous problem are smooth enough. Early
evidence of this were the classic results of Dupont, \cite{d2}, in the case of $r=2$, and $r=4$ (with $k=2$, i.e. cubic splines),
and e.g. of Thom\'{e}e and Wendroff, \cite{twe}, for problems with variable coefficients in the case of subspaces consisting of
smooth splines of arbitrary order ($k=r-2$, $r\geq 2$). In these works the \emph{periodic} initial-value problem was under
consideration; the spatial periodicity and the assumed smoothness of solutions automatically furnishes the requisite compatibility
conditions at the boundary that yield superaccuracy. In Section 6.1 of \cite{adarxiv} we point out how compatibility at the boundary
for smooth solutions of a simple initial-boundary-value problem for a first-order linear hyperbolic equation gives the
superaccuracy estimate in the case $r=2$ for uniform mesh. We also refer the reader to the papers \cite{l} and \cite{zl} for
results and references to the Chinese literature on related topics. \par
In order to treat the nonlinear case, in Section 3 of the paper at hand we prove some superconvergence properties of the $L^{2}$
projections of smooth functions on $[0,1]$ satisfying suitable boundary conditions, onto spaces of piecewise linear, continuous
functions defined on a uniform mesh in $[0,1]$. The key results are Lemmas 3.3 and 3.6 in which it is shown that integrals of the form
$\int_{I_{i}}wedx$, where $w$ is a $C^{2}$ function and $e$ is the error of the $L^{2}$ projection of a $C^{4}$ function
satisfying suitable boundary conditions at $0$ and $1$, is, for any mesh interval $I_{i}$, of $O(h^{5})$. These results are used
in Section 4 where the optimal-order $O(h^{2})$ $L^{2}$-error estimates for the Galerkin semidiscretizations of \eqref{eqssw} and
\eqref{eqsw} are established. It is assumed that the ibvp's have classical, sufficiently smooth solutions, which, as a consequence
of their smoothness, must satisfy natural compatibility conditions at $0$ and $1$. Again the proof for the \eqref{eqssw} is
relatively straightforward, while in the case of \eqref{eqsw} some additional twists are needed. These theoretical results are
confirmed in numerical experiments at the end of Section 4. These also indicate that the analogous $L^{2}$ errors for
spatial discretizations with cubic splines ($k=2$, $r=4$) on uniform meshes have convergence rates which are practically equal to
4, i.e. optimal. \par
In Section 5 we turn to the temporal discretizations of the o.d.e. systems represented by the semidiscretizations considered
in Sections 2 and 4. In \cite{d1} Dupont analyzed, in the case of a system similar to the shallow water equations, the convergence
of a linearized Crank-Nicolson scheme. In the paper at hand we analyze three explicit Runge-Kutta schemes: (i) The explicit
Euler method, of first-order accuracy, which requires for stability the restrictive mesh condition $k=O(h^{2})$, where $k$ is the
(uniform) time step. (ii) The `improved' Euler method (explicit midpoint rule) of second-order accuracy, which requires for
stability the mesh condition $k=O(h^{4/3})$. (iii) An explicit, third-order Runge-Kutta method due to Shu and Osher, \cite{so}, that
needs the condition $k/h\leq \lambda_{0}$ for a small enough constant $\lambda_{0}$. (It is to be noted that these stability
restrictions are consistent with those predicted by the stability analysis of the temporal discretization with these methods of the
stiff linear systems of o.d.e.'s resulting e.g. from the standard Galerkin semidiscretization of $u_{t} + u_{x}=0$ with periodic
boundary conditions.)\par
Since our emphasis is on the energy proofs of stability and convergence of the time-discrete schemes, we chose the easiest p.d.e.
system, i.e. \eqref{eqssw}, and the most straightforward spatial discretization, i.e. the one with piecewise polynomial functions
of order $r$ on a quasiuniform mesh, in order to prove the $L^{2}$-error estimates; these have bounds of $O(k+h^{r-1})$ for
$r\geq 2$, $O(k^{2}+h^{r-1})$ for $r\geq 3$, and $O(k^{3} + h^{r-1})$ for $r\geq 3$, for the three Runge-Kutta schemes considered,
having orders of accuracy 1,2 and 3, respectively; these error estimates hold under the mesh stability conditions previously
mentioned. Similar results hold for the \eqref{eqsw} system but the proofs are omitted here. Numerical experiments in Section 5
with \eqref{eqsw} and \eqref{eqssw} indicate that the condition $k=O(h^{4/3})$ is probably necessary for the stability the improved
Euler method even in the simple case of a spatial discretization using piecewise linear, continuous functions on a uniform mesh.
The results of another numerical experiment suggest that in addition to the $L^{2}$-error, the $L^{\infty}$- and $H^{1}$-errors of
full discretizations of \eqref{eqssw} with the Shu-Osher scheme have an $O(k^{3})$ temporal behaviour. \par
We should point out that in recent years there have appeared a number of papers with proofs of error estimates of full
discretizations of Galerkin type methods with explicit Runge-Kutta methods. For example, Zhang and Shu have analyzed discontinuous
Galerkin methods for scalar conservation laws in \cite{zs1} and for symmetrizable systems of conservation laws in \cite{zs2}
using a second-order explicit Runge-Kutta method (the explicit trapezoidal rule) for time-stepping. For the DG methods analyzed
in these papers this full discretization turns out to be stable if $k=O(h)$ for a $\mathbb{P}_{1}$ spatial discretization but
needs $k$ to be of $O(h^{4/3})$ for higher-order polynomial spaces. The same Runge-Kutta scheme is used by Ying, \cite{y}, and
proved to yield a stable full discretization and the expected error estimates for a standard Galerkin method for scalar
conservation laws in several space dimensions under the condition $k=O(h^{4/3})$. In \cite{zs3} Zhang and Shu prove error estimates
for a fully discrete DG-$3^{\text{d}}$ order Shu-Osher scheme for scalar conservation laws under the condition $k=O(h)$.
In \cite{bef} Burman \emph{et al.} consider
initial-boundary value problems for first-order linear hyperbolic systems of Friedrichs type in several space dimensions, discretized
in space by a class of symmetrically stabilized finite element methods that includes DG schemes, and in time by explicit Runge-Kutta
schemes of second (RK2) and third (RK3) order of accuracy. They prove $L^{2}$-error estimates of optimal order in time and quasioptimal
in space under Courant-number restrictions for RK2 schemes with $\mathbb{P}_{1}$ elements and under the condition $k=O(h^{4/3})$ for
higher-order elements, and under Courant-number restrictions for RK3 schemes. Let us
also mention that for a closely related to the shallow water equations \emph{dispersive} system (the `classical' Boussinesq
equations), we proved error estimates in \cite{adarxiv}, \cite{ad}, for the classical, four-stage, fourth-order explicit
Runge-Kutta temporal discretization of standard Galerkin methods with cubic splines; the error bounds had an $O(k^{4})$
dependence under a $k=O(h)$ stability condition. \par
We close the paper by a series of supplementary remarks in Section 6. In Section 6.1 we consider the \emph{periodic} initial-value
problem for the shallow water system and its symmetric version and discretize it in space using the standard Galerkin method with
smooth periodic splines of order $r\geq 2$ on a uniform mesh. Using suitable \emph{quasiinterpolants} in the space of periodic
splines, cf. \cite{twe}, we prove optimal-order, i.e. $O(h^{r})$, $L^{2}$-error estimates for both systems. In Section 6.2 we
state precisely the local existence result proved by Petcu and Temam in \cite{pt} for \eqref{eqsw}; we also state the analogous
result that one may derive for \eqref{eqssw} following the proof of \cite{pt}. \par
Finally, in Section 6.3 we first recall the nondimensional \emph{scaled} form of the shallow water equations in the case of long
surface waves of \emph{small} amplitude (in which the nonlinear terms of the system are multiplied by the small parameter
$\ve=a/h_{0}$, where $a$ is a typical wave amplitude and $h_{0}$ the depth of the channel), and derive the analogous scaled form
of the symmetric shallow water equations using the nonlinear change of variables of Bona, Colin, and Lannes, \cite{bcl}.
In view of the classical theory of local existence of solutions of initial-value problems of quasilinear hyperbolic
systems and the results of \cite{bcl} we argue that the
difference in suitable norms of appropriately transformed solutions of the two systems is of $O(\ve^{2}t)$ for times $t$ up to
$O(1/\ve)$. Given that initially smooth solutions of both systems are expected in general to develop singularities after times of
$O(1/\ve)$, this result indicates that appropriately transformed, smooth, small-amplitude solutions of the symmetric system remain
close to corresponding smooth solutions of the usual system within their life span, and provides a justification for the
symmetric system. Section 6.3 closes with some numerical experiments which suggest that the difference of solutions of \eqref{eqsw}
and \eqref{eqssw} (i.e. of the ibvp's) also behaves like $\ve^{2}t$ for times up to $O(1/\ve)$. \par
In this paper we use the following notation: We let $C^{k}=C^{k}[0,1]$, $k=0,1,2,\dots$, denote the space of $k$ times continuously
differentiable functions on $[0,1]$ and define $C_{0}^{k}=\{\phi\in C^{k} : \phi(0)=\phi(1)=0\}$. For integer $k\geq 0$ $H^{k}$,
$\|\cdot\|_{k}$ will denote the usual, $L^{2}$-based Sobolev space of classes of functions on $[0,1]$ and its associated norm.
The inner product and norm on $L^{2}=L^{2}(0,1)$ is denoted simply by $\|\cdot\|$, $(\cdot,\cdot)$, respectively, while the
norms on $L^{\infty}=L^{\infty}(0,1)$ and on the $L^{\infty}$-based Sobolev space $W_{\infty}^{k}=W_{\infty}^{k}(0,1)$ by
$\|\cdot\|_{\infty}$, $\|\cdot\|_{k,\infty}$, respectively. We let $\mathbb{P}_{r}$ be the polynomials of degree $\leq r$,
and $\langle\cdot , \cdot\rangle$, $\abs{\cdot}$ be the Euclidean inner product and norm on $\mathbb{R}^{N}$. Finally, for a Banach
space $X$ of functions on $[0,1]$, $C(0,T;X)$ will denote the space of continuous maps from the interval $[0,T]$ into $X$.\\
\emph{Acknowledgment}: The authors would like to thank Prof. David Lannes for his advice on the theoretical results of Section 6.3.
and Profs. Ch. Makridakis and T. Katsaounis for helpful discussions.
\section{Semidiscretization on quasiuniform meshes}
Let $0=x_{1}<x_{2}<\dots<x_{N+1}=1$ denote a quasiuniform partition of $[0,1]$ with
$h:=\max_{i}(x_{i+1}-x_{i})$, and for integers $r$, $k$ such that $r\geq 2$, $0\leq k\leq r-2$, let
$S_{h}=S_{h}^{r,k}:=\{\phi \in C^{k} : \phi\big|_{[x_{j},x_{j+1}]} \in \mathbb{P}_{r-1}\,, 1\leq j\leq N\}$,
and $S_{h,0}=S_{h,0}^{k,r}=\{\phi \in S_{h}^{k,r}\,, \phi(0)=\phi(1)=0\}$.
It is well known that given $w\in H^{r}$ there exists an element $\chi \in S_{h}$ such that
\begin{equation}
\|w-\chi\| + h\|w' - \chi'\| \leq Ch^{r}\|w^{(r)}\|,
\stepcounter{equation}\tag{\theequation a}\label{eq21a}
\end{equation}
and if $r\geq 3$ in addition, cf. \cite{sch},
\begin{equation}
\|w-\chi\|_{2}\leq Ch^{r-2}\|w^{(r)}\|,
\label{eq21b} \tag{2.1b}
\end{equation}
for some constant $C$  independent of $h$ and $w$, and that a similar property holds in $S_{h,0}$ if
$w \in H^{r}\cap H_{0}^{1}$. Let $P$, $P_{0}$ denote the $L^{2}$-projection operators onto $S_{h}$, $S_{h,0}$,
respectively. Then, cf. \cite{ddw}, there holds that
\begin{subequations}
\begin{align}
  \|Pv\|_{\infty}&\leq C\|v\|_{\infty}\quad \text{if}\quad v\in L^{\infty},
\label{eq22a} \tag{2.2a}\\
\intertext{and}
 \|Pv - v\|_{\infty}&\leq Ch^{r}\|v\|_{r,\infty}\quad \text{if}\quad v\in W^{r,\infty},
\label{eq22b} \tag{2.2b}
\end{align}
\end{subequations}
and that a similar property holds for $P_{0}$ if $v \in W^{r,\infty}\cap H_{0}^{1}$.
(Here and in the sequel we will denote by $C$ generic constants independent of discretization parameters).\par
As a consequence of the quasiuniformity of the mesh the inverse inequalities
\begin{align}
\|\chi\|_{1} & \leq C h^{-1}\|\chi\|,
\label{eq23}\\
\|\chi\|_{j,\infty} & \leq Ch^{-(j+1/2)} \|\chi\|,\quad j=0,1,
\label{eq24}
\end{align}
hold for $\chi \in S_{h}$. (In \eqref{eq24} $\|\chi\|_{0,\infty}=\|\chi\|_{\infty}$). \par
We let the {\it{standard Galerkin semidiscretization}} of \eqref{eqsw} be defined as follows: We seek
$\eta_{h} : [0,T] \to S_{h}$, $u_{h} : [0,T] \to S_{h,0}$, such that for $t \in [0,T]$
\begin{equation}
\begin{aligned}
(\eta_{ht},\phi) + (u_{hx},\phi) + ((\eta_{h}u_{h})_{x},\phi) & = 0, \quad \forall \phi \in S_{h}\,,\\
(u_{ht},\chi) + (\eta_{hx},\chi) + (u_{h}u_{hx},\chi) & = 0, \quad \forall \chi \in S_{h,0}\,,
\end{aligned}
\label{eq25}
\end{equation}
with initial conditions
\begin{equation}
\eta_{h}(0) = P\eta_{0}\,, \qquad u_{h}(0)=P_{0}u_{0}\,.
\label{eq26}
\end{equation}
Similarly, we define the analogous semidiscretization of \eqref{eqssw}, given for $t\in [0,T]$ by
\begin{equation}
\begin{aligned}
(\eta_{ht},\phi) + (u_{hx},\phi) + \tfrac{1}{2}((\eta_{h}u_{h})_{x},\phi) & = 0,
\quad \forall \phi \in S_{h},\\
(u_{ht},\chi) + (\eta_{hx},\chi) + \tfrac{3}{2}(u_{h}u_{hx},\chi) + \tfrac{1}{2}(\eta_{h}\eta_{hx},\chi)& = 0,
\quad \forall \chi \in S_{h,0},
\end{aligned}
\label{eq27}
\end{equation}
with
\begin{equation}
\eta_{h}(0) = P\eta_{0}\,, \qquad u_{h}(0)=P_{0}u_{0}.
\label{eq28}
\end{equation}
Upon choice of bases for $S_{h}$, $S_{h}^{0}$, it is seen that the semidiscrete problems
\eqref{eq25}-\eqref{eq26} and \eqref{eq27}-\eqref{eq28} represent initial-value problems (ivp's)
for systems of o.d.e's. Clearly, these ivp's have unique solutions at least locally in time.
One conclusion of Propositions 2.1 and 2.2 is that they possess unique solutions up to at least $t=T$,
where $[0,T]$ is the interval of existence of smooth solutions of \eqref{eqsw} or \eqref{eqssw} as the case may be.
We start with the error analysis of the semidiscrete symmetric system \eqref{eq27}-\eqref{eq28}, which is quite
straightforward, due to the symmetry of \eqref{eqssw}.
\begin{proposition}
Let $(\eta,u)$ be the solution of \eqref{eqssw}. Then the semidiscrete ivp \eqref{eq27}-\eqref{eq28} has a unique
solution $(\eta_{h},u_{h})$ for $0\leq t\leq T$ satisfying
\begin{equation}
\max_{0\leq t\leq T}\bigl(\|\eta(t) - \eta_{h}(t)\| + \|u(t)-u_{h}(t)\|\bigr) \leq C h^{r-1}.
\label{eq29}
\end{equation}
\end{proposition}
\begin{proof} Setting $\phi = \eta_{h}$ and $\chi = u_{h}$ in \eqref{eq27} and adding the resulting equations
we obtain the discrete analog of \eqref{eq12}, i.e. that the conservation property
\begin{equation}
\|\eta_{h}(t)\|^{2} + \|u_{h}(t)\|^{2} = \|\eta_{h}(0)\|^{2} + \|u_{h}(0)\|^{2}
\label{eq210}
\end{equation}
holds in the interval of existence of solutions of \eqref{eq27}-\eqref{eq28}. By standard o.d.e. theory we conclude
that the ivp \eqref{eq27}-\eqref{eq28} possesses unique solutions in any finite temporal interval $[0,t^{*}]$ and in
particular in $[0,T]$.  \par
We now let $\rho := \eta - P\eta$, $\theta := P\eta - \eta_{h}$, $\sigma := u-P_{0}u$, $\xi := P_{0}u - u_{h}$.
Using \eqref{eqssw} and \eqref{eq27}-\eqref{eq28} we obtain for $0\leq t\leq T$,
\small
\begin{align}
(\theta_{t},\phi) + (\sigma_{x} + \xi_{x},\phi) +
\tfrac{1}{2}((\eta u - \eta_{h}u_{h})_{x},\phi) & = 0, \quad \forall \phi \in S_{h},
\label{eq211} \\
(\xi_{t},\chi) + (\rho_{x}+\theta_{x},\chi) + \tfrac{3}{2}(uu_{x}-u_{h}u_{hx},\chi)
+\tfrac{1}{2}(\eta\eta_{x}-\eta_{h}\eta_{hx},\chi) & = 0, \quad \forall \chi \in S_{h,0}.
\label{eq212}
\end{align}
\normalsize
Note that
\begin{align*}
\eta u - \eta_{h}u_{h} & = \eta (\sigma + \xi) + u(\rho + \theta) - (\rho + \theta)(\sigma + \xi), \\
uu_{x} - u_{h}u_{hx} & = (u\sigma)_{x} + (u\xi)_{x} - (\sigma \xi)_{x} - \sigma\sigma_{x} - \xi\xi_{x},\\
\eta\eta_{x} - \eta_{h}\eta_{hx} & = (\eta\rho)_{x} -\theta\theta_{x} + (\eta\theta)_{x} - (\rho\theta)_{x} - \rho\rho_{x}.
\end{align*}
Take $\phi=\theta$ in $(\ref{eq211})$ and obtain, for $0\leq t\leq T$, using integration by parts
\begin{equation}
\begin{aligned}
\tfrac{1}{2}\tfrac{d}{dt}\|\theta\|^{2} & + \bigl([(1+\tfrac{\eta}{2})\xi]_{x},\theta\bigr)= - (\sigma_{x},\theta)
-\tfrac{1}{2}((\eta\sigma)_{x},\theta)\\
& -\tfrac{1}{2}((u\rho)_{x},\theta) - \tfrac{1}{2}((u\theta)_{x},\theta) +\tfrac{1}{2}((\rho\sigma)_{x},\theta)
+ \tfrac{1}{2}((\theta\sigma)_{x},\theta) \\
& +\tfrac{1}{2}((\rho\xi)_{x},\theta) + \tfrac{1}{2}((\theta\xi)_{x},\theta).
\end{aligned}
\label{eq213}
\end{equation}
We now examine the various terms in the r.h.s. of \eqref{eq213}. Integration by parts yields
\[
((\theta\xi)_{x},\theta) = \tfrac{1}{2}(\xi_{x}\theta,\theta).
\]
Using now \eqref{eq21a}, \eqref{eq22b}, \eqref{eq23}, \eqref{eq24} and integration by parts we have
\begin{align*}
\abs{(\sigma_{x},\theta)} &\leq \|\sigma_{x}\|\|\theta\|\leq Ch^{r-1}\|\theta\|,\\
\abs{((\eta\sigma)_{x},\theta)} & \leq C\|\sigma\|_{1}\|\theta\|\leq Ch^{r-1}\|\theta\|, \\
\abs{((u\rho)_{x},\theta)} & \leq C\|\rho\|_{1}\|\theta\| \leq Ch^{r-1}\|\theta\|,\\
\abs{((u\theta)_{x},\theta)}& =\tfrac{1}{2}\abs{(u_{x}\theta,\theta)}\leq C\|\theta\|^{2},\\
\abs{((\rho\sigma)_{x},\theta)}&\leq \|\rho\|_{\infty}\|\sigma_{x}\|\|\theta\| +
\|\sigma\|_{\infty}\|\rho_{x}\|\|\theta\|\leq Ch^{2r-1}\|\theta\|,\\
\abs{((\theta\sigma)_{x},\theta)}&=\tfrac{1}{2}\abs{(\sigma_{x}\theta,\theta)}
\leq C\|\sigma_{x}\|_{\infty}\|\theta\|^{2}\leq C\|\theta\|^{2},\\
\abs{((\rho\xi)_{x},\theta)}&\leq \|\rho_{x}\|_{\infty}\|\xi\|\|\theta\| + \|\rho\|_{\infty}\|\xi_{x}\|\|\theta\|
\leq C\|\xi\|\|\theta\|.
\end{align*}
Therefore \eqref{eq213} and the above yield for $0\leq t\leq T$
\begin{equation}
\tfrac{1}{2}\tfrac{d}{dt}\|\theta\|^{2} + \bigl([(1+\tfrac{\eta}{2})\xi]_{x},\theta\bigr)
\leq \tfrac{1}{4}(\xi_{x}\theta,\theta) + C(h^{r-1}\|\theta\| + \|\theta\|^{2} + \|\xi\|^{2}).
\label{eq214}
\end{equation}
Take now $\chi=\xi$ in \eqref{eq212}. Then for $0\leq t\leq T$ using integration by parts we have
\begin{equation}
\begin{aligned}
\tfrac{1}{2}\tfrac{d}{dt}\|\xi\|^{2} & - \bigl([(1+\tfrac{\eta}{2})\xi]_{x},\theta\bigr)= - \tfrac{1}{4}(\xi_{x}\theta,\theta)
-(\rho_{x},\xi) - \tfrac{3}{2}((u\sigma)_{x},\xi)\\
& -\tfrac{3}{2}((u\xi)_{x},\xi) + \tfrac{3}{2}((\sigma\xi)_{x},\xi) +\tfrac{3}{2}(\sigma\sigma_{x},\xi)
- \tfrac{1}{2}((\eta\rho)_{x},\xi) - \tfrac{1}{2}(\eta_{x}\theta,\xi) \\
& +\tfrac{1}{2}((\rho\theta)_{x},\xi) + \tfrac{1}{2}(\rho\rho_{x},\xi).
\end{aligned}
\label{eq215}
\end{equation}
Using again \eqref{eq21a}-\eqref{eq24} and integration by parts we see that
\begin{align*}
\abs{(\rho_{x},\xi)} &\leq \|\rho_{x}\|\|\xi\|\leq Ch^{r-1}\|\xi\|,\\
\abs{((u\sigma)_{x},\xi)} & \leq C\|\sigma\|_{1}\|\xi\|\leq Ch^{r-1}\|\xi\|, \\
\abs{((u\xi)_{x},\xi)} & =\tfrac{1}{2}\abs{(u_{x}\xi,\xi)} \leq C\|\xi\|^{2},\\
\abs{((\sigma\xi)_{x},\xi)}& =\tfrac{1}{2}\abs{(\sigma_{x}\xi,\xi)}\leq C\|\sigma_{x}\|_{\infty}\|\xi\|^{2}\leq C\|\xi\|^{2},\\
\abs{(\sigma\sigma_{x},\xi)}&\leq \|\sigma\|_{\infty}\|\sigma_{x}\|\|\xi\| \leq Ch^{2r-1}\|\xi\|,\\
\abs{((\eta\rho)_{x},\xi)}& \leq C\|\rho\|_{1}\|\xi\| \leq Ch^{r-1}\|\xi\|,\\
\abs{(\eta_{x}\theta,\xi)}&\leq C\|\theta\|\|\xi\|,\\
\abs{((\rho\theta)_{x},\xi)}& \leq\|\rho_{x}\|_{\infty}\|\theta\|\|\xi\|
+ \|\rho\|_{\infty}\|\theta_{x}\|\|\xi\|\leq C\|\theta\|\|\xi\|,\\
\abs{(\rho\rho_{x},\xi)}&\leq \|\rho\|_{\infty}\|\rho_{x}\|\|\xi\| \leq Ch^{2r-1}\|\xi\|.
\end{align*}
Therefore, by \eqref{eq215}, for $0\leq t\leq T$
\begin{equation}
\tfrac{1}{2}\tfrac{d}{dt}\|\xi(t)\|^{2} - \bigl([(1+\tfrac{\eta}{2})\xi]_{x},\theta\bigr)
\leq -\tfrac{1}{4}(\xi_{x}\theta,\theta) + C(h^{r-1}\|\xi\| + \|\theta\|^{2} + \|\xi\|^{2}).
\label{eq216}
\end{equation}
Adding \eqref{eq214} and \eqref{eq216} gives
\[
\tfrac{d}{dt}(\|\xi\|^{2}  + \|\theta\|^{2})\leq C[h^{r-1}(\|\theta\| + \|\xi\|) + \|\theta\|^{2} + \|\xi\|^{2}],\quad
0\leq t\leq T.
\]
Therefore, by Gronwall's inequality and \eqref{eq26} we see that
\[
\|\theta\| + \|\xi\| \leq Ch^{r-1}, \quad 0\leq t\leq T,
\]
from which \eqref{eq29} follows.
\end{proof}
We turn now to the semidiscrete approximation to the \eqref{eqsw}. The error analysis that follows is similar to that
of Dupont \cite{d1} and the proof assumes that $r\geq 3$ and that the solution of \eqref{eqsw} satisfies $1+\eta>0$,
cf. \cite{pt} and Section 6.2.
\begin{proposition} Let $(\eta,u)$ be the solution of \eqref{eqsw}, satisfying $1+\eta>0$ for $t\in [0,T]$, $r\geq 3$, and
$h$ be sufficiently small. Then the semidiscrete ivp \eqref{eq25}-\eqref{eq26} has a unique solution $(\eta_{h},u_{h})$ for
$0\leq t\leq T$ satisfying
\begin{equation}
\max_{0\leq t\leq T}\bigl(\|\eta(t) - \eta_{h}(t)\| + \|u(t)-u_{h}(t)\|\bigr) \leq C h^{r-1}.
\label{eq217}
\end{equation}
\end{proposition}
\begin{proof}
We use the same notation as in the proof of Proposition 2.1. While the solution of \eqref{eq25}-\eqref{eq26} exists we have
\small
\begin{align}
(\theta_{t},\phi) + (\xi_{x}+\sigma_{x},\phi) + ((\eta u)_{x} - (\eta_{h}u_{h})_{x},\phi) & = 0, \quad \forall \phi \in S_{h},
\label{eq218} \\
(\xi_{t},\chi) + (\theta_{x} + \rho_{x},\chi) + (uu_{x}-u_{h}u_{hx},\chi) & = 0, \quad \forall \chi \in S_{h,0}.
\label{eq219}
\end{align}
\normalsize
Taking $\phi=\theta$ in \eqref{eq218} and using integration by parts we have
\begin{equation}
\begin{aligned}
\tfrac{1}{2}\tfrac{d}{dt}\|\theta\|^{2} & + \bigl([(1+\eta)\xi]_{x},\theta\bigr)= - (\sigma_{x},\theta)
-((\eta\sigma)_{x},\theta)-((u\rho)_{x},\theta) \\
& - ((u\theta)_{x},\theta) + ((\rho\sigma)_{x},\theta) + ((\theta\sigma)_{x},\theta)
+ ((\rho\xi)_{x},\theta) + ((\theta\xi)_{x},\theta).
\end{aligned}
\label{eq220}
\end{equation}
In view of \eqref{eq26}, by continuity we conclude that there exists a maximal temporal instance $t_{h}>0$ such that
$(\eta_{h},u_{h})$ exist and $\|\xi_{x}\|_{\infty}\leq 1$ for $t\leq t_{h}$. Suppose that $t_{h}<T$. Using
$(2.1)$-\eqref{eq24} and integration by parts we may then estimate the various terms in the r.h.s. of \eqref{eq220}
for $t\in[0,t_{h}]$ as follows
\begin{align*}
\abs{(\sigma_{x},\theta)} &\leq \|\sigma_{x}\|\|\theta\|\leq Ch^{r-1}\|\theta\|,\\
\abs{((\eta\sigma)_{x},\theta)} & \leq C\|\sigma\|_{1}\|\theta\|\leq Ch^{r-1}\|\theta\|, \\
\abs{((u\rho)_{x},\theta)} & \leq C\|\rho\|_{1}\|\theta\| \leq Ch^{r-1}\|\theta\|,\\
\abs{((u\theta)_{x},\theta)}& =\tfrac{1}{2}\abs{(u_{x}\theta,\theta)}\leq C\|\theta\|^{2},\\
\abs{((\rho\sigma)_{x},\theta)}&\leq \|\rho\|_{\infty}\|\sigma_{x}\|\|\theta\| +
\|\sigma\|_{\infty}\|\rho_{x}\|\|\theta\|\leq Ch^{2r-1}\|\theta\|,\\
\abs{((\theta\sigma)_{x},\theta)}&=\tfrac{1}{2}\abs{(\sigma_{x}\theta,\theta)}
\leq C\|\sigma_{x}\|_{\infty}\|\theta\|^{2}\leq C\|\theta\|^{2},\\
\abs{((\rho\xi)_{x},\theta)}&\leq \|\rho_{x}\|_{\infty}\|\xi\|\|\theta\| + \|\rho\|_{\infty}\|\xi_{x}\|\|\theta\|
\leq C\|\xi\|\|\theta\|,\\
\abs{((\theta\xi)_{x},\theta)}&=\tfrac{1}{2}\abs{(\xi_{x}\theta,\theta)}
\leq\tfrac{1}{2}\|\xi_{x}\|_{\infty}\|\theta\|^{2}\leq \tfrac{1}{2}\|\theta\|^{2}.
\end{align*}
Hence, we conclude from \eqref{eq220} that for $t\in[0,t_{h}]$
\begin{equation}
\tfrac{1}{2}\tfrac{d}{dt}\|\theta\|^{2} - (\gamma,\theta_{x}) \leq C(h^{r-1}\|\theta\| + \|\theta\|^{2} + \|\xi\|^{2}),
\label{eq221}
\end{equation}
where we have put $\gamma:=(1+\eta)\xi$. \par
We turn now to \eqref{eq219} in which we set $\chi=P_{0}\gamma=P_{0}[(1+\eta)\xi]$. Then for $0\leq t\leq t_{h}$ it holds
that
\begin{equation}
\begin{aligned}
(\xi_{t},\gamma)  + (\theta_{x},P_{0}\gamma) = & - (\rho_{x},P_{0}\gamma) - ((u\sigma)_{x},P_{0}\gamma) - ((u\xi)_{x},P_{0}\gamma)\\
 & + ((\sigma\xi)_{x},P_{0}\gamma) + (\sigma\sigma_{x},P_{0}\gamma) + (\xi\xi_{x},P_{0}\gamma).
\end{aligned}
\label{eq222}
\end{equation}
For the first two terms in the r.h.s. of \eqref{eq222} we have
\begin{align*}
\abs{(\rho_{x},P_{0}\gamma)} & \leq \|\rho_{x}\|\|P_{0}\gamma\|\leq Ch^{r-1}\|\gamma\|\leq Ch^{r-1}\|\xi\|,\\
\abs{((u\sigma)_{x},P_{0}\gamma)} &\leq C\|\sigma\|_{1}\|P_{0}\gamma\| \leq Ch^{r-1}\|\xi\|.
\end{align*}
Note now that
\begin{align*}
((u\xi)_{x},P_{0}\gamma) & = ((u\xi)_{x},P_{0}\gamma-\gamma) + ((u\xi)_{x},\gamma)\\
& = ((u\xi)_{x},P_{0}\gamma-\gamma) + (u_{x}\xi,(1+\eta)\xi) + (u\xi_{x},(1+\eta)\xi)\\
& = ((u\xi)_{x},P_{0}\gamma-\gamma) + (u_{x}(1+\eta),\xi^{2}) - \tfrac{1}{2}\bigl([(1+\eta)u]_{x},\xi^{2}\bigr).
\end{align*}
We now use a well-known {\it{superapproximation}} property of $S_{h,0}$, cf. \cite{d1}, \cite{ddw}, to estimate the term
$P_{0}\gamma-\gamma$:
\begin{equation}
\|P_{0}\gamma - \gamma\|=\|P_{0}[(1+\eta)\xi] - (1+\eta)\xi\| \leq Ch\|\xi\|.
\label{eq223}
\end{equation}
Therefore, by \eqref{eq23}
\[
\abs{((u\xi)_{x},P_{0}\gamma)} \leq Ch\|\xi\|_{1}\|\xi\| + C\|\xi\|^{2}\leq C\|\xi\|^{2}.
\]
Similarly, using $(2.1)$-\eqref{eq24} and \eqref{eq223} we have
\begin{align*}
\abs{((\sigma\xi)_{x},P_{0}\gamma)} & \leq \abs{(\sigma_{x}\xi,P_{0}\gamma-\gamma)}
+ \abs{(\sigma\xi_{x},P_{0}\gamma-\gamma)} + \abs{((\sigma\xi)_{x},\gamma)} \\
& \leq C\|\sigma_{x}\|_{\infty}h\|\xi\|^{2} + C\|\sigma\|_{\infty}\|\xi_{x}\|h\|\xi\|
+ C \|\sigma_{x}\|_{\infty}\|\xi\|^{2} + C\|\sigma\|_{\infty}\|\xi_{x}\|\|\xi\|\\
& \leq C\|\xi\|^{2},
\end{align*}
\begin{align*}
\abs{(\sigma\sigma_{x},P_{0}\gamma)} & \leq \abs{(\sigma\sigma_{x},P_{0}\gamma-\gamma)}
+ \abs{(\sigma\sigma_{x},\gamma)}  \\
& \leq C\|\sigma\|_{\infty}\|\sigma_{x}\|h\|\xi\| + C\|\sigma\|_{\infty}\|\sigma_{x}\|\|\xi\|\\
& \leq Ch^{2r-1}\|\xi\|,
\end{align*}
\begin{align*}
\abs{(\xi\xi_{x},P_{0}\gamma)} & \leq \abs{(\xi\xi_{x},P_{0}\gamma-\gamma)}
+ \abs{(\xi\xi_{x},(1+\eta)\xi)} \\
& \leq C\|\xi_{x}\|_{\infty}h\|\xi\|^{2} + C\|\xi_{x}\|_{\infty}\|\xi\|^{2}\\
& \leq C\|\xi\|^{2}.
\end{align*}
Therefore, using \eqref{eq222} we have for $0\leq t\leq t_{h}$
\begin{equation}
(\xi_{t},(1+\eta)\xi) + (\theta_{x},P_{0}\gamma) \leq C(h^{r-1}\|\xi\| + \|\xi\|^{2}).
\label{eq224}
\end{equation}
Adding now \eqref{eq221} and \eqref{eq224} we obtain
\[
\tfrac{1}{2}\tfrac{d}{dt}\|\theta\|^{2} + (\xi_{t},(1+\eta)\xi) +
(\theta_{x},P_{0}\gamma-\gamma) \leq C\bigl[h^{r-1}(\|\theta\| + \|\xi\|) + \|\theta\|^{2} + \|\xi\|^{2}\bigr].
\]
But
\[
(\xi_{t},(1+\eta)\xi) = \tfrac{1}{2}\tfrac{d}{dt}((1+\eta)\xi,\xi) - \tfrac{1}{2}(\eta_{t}\xi,\xi).
\]
Therefore, for $0\leq t\leq t_{h}$
\[
\tfrac{1}{2}\tfrac{d}{dt}\bigl[\|\theta\|^{2} + ((1+\eta)\xi,\xi)\bigr]
\leq C\bigl[h^{r-1}(\|\theta\| + \|\xi\|) + \|\theta\|^{2} + \|\xi\|^{2}\bigr],
\]
for a constant $C$ independent of $h$ and $t_{h}$.\par
Since $1+\eta>0$, the norm $((1+\eta)\cdot,\cdot)^{1/2}$ is equivalent to that of $L^{2}$ uniformly for $t\in[0,T]$.
Hence, Gronwall's inequality and \eqref{eq26} yield for a constant $C=C(T)$
\begin{equation}
\|\theta\| + \|\xi\| \leq Ch^{r-1}\quad for \quad 0\leq t\leq t_{h}.
\label{eq225}
\end{equation}
We conclude from \eqref{eq24} that $\|\xi_{x}\|_{\infty}\leq Ch^{r-5/2}$ for $0\leq t\leq t_{h}$, and, since
$r\geq 3$, if $h$ was sufficiently small, we see that $t_{h}$ is not maximal. Hence we may take $t_{h}=T$ and
\eqref{eq217} follows from \eqref{eq225}.
\end{proof}
We close this section with some numerical experiments. Table \ref{tbl21} shows the errors and associated orders of convergence
in the $L^{2}$ and $L^{\infty}$ norms at $t=1$ of the standard Galerkin approximation with piecewise linear continuous
functions (i.e. $r=2$) of \eqref{eqssw} with suitable right-hand side and initial conditions so that its exact solution
is $\eta = \exp(2t)(\cos(\pi x) + x + 2)$, $u=\exp(-xt)\sin(\pi x)$. The semidiscrete i.v.p. was integrated in time with
the `classical', four-stage, fourth-order explicit Runge-Kutta (RK) method using a time step $k=\Delta x/20$.
\scriptsize
\begin{table}[h]
\begin{tabular}[t]{ | c | c | c | c | c | c | c | c | c | }\hline
 & \multicolumn{4}{| c |}{$L^{2}-errors$} & \multicolumn{4}{| c |}{$L^{\infty}-errors$} \\ \hline
$N$   &    $\eta$      &  $order$  &      $u$       &  $order$ &
$\eta$     &  $order$  &      $u$       & $order$    \\ \hline
$40$  & $0.7332(-1)$ & $      $ & $0.1169(-2)$ & $      $ & $0.1412$     & $      $ & $0.3316(-2)$ & $      $  \\ \hline
$80$  & $0.3605(-1)$ & $1.0240$ & $0.7283(-3)$ & $0.6830$ & $0.7116(-1)$ & $0.9882$ & $0.2351(-2)$ & $0.4961$  \\ \hline
$120$ & $0.2390(-1)$ & $1.0135$ & $0.4655(-3)$ & $1.1039$ & $0.4700(-1)$ & $1.0230$ & $0.1314(-2)$ & $1.4357$  \\ \hline
$160$ & $0.1788(-1)$ & $1.0096$ & $0.3329(-3)$ & $1.1660$ & $0.3522(-1)$ & $1.0031$ & $0.9258(-3)$ & $1.2162$  \\ \hline
$200$ & $0.1428(-1)$ & $1.0077$ & $0.2601(-3)$ & $1.1052$ & $0.2816(-1)$ & $1.0019$ & $0.7443(-3)$ & $0.9778$  \\ \hline
$240$ & $0.1189(-1)$ & $1.0063$ & $0.2138(-3)$ & $1.0752$ & $0.2340(-1)$ & $1.0161$ & $0.5879(-3)$ & $1.2941$  \\ \hline
$280$ & $0.1018(-1)$ & $1.0054$ & $0.1815(-3)$ & $1.0614$ & $0.2006(-1)$ & $0.9990$ & $0.4983(-3)$ & $1.0727$  \\ \hline
$320$ & $0.8901(-2)$ & $1.0046$ & $0.1580(-3)$ & $1.0420$ & $0.1754(-1)$ & $1.0068$ & $0.4240(-3)$ & $1.2088$  \\ \hline
$360$ & $0.7908(-2)$ & $1.0041$ & $0.1398(-3)$ & $1.0345$ & $0.1557(-1)$ & $1.0103$ & $0.3703(-3)$ & $1.1505$  \\ \hline
$400$ & $0.7115(-2)$ & $1.0037$ & $0.1254(-3)$ & $1.0308$ & $0.1402(-1)$ & $0.9954$ & $0.3328(-3)$ & $1.0139$  \\ \hline
$440$ & $0.6466(-2)$ & $1.0033$ & $0.1139(-3)$ & $1.0235$ & $0.1274(-1)$ & $1.0047$ & $0.3062(-3)$ & $0.8731$  \\ \hline
$480$ & $0.5925(-2)$ & $1.0030$ & $0.1041(-3)$ & $1.0200$ & $0.1167(-1)$ & $1.0036$ & $0.2818(-3)$ & $0.9556$  \\ \hline
$520$ & $0.5468(-2)$ & $1.0028$ & $0.9597(-4)$ & $1.0187$ & $0.1078(-1)$ & $0.9975$ & $0.2616(-3)$ & $0.9259$  \\ \hline
$560$ & $0.5077(-2)$ & $1.0026$ & $0.8901(-4)$ & $1.0157$ & $0.1000(-1)$ & $1.0069$ & $0.2447(-3)$ & $0.9025$  \\ \hline
$600$ & $0.4738(-2)$ & $1.0024$ & $0.8300(-4)$ & $1.0131$ & $0.9334(-2)$ & $1.0032$ & $0.2285(-3)$ & $0.9930$  \\ \hline
$640$ & $0.4441(-2)$ & $1.0023$ & $0.7705(-4)$ & $1.0125$ & $0.8751(-2)$ & $0.9991$ & $0.2148(-3)$ & $0.9593$  \\ \hline
\end{tabular}
\normalsize
\caption{
Errors and orders of convergence. \eqref{eqssw} system, standard Galerkin semidiscretization
with piecewise linear, continuous elements on a quasiuniform mesh, $t=1$.}
\label{tbl21}
\end{table}
\normalsize
(This method is stable for systems like \eqref{eqssw} and \eqref{eqsw} for $k/\Delta x$ sufficiently small. We checked
that the temporal error of the discretization was very small compared with the spatial error, so that the errors and
rates of convergence shown are essentially those of the semidiscrete problem.) On the spatial interval $[0,1]$
we used the quasiuniform mesh given by $h_{2i-1}=0.75\Delta x$, $h_{2i}=0.5\Delta x$, $i=1,\dots,N/2$, where
$h_{i}=x_{i+1}-x_{i}$ and $\Delta x=1.6/N$. The table suggests that the $L^{2}$-errors for $\eta$ and $u$ are
of $O(h)$, thus confirming the result of Proposition 2.1. It also suggests that the $L^{\infty}$-errors are also $O(h)$.
(The $H^{1}$-errors were found to be of $O(1)$).
\scriptsize
\begin{table}[h]
\begin{tabular}[t]{ | c | c | c | c | c | c | c | c | c | }\hline
 & \multicolumn{4}{| c |}{$L^{2}-errors$} & \multicolumn{4}{| c |}{$L^{\infty}-errors$} \\ \hline
$N$   &    $\eta$      &  $order$  &      $u$       &  $order$ &
$\eta$     &  $order$  &      $u$       & $order$    \\ \hline
$40$  & $0.1216$ & $      $     & $0.1749(-2)$ & $      $ & $0.2099$     & $      $ & $0.4090(-2)$ & $      $  \\ \hline
$80$  & $0.5973(-1)$ & $1.0260$ & $0.8259(-2)$ & $1.0827$ & $0.1051$     & $0.9970$ & $0.1935(-2)$ & $1.1080$  \\ \hline
$120$ & $0.3958(-1)$ & $1.0149$ & $0.5467(-3)$ & $1.0174$ & $0.6960(-1)$ & $1.0177$ & $0.1332(-2)$ & $0.9219$  \\ \hline
$160$ & $0.2959(-1)$ & $1.0106$ & $0.4092(-3)$ & $1.0068$ & $0.5188(-1)$ & $1.0210$ & $0.1015(-2)$ & $0.9423$  \\ \hline
$200$ & $0.2363(-1)$ & $1.0082$ & $0.3271(-3)$ & $1.0039$ & $0.4146(-1)$ & $1.0049$ & $0.8080(-3)$ & $1.0245$  \\ \hline
$240$ & $0.1967(-1)$ & $1.0067$ & $0.2724(-3)$ & $1.0035$ & $0.3458(-1)$ & $0.9960$ & $0.6834(-3)$ & $0.9186$  \\ \hline
$280$ & $0.1684(-1)$ & $1.0057$ & $0.2334(-3)$ & $1.0034$ & $0.2961(-1)$ & $1.0067$ & $0.5823(-3)$ & $1.0380$  \\ \hline
$320$ & $0.1473(-1)$ & $1.0049$ & $0.2041(-3)$ & $1.0036$ & $0.2587(-1)$ & $1.0112$ & $0.5115(-3)$ & $0.9718$  \\ \hline
$360$ & $0.1309(-1)$ & $1.0044$ & $0.1814(-3)$ & $1.0034$ & $0.2299(-1)$ & $1.0009$ & $0.4521(-3)$ & $1.0468$  \\ \hline
$400$ & $0.1177(-1)$ & $1.0039$ & $0.1632(-3)$ & $1.0033$ & $0.2070(-1)$ & $0.9977$ & $0.4095(-3)$ & $0.9395$  \\ \hline
$440$ & $0.1070(-1)$ & $1.0035$ & $0.1483(-3)$ & $1.0031$ & $0.1880(-1)$ & $1.0061$ & $0.3716(-3)$ & $1.0184$  \\ \hline
$480$ & $0.9804(-2)$ & $1.0032$ & $0.1359(-3)$ & $1.0029$ & $0.1723(-1)$ & $1.0066$ & $0.3418(-3)$ & $0.9631$  \\ \hline
$520$ & $0.9048(-2)$ & $1.0030$ & $0.1254(-3)$ & $1.0027$ & $0.1590(-1)$ & $0.9986$ & $0.3144(-3)$ & $1.0446$  \\ \hline
$560$ & $0.8400(-2)$ & $1.0028$ & $0.1164(-3)$ & $1.0025$ & $0.1477(-1)$ & $0.9996$ & $0.2924(-3)$ & $0.9770$  \\ \hline
$600$ & $0.7839(-2)$ & $1.0026$ & $0.1087(-3)$ & $1.0024$ & $0.1378(-1)$ & $1.0071$ & $0.2727(-3)$ & $1.0112$  \\ \hline
$640$ & $0.7347(-2)$ & $1.0024$ & $0.1019(-3)$ & $1.0023$ & $0.1291(-1)$ & $1.0048$ & $0.2562(-3)$ & $0.9692$  \\ \hline
\end{tabular}
\normalsize
\caption{
Errors and orders of convergence. \eqref{eqsw} system, standard Galerkin semidiscretization
with piecewise linear, continuous elements on a quasiuniform mesh, $t=1$.}
\label{tbl22}
\end{table}
\normalsize
The proof of Proposition 2.2 for the analysis of the semidiscretization of \eqref{eqsw} needs the assumption that
$r\geq 3$. Table \ref{tbl22} suggests that the result holds for $r=2$, i.e. for piecewise linear continuous functions, as well.
Specifically, for the quasiuniform spatial mesh and the exact solution used for the computations of Table \ref{tbl21}, it suggests
that the $L^{2}$-errors for $\eta$ and $u$ at $t=1$ are of $O(h)$; the same rate is suggested for the $L^{\infty}$-errors
as well. (The $H^{1}$-errors were of $O(1)$.) Again, we used the `classical' fourth-order RK method with $k=\Delta x/10$
for time stepping. \par
For special quasiuniform meshes (e.g. piecewise uniform and slowly varying meshes) numerical experiments indicate that the
$L^{2}$- and $L^{\infty}$- errors are $O(h^{2})$, as in the case of a uniform mesh (cf. Section 4). For example, in Table \ref{tbl23}
we present the $L^{2}$- and $L^{\infty}$- errors and orders of convergence in the case of the same \eqref{eqsw} example used for
computing the results of Table \ref{tbl22} and the same temporal discretization. The spatial mesh that we used was piecewise uniform and
was defined, for $N$  a given integer, by taking
\begin{align*}
& h_{i} = \frac{1}{4N_{1}}\,, \qquad \quad 1\leq i\leq N_{1}\,, \quad N_{1}=\frac{3N}{10}\,, \quad \text{on}\,\,\, [0,0.25]\,,\\[3pt]
& h_{i+N_{1}}=\frac{1}{2N_{2}}\,, \qquad 1\leq i\leq N_{2}\,, \quad N_{2}=\frac{5N}{10}\,, \quad \text{on}\,\,\, [0.25, 0.75]\,,\\[3pt]
& h_{i+N_{1}+N_{2}} = \frac{1}{4N_{3}}\,, \quad 1\leq i\leq N_{3}\,, \quad N_{3}=\frac{2N}{10}\,,\quad  \text{on}\,\,\, [0.75, 1]\,.
\end{align*}
\scriptsize
\begin{table}[h]
\begin{tabular}[t]{ | c | c | c | c | c | c | c | c | c | }\hline
 & \multicolumn{4}{| c |}{$L^{2}-errors$} & \multicolumn{4}{| c |}{$L^{\infty}-errors$} \\ \hline
$N$   &    $\eta$      &  $order$  &      $u$       &  $order$ &
$\eta$     &  $order$  &      $u$       & $order$    \\ \hline
$40$  & $0.7403(-2)$ & $      $ & $0.2625(-3)$ & $      $ & $0.1752(-1)$ & $      $ & $0.5473(-3)$ & $      $  \\ \hline
$80$  & $0.1850(-2)$ & $2.0006$ & $0.6539(-4)$ & $2.0054$ & $0.4348(-2)$ & $2.0108$ & $0.1470(-3)$ & $1.8964$  \\ \hline
$120$ & $0.8222(-3)$ & $2.0000$ & $0.2899(-4)$ & $2.0061$ & $0.1947(-2)$ & $1.9807$ & $0.6151(-4)$ & $2.1490$  \\ \hline
$160$ & $0.4625(-3)$ & $2.0001$ & $0.1626(-4)$ & $2.0092$ & $0.1096(-2)$ & $1.9981$ & $0.3464(-4)$ & $1.9963$  \\ \hline
$200$ & $0.2960(-3)$ & $2.0000$ & $0.1040(-4)$ & $2.0033$ & $0.7061(-3)$ & $1.9707$ & $0.2349(-4)$ & $1.7393$  \\ \hline
$240$ & $0.2056(-3)$ & $2.0000$ & $0.7227(-5)$ & $1.9967$ & $0.4901(-3)$ & $2.0021$ & $0.1581(-4)$ & $2.1729$  \\ \hline
$280$ & $0.1510(-3)$ & $1.9998$ & $0.5310(-5)$ & $1.9996$ & $0.3615(-3)$ & $1.9752$ & $0.1165(-4)$ & $1.9787$  \\ \hline
$320$ & $0.1156(-3)$ & $2.0002$ & $0.4068(-5)$ & $1.9965$ & $0.2776(-3)$ & $1.9781$ & $0.8902(-5)$ & $2.0166$  \\ \hline
$360$ & $0.9136(-4)$ & $2.0001$ & $0.3214(-5)$ & $2.0009$ & $0.2191(-3)$ & $2.0072$ & $0.6945(-5)$ & $2.1086$  \\ \hline
$400$ & $0.7400(-4)$ & $2.0002$ & $0.2603(-5)$ & $2.0007$ & $0.1779(-3)$ & $1.9773$ & $0.5641(-5)$ & $1.9727$  \\ \hline
$440$ & $0.6116(-4)$ & $1.9999$ & $0.2151(-5)$ & $1.9982$ & $0.1471(-3)$ & $1.9935$ & $0.4789(-5)$ & $1.7181$  \\ \hline
$480$ & $0.5139(-4)$ & $2.0000$ & $0.1808(-5)$ & $2.0003$ & $0.1242(-3)$ & $1.9466$ & $0.3971(-5)$ & $2.1518$  \\ \hline
$520$ & $0.4379(-4)$ & $2.0001$ & $0.1541(-5)$ & $1.9967$ & $0.1054(-3)$ & $2.0517$ & $0.3343(-5)$ & $2.1517$  \\ \hline
\end{tabular}
\normalsize
\caption{
Errors and orders of convergence. \eqref{eqsw} system, standard Galerkin spatial discretization
with piecewise linear, continuous elements on a piecewise uniform mesh, $t=1$.}
\label{tbl23}
\end{table}\\
\normalsize
The $L^{2}$-errors are practically equal to 2 while the $L^{\infty}$-errors are very close to $2$. Table \ref{tbl24} shows that
the analogous errors are again practically equal to $2$ when the same example was solved on a `slowly varying' mesh defined by the
meshlengths
\begin{align*}
& h_{i} = \wh{h}_{1}=\frac{1}{4N_{1}}\,, \qquad 1\leq i\leq N_{1}\,, \quad N_{1}=\frac{3N}{7}\,, \quad \text{on}\,\,\, [0,0.25]\,,\\[3pt]
& h_{i+N_{1}}=\wh{h}_{1} + (\wh{h}_{3}-\wh{h}_{1})\frac{i-1}{N_{2}-1}\,,
\quad 1\leq i\leq N_{2}\,, \quad N_{2}=\frac{3N}{7}\,, \quad \text{on}\,\,\, [0.25, 0.75]\,,\\[3pt]
& h_{i+N_{1}+N_{2}} = \wh{h}_{3}= \frac{1}{4N_{3}}\,, \quad 1\leq i\leq N_{3}\,,
\quad N_{3}=\frac{N}{7}\,,\quad  \text{on}\,\,\, [0.75, 1]\,.
\end{align*}
(Thus, the meshlength in the middle interval $[0,25, 0.75]$ varies linearly
\scriptsize
\begin{table}[h]
\begin{tabular}[t]{ | c | c | c | c | c | c | c | c | c | }\hline
 & \multicolumn{4}{| c |}{$L^{2}-errors$} & \multicolumn{4}{| c |}{$L^{\infty}-errors$} \\ \hline
$N$   &    $\eta$      &  $order$  &      $u$       &  $order$ &
$\eta$     &  $order$  &      $u$       & $order$    \\ \hline
$70$  & $0.1957(-2)$ & $      $ & $0.1172(-3)$ & $      $ & $0.5554(-2)$ & $      $ & $0.3039(-3)$ & $      $  \\ \hline
$140$ & $0.4979(-3)$ & $1.9744$ & $0.2922(-4)$ & $2.0036$ & $0.1414(-2)$ & $1.9739$ & $0.7555(-4)$ & $2.0082$  \\ \hline
$210$ & $0.2225(-3)$ & $1.9861$ & $0.1297(-4)$ & $2.0033$ & $0.6313(-3)$ & $1.9883$ & $0.3333(-4)$ & $2.0180$  \\ \hline
$280$ & $0.1255(-3)$ & $1.9903$ & $0.7289(-5)$ & $2.0028$ & $0.3560(-3)$ & $1.9916$ & $0.1878(-4)$ & $1.9946$  \\ \hline
$350$ & $0.8047(-4)$ & $1.9926$ & $0.4663(-5)$ & $2.0019$ & $0.2283(-3)$ & $1.9904$ & $0.1201(-4)$ & $2.0027$  \\ \hline
$420$ & $0.5594(-4)$ & $1.9940$ & $0.3237(-5)$ & $2.0014$ & $0.1586(-3)$ & $1.9979$ & $0.8330(-5)$ & $2.0073$  \\ \hline
$490$ & $0.4113(-4)$ & $1.9949$ & $0.2378(-5)$ & $2.0015$ & $0.1166(-3)$ & $1.9954$ & $0.6119(-5)$ & $2.0007$  \\ \hline
$560$ & $0.3151(-4)$ & $1.9956$ & $0.1820(-5)$ & $2.0011$ & $0.8934(-4)$ & $1.9952$ & $0.4684(-5)$ & $2.0012$  \\ \hline
$630$ & $0.2491(-4)$ & $1.9962$ & $0.1438(-5)$ & $2.0009$ & $0.7062(-4)$ & $1.9968$ & $0.3700(-5)$ & $2.0021$  \\ \hline
$700$ & $0.2018(-4)$ & $1.9966$ & $0.1165(-5)$ & $2.0010$ & $0.5722(-4)$ & $1.9967$ & $0.2996(-5)$ & $2.0028$  \\ \hline
$770$ & $0.1669(-4)$ & $1.9969$ & $0.9625(-6)$ & $2.0008$ & $0.4731(-4)$ & $1.9962$ & $0.2477(-5)$ & $1.9977$  \\ \hline
$840$ & $0.1402(-4)$ & $1.9972$ & $0.8087(-6)$ & $2.0007$ & $0.3976(-4)$ & $1.9978$ & $0.2081(-5)$ & $1.9991$  \\ \hline
\end{tabular}
\normalsize
\caption{
Errors and orders of convergence. \eqref{eqsw} system, standard Galerkin spatial discretization
with piecewise linear, continuous elements on a slowly varying mesh, $t=1$.}
\label{tbl24}
\end{table}
\normalsize
between the values of $\wh{h}_{1}$ and $\wh{h}_{3}$ of the meshlengths of the uniform meshes on $[0, 0.25]$ and $[0.75, 1]$,
respectively.) Finally, Table \ref{tbl25} shows again that the errors are of optimal order in the case of a quasiuniform mesh obtained by
perturbing a uniform mesh with meshlength $h$ by $O(h^{2})$ quantities as follows:
\begin{align*}
& h_{4i-3} = h-0.25 h^{2}\,, \quad h_{4i-2} = h + 0.5 h^{2}\,, \quad h_{4i-2} = h - 0.5 h^{2}\,,\\[3pt]
& h_{4i}=h + 0.25 h^{2}\,, \quad  1\leq i\leq N/4\,, \quad h=1/N\,.
\end{align*}
\scriptsize
\begin{table}[h]
\begin{tabular}[t]{ | c | c | c | c | c | c | c | c | c | }\hline
 & \multicolumn{4}{| c |}{$L^{2}-errors$} & \multicolumn{4}{| c |}{$L^{\infty}-errors$} \\ \hline
$N$   &    $\eta$    &  $order$  &      $u$       &  $order$ &
$\eta$               &  $order$ &      $u$       & $order$    \\ \hline
$40$  & $0.1864(-2)$ & $      $ & $0.2509(-3)$ & $      $ & $0.6354(-2)$ & $      $ & $0.5334(-3)$ & $      $  \\ \hline
$80$  & $0.4654(-3)$ & $2.0021$ & $0.6398(-4)$ & $1.9711$ & $0.1632(-2)$ & $1.9615$ & $0.1350(-3)$ & $1.9823$  \\ \hline
$120$ & $0.2069(-3)$ & $1.9996$ & $0.2856(-4)$ & $1.9890$ & $0.7233(-3)$ & $2.0061$ & $0.6022(-4)$ & $1.9908$  \\ \hline
$160$ & $0.1164(-3)$ & $1.9996$ & $0.1611(-4)$ & $1.9913$ & $0.4052(-3)$ & $2.0147$ & $0.3413(-4)$ & $1.9735$  \\ \hline
$200$ & $0.7449(-4)$ & $1.9995$ & $0.1032(-4)$ & $1.9932$ & $0.2595(-3)$ & $1.9972$ & $0.2203(-4)$ & $1.9611$  \\ \hline
$240$ & $0.5173(-4)$ & $1.9996$ & $0.7175(-5)$ & $1.9958$ & $0.1806(-3)$ & $1.9873$ & $0.1539(-4)$ & $1.9690$  \\ \hline
$280$ & $0.3801(-4)$ & $1.9997$ & $0.5273(-5)$ & $1.9974$ & $0.1327(-3)$ & $2.0008$ & $0.1137(-4)$ & $1.9654$  \\ \hline
$320$ & $0.2910(-4)$ & $1.9997$ & $0.4039(-5)$ & $1.9980$ & $0.1015(-3)$ & $2.0094$ & $0.8717(-5)$ & $1.9871$  \\ \hline
$360$ & $0.2300(-4)$ & $1.9998$ & $0.3191(-5)$ & $1.9986$ & $0.8020(-4)$ & $1.9953$ & $0.6884(-5)$ & $2.0050$  \\ \hline
$400$ & $0.1863(-4)$ & $1.9998$ & $0.2585(-5)$ & $1.9994$ & $0.6502(-4)$ & $1.9918$ & $0.5576(-5)$ & $2.0002$  \\ \hline
$440$ & $0.1539(-4)$ & $1.9999$ & $0.2137(-5)$ & $1.9998$ & $0.5372(-4)$ & $2.0029$ & $0.4602(-5)$ & $2.0137$  \\ \hline
$480$ & $0.1294(-4)$ & $1.9999$ & $0.1795(-5)$ & $1.9999$ & $0.4511(-4)$ & $2.0065$ & $0.3863(-5)$ & $2.0112$  \\ \hline
$520$ & $0.1102(-4)$ & $1.9999$ & $0.1530(-5)$ & $2.0000$ & $0.3846(-4)$ & $1.9947$ & $0.3289(-5)$ & $2.0087$  \\ \hline
$560$ & $0.9504(-5)$ & $1.9999$ & $0.1319(-5)$ & $2.0001$ & $0.3317(-4)$ & $1.9942$ & $0.2838(-5)$ & $1.9907$  \\ \hline
$600$ & $0.8279(-5)$ & $2.0000$ & $0.1149(-5)$ & $2.0001$ & $0.2889(-4)$ & $2.0047$ & $0.2473(-5)$ & $1.9988$  \\ \hline
$640$ & $0.7276(-5)$ & $2.0000$ & $0.1010(-5)$ & $1.9999$ & $0.2538(-4)$ & $2.0056$ & $0.2175(-5)$ & $1.9903$  \\ \hline
\end{tabular}
\normalsize
\caption{
Errors and orders of convergence. \eqref{eqsw} system, standard Galerkin spatial discretization
with piecewise linear, continuous elements on a perturbed uniform mesh, $t=1$.}
\label{tbl25}
\end{table}
\normalsize
\section{Some superaccuracy properties of the $L^{2}$ projection on spaces of continuous, piecewise linear functions}
In this section we will prove in a series of Lemmas some {\it{superaccuracy (superconvergence)}} properties of the
$L^{2}$ projection of smooth functions that satisfy suitable boundary conditions onto spaces of piecewise linear,
continuous functions defined on a {\it{uniform mesh}} in $[0,1]$. These properties will be used in Section 4 to establish
optimal-order $L^{2}$ error estimates for the semidiscrete approximations of \eqref{eqsw} and \eqref{eqssw} in these
finite element spaces.\par
For the purposes of this section (and of \S 4) for integer $N\geq 2$ we let $h=1/N$, $x_{i}=(i-1)h$, $i=1,\dots,N+1$, be
a uniform partition of $[0,1]$ and $I_{i}=x_{i+1}-x_{i}$, $1\leq i\leq N$. We put $x_{i+1/2}=(x_{i}+x_{i+1})/2$.
We also let $S_{h}=S_{h}^{0,2}:=\{\phi \in C^{0} : \phi\big|_{[x_{j},x_{j+1}]} \in \mathbb{P}_{1}\,, 1\leq j\leq N\}$ and
$S_{h,0}=S_{h,0}^{0,2}=\{\phi \in S_{h} : \phi(0)=\phi(1)=0\}$. We equip $S_{h}$ with the basis $\{\phi_{i}\}_{i=1}^{N+1}$,
where $\phi_{i}\in S_{h}$ and $\phi_{i}(x_{j})=\delta_{ij}$, $1\leq i,j,\leq N+1$, and $S_{h,0}$ with the basis
$\{\chi_{i}\}_{i=1}^{N-1}$, where $\chi_{i}=\phi_{i+1}$, $1\leq i\leq N-1$. We let again $P$, $P_{0}$ be the
$L^{2}$ projection operators onto $S_{h}$, $S_{h,0}$, respectively.
\begin{lemma} Let $u\in C_{0}^{4}$, $u''(0)=u''(1)=0$, and $\sigma=u-P_{0}u$. Then, there exists a constant
$C=C(\|u^{(4)}\|_{\infty})$ such that
\begin{equation}
\max_{1\leq i\leq N}\biggl|\int_{I_{i}}\sigma dx\biggr|\leq Ch^{5}.
\label{eq31}
\end{equation}
\end{lemma}
\begin{proof}
By the definition of $P_{0}$ we have for $x\in I_{i}$, $2\leq i\leq N-1$,
$\sigma=u-P_{0}u=u-(d_{i-1}\chi_{i-1} + d_{i}\chi_{i})$, giving
\begin{equation}
\int_{I_{i}}\sigma dx = \int_{I_{i}}u dx -\tfrac{h}{2}(d_{i-1} + d_{i}), \quad 2\leq i\leq N-1.
\label{eq32}
\end{equation}
Similarly
\begin{equation}
\int_{I_{i}}\sigma dx = \int_{I_{i}}u dx - \tfrac{h}{2}d_{1}, \quad
\int_{I_{N}}\sigma dx = \int_{I_{N}}u dx - \tfrac{h}{2}d_{N-1}.
\label{eq33}
\end{equation}
Here we have denoted by $d=(d_{1},d_{2},\dots,d_{N-1})^{T}$ the coefficients of $P_{0}u$ with respect to the basis
$\{\chi_{i}\}_{i=1}^{N-1}$, i.e. the solution of the linear system $G^{0}d=b$, where $G_{ij}^{0}=(\chi_{j},\chi_{i})$,
$1\leq i,j\leq N-1$, and $b_{i}=(u,\chi_{i})$, $1\leq i\leq N-1$. The equations of this system may be written
explicitly as
\begin{equation}
\begin{aligned}
4d_{1} + d_{2} & = 6b_{1}/h,\\
d_{i-1} + 4d_{i} + d_{i+1}&=6b_{i}/h, \quad 2\leq i\leq N-2,\\
d_{N-2} + 4d_{N-1} & = 6b_{N-1}/h.
\end{aligned}
\label{eq34}
\end{equation}
From \eqref{eq34} it is straightforward to infer that the last terms in the right-hand sides of \eqref{eq32} and \eqref{eq33}
satisfy the equations
\begin{equation}
\begin{aligned}
3\cdot\tfrac{h}{2}d_{1} + \tfrac{h}{2}(d_{1} + d_{2}) & = 3b_{1},\\
\tfrac{h}{2}d_{1} + 4\cdot\tfrac{h}{2}(d_{1} + d_{2}) + \tfrac{h}{2}(d_{2} + d_{3})& = 3(b_{1} + b_{2}),\\
\tfrac{h}{2}(d_{i-1} + d_{i}) + 4\cdot\tfrac{h}{2}(d_{i} + d_{i+1}) + \tfrac{h}{2}(d_{i+1} + d_{i+2})&=3(b_{i}+b_{i+1}),
\quad 2\leq i\leq N-3,\\
\tfrac{h}{2}(d_{N-3} + d_{N-2}) + 4\cdot\tfrac{h}{2}(d_{N-2} + d_{N-1}) + \tfrac{h}{2}d_{N-1}& = 3(b_{N-2}+b_{N-1}),\\
\tfrac{h}{2}(d_{N-2} + d_{N-1}) + 3\cdot\tfrac{h}{2}d_{N-1} & = 3b_{N-1}.
\end{aligned}
\label{eq35}
\end{equation}
Hence, by \eqref{eq32}, \eqref{eq33}, and the above equations, if $\ve_{i}=\int_{I_{i}}\sigma dx$, $1\leq i\leq N$,
we see that $\ve=(\ve_{1},\dots,\ve_{N})^{T}$ is the solution of the linear system $\Gamma\ve=r$, where $\Gamma$ is the
$N\times N$ tridiagonal matrix with elements $\Gamma_{11}=\Gamma_{NN}=3$, $\Gamma_{ii}=4$, $2\leq i\leq N-1$, and
$\Gamma_{ij}=1$ if $\abs{i-j}=1$, and $r=(r_{1}\dots,r_{N})^{T}$ is given by
\begin{equation}
\begin{aligned}
r_{1}& = 3\int_{I_{1}}u dx + \int_{I_{2}}u dx - 3b_{1},\\
r_{i}& = \int_{I_{i-1}}u dx + 4\int_{I_{i}}u dx + \int_{I_{i+1}}u dx - 3(b_{i-1} + b_{i}), \quad 2\leq i\leq N-1,\\
r_{N}& = \int_{I_{N-1}}u dx + 3\int_{I_{N}}u dx - 3b_{N-1}.
\end{aligned}
\label{eq36}
\end{equation}
We will show that $r_{i}=O(h^{5})$, $1\leq i\leq N$. For $r_{1}$ we have by \eqref{eq36}, \eqref{eq35}, \eqref{eq32},
and \eqref{eq33}, that
\begin{align*}
r_{1} & = 3\int_{I_{1}}u dx + \int_{I_{2}}u dx - \tfrac{3}{h}\int_{I_{1}}xu dx - \tfrac{3}{h}\int_{I_{2}}(2h-x)u dx\\
& = \tfrac{3}{h}J_{1} + \tfrac{1}{h}J_{2},
\end{align*}
where
\[
J_{1} = \int_{I_{1}}(h-x)u dx, \quad J_{2}=\int_{I_{2}}(3x-5h)u dx.
\]
From Taylor's theorem, using our hypotheses on $u$, we obtain
\begin{align*}
J_{1} & = \tfrac{h^{3}}{6}u'(0) + \tfrac{h^{5}}{120}u'''(0) + O(h^{6}),\\
J_{2} & = -\tfrac{h^{3}}{2}u'(0) - \tfrac{h^{5}}{40}u'''(0) + O(h^{6}),
\end{align*}
which give that $r_{1}=O(h^{5})$. For $r_{i}$, $2\leq i\leq N-2$, we have by \eqref{eq36}, \eqref{eq35}, \eqref{eq32}, and
\eqref{eq33} that
\[
r_{i} = \int_{I_{i-1}\cup I_{i}\cup I_{i+1}}u dx - 3\int_{I_{i-1}}\frac{x-x_{i-1}}{h}u dx
-3\int_{I_{i+1}}\frac{x_{i+2}-x}{h}u dx,\quad 2\leq i\leq N-1.
\]
Since $u\in C^{4}$, it follows from Simpson's rule and Taylor's theorem, as in the first part of the proof of Lemma 5.7
of \cite{adarxiv}, that $r_{i}=O(h^{5})$, $2\leq i\leq N-1$. Finally, since
\[
r_{N} = 3\int_{I_{1}}v dx + \int_{I_{2}}v dx - 3(v,\chi_{1}),
\]
where $v(x):=u(1-x)$, we see that $r_{N}$ is the same as $r_{1}$ with $v$ replacing $u$. It follows that $r_{N}=O(h^{5})$.
Note that $\frac{1}{4}\Gamma\ve = \frac{1}{4}r$, where $\frac{1}{4}\Gamma=I-E$, and $E$ is a $N\times N$ matrix with
$\|E\|_{\infty}=1/2$. Hence, $\|(I-E)^{-1}\|_{\infty}\leq 2$, and thus
\[
\max_{1\leq i\leq N}\abs{\ve_{i}} \leq \tfrac{1}{2}\max_{1\leq i\leq N}\abs{r_{i}} \leq Ch^{5}.
\]
\end{proof}
\begin{lemma} Let $u\in C_{0}^{3}$, $u''(0)=u''(1)=0$, and $\sigma=u-P_{0}u$. Then, there exists a constant
$C=C(\|u^{(3)}\|_{\infty})$ such that
\begin{equation}
\max_{1\leq i\leq N}\bigl|\sigma'(x_{i+1/2})\bigr| \leq Ch^{2}.
\label{eq37}
\end{equation}
\end{lemma}
\begin{proof}
Let $d=(d_{1},\dots,d_{N-1})^{T}$ be defined as in the proof of Lemma 3.1. Then, for $x\in I_{1}$ we have
\[
\sigma'(x)=u'(x) - (P_{0}u)'(x)=u'(x) - \frac{d_{1}}{h}.
\]
In addition
\begin{align*}
\sigma'(x) & = u'(x) - \frac{d_{i}-d_{i-1}}{h}, \quad x\in I_{i}, \quad 2\leq i\leq N-1,\\
\sigma'(x) & = u'(x) - \frac{-d_{N-1}}{h}, \quad x\in I_{N}.
\end{align*}
Now, it follows from the equations \eqref{eq34} that
\begin{align*}
5d_{1} + (d_{2}-d_{1}) & = 6b_{1}/h,\\
d_{1} + 4(d_{2}-d_{1}) + (d_{3}-d_{2}) & = 6(b_{2}-b_{1})/h,\\
(d_{i-1}-d_{i-2}) + 4(d_{i}-d_{i-1}) + (d_{i+1}-d_{i}) & = 6(b_{i}-b_{i-1})/h,\quad 3\leq i\leq N-2,\\
(d_{N-2} - d_{N-3}) + 4(d_{N-1} - d_{N-2}) + (-d_{N-1}) & = 6(b_{N-1}-b_{N-2})/h,\\
(d_{N-1} - d_{N-2}) + 5(-d_{N-1}) & = -6b_{N-1}/h.
\end{align*}
Hence, if $\ve_{i}':=\sigma'(x_{i+1/2})$, $1\leq i\leq N$, then the vector $\ve'=(\ve_{1}',\dots,\ve_{N}')^{T}$
is the solution of the system $A\ve'=r'$, where $A$ is the $N\times N$ tridiagonal matrix with elements
$A_{11}=A_{NN}=5$, $A_{ii}=4$, $2\leq i\leq N-1$, and $A_{ij}=1$ if $\abs{i-j}=1$, and $r'=(r_{1}',\dots,r_{N}')^{T}$
is given by
\begin{align*}
r_{1}' & = 5u'(x_{1+1/2}) + u'(x_{2+1/2}) - 6b_{1}/h^{2},\\
r_{i}' & = u'(x_{i-1/2}) + 4u'(x_{i+1/2}) + u'(x_{i+3/2}) - 6(b_{i}-b_{i-1})/h^{2},\quad 2\leq i\leq N-1,\\
r_{N}' & = u'(x_{N-1/2}) + 5u'(x_{N+1/2}) + 6b_{N-1}/h^{2}.
\end{align*}
We will show that $r_{i}'=O(h^{2})$, $1\leq i\leq N$. By Taylor's theorem and our assumptions on $u$ we first have
\[
b_{1}=(u,\chi_{1})=\tfrac{1}{4}\int_{I_{1}}xu dx + \tfrac{1}{h}\int_{I_{2}}(2h-x)u dx = h^{2}u'(0) + O(h^{4}).
\]
Therefore,
\[
r_{1}'=5u'(0) + u'(0) - 6u'(0) + O(h^{2}) = O(h^{2}).
\]
For $r_{i}'$ , $2\leq i\leq N-1$, we have, since $\phi_{i}=\chi_{i-1}$,
\[
r_{i}' =u'(x_{i-1/2}) + 4u'(x_{i+1/2}) + u'(x_{i+3/2}) -
\tfrac{6}{h^{2}}\bigl((u,\phi_{i+1}) - (u,\phi_{i})\bigr).
\]
It then follows from the relations $(5.13)$-$(5.17)$ et seq. in the proof of Lemma 5.5 of \cite{adarxiv} that
$r_{i}' = O(h^{2})$, $2\leq i\leq N-1$. Finally, since
\[
r_{N}' = -\bigl[v'(x_{2+1/2}) + 5v'(x_{1+1/2}) - 6(v,\chi_{1})/h^{2}\bigr],
\]
where we have denoted $v(x):=u(1-x)$, we see that $r_{N}'$ is given by $-r_{1}'$ with $u$ replaced by $v$. It follows
that $r_{N}' = O(h^{2})$ as well. Obviously $r_{i}'=O(h^{2})$, $1\leq i\leq N$, implies that $\ve_{i}'=O(h^{2})$
in view of the properties of the matrix $A$.
\end{proof}
\begin{lemma} Suppose that $v\in C^{2}$, $u\in C_{0}^{4}$, $u''(0)=u''(1)=0$, and $\sigma=u-P_{0}u$. Then there
exists a constant $C$ independent of $h$ such that
\begin{equation}
\max_{1\leq i\leq N}\biggl|\int_{I_{i}}v\sigma dx\biggr| \leq Ch^{5}.
\label{eq38}
\end{equation}
\end{lemma}
\begin{proof}
Since $\|\sigma\|_{\infty}=O(h^{2})$ by \eqref{eq22b}, a Taylor expansion of $v$ gives
\[
\int_{I_{i}}v\sigma dx=v(x_{i+1/2})\int_{I_{i}}\sigma dx + v'(x_{i+1/2})\int_{I_{i}}(x-x_{i+1/2})\sigma dx + O(h^{5}).
\]
For the second integral in the right-hand side of this relation a Taylor expansion of $\sigma$ and the fact that
$P_{0}u\bigl|_{I_{i}}\in \mathbb{P}_{1}$ yield
\[
\int_{I_{i}}(x-x_{i+1/2})\sigma dx = \tfrac{h^{3}}{12}\sigma'(x_{i+1/2}) + O(h^{5}).
\]
The estimate \eqref{eq38} now follows from \eqref{eq31} and \eqref{eq37}.
\end{proof}
\begin{lemma} Let $\eta\in C^{4}$ with $\eta'(0)=\eta'(1)=0$ and $\eta'''(0)=\eta'''(1)=0$. If $\rho=\eta-P\eta$,
then there exists a constant $C=C(\|\eta^{(4)}\|_{L^{\infty}})$ such that
\begin{equation}
\max_{1\leq i\leq N}\biggl|\int_{I_{i}}\rho dx\biggr| \leq Ch^{5}.
\label{eq39}
\end{equation}
\end{lemma}
\begin{proof}
Let $\ve=(\ve_{1},\dots,\ve_{N})^{T}$, where $\ve_{i}:=\int_{I_{i}}\rho dx$, and $r=(r_{1},\dots,r_{N})^{T}$, where
\begin{align*}
r_{1} & = 5\int_{I_{1}}\eta dx + \int_{I_{2}}\eta dx - 3(2b_{1}+b_{2}),\\
r_{i} & = \int_{I_{i-1}}\eta dx + 4\int_{I_{i}}\eta dx + \int_{I_{i+1}}\eta dx - 3(b_{i} + b_{i+1}), \quad 2\leq i\leq N-1,\\
r_{N} & = \int_{I_{N-1}}\eta dx + 5\int_{I_{N}}\eta dx - 3(b_{N} + 2b_{N+1}),
\end{align*}
and $b_{i}=(\eta,\phi_{i})$, $1\leq i\leq N+1$. Then  $\Gamma\ve=r$, where $\Gamma$ is the $N\times N$ tridiagonal matrix
with $\Gamma_{11}=\Gamma_{NN}=5$, $\Gamma_{ii}=4$, $2\leq i \leq N-1$, and $\Gamma_{ij}=1$ if $\abs{i-j}=1$.
In Lemma 5.7 of \cite{adarxiv} it was proved that $r_{i}=O(h^{5})$, $2\leq i\leq N-1$, under the sole assumption that
$\eta\in C^{4}$. Here we will show, as a result of the boundary conditions imposed on $\eta$ in our hypotheses, that
we also have $r_{1}=O(h^{5})$ and $r_{N}=O(h^{5})$. Since
\[
2b_{1} + b_{2} = \tfrac{1}{h}\int_{I_{1}}(2h-x)\eta dx + \tfrac{1}{h}\int_{I_{2}}(2h-x)\eta dx,
\]
if follows that
\[
r_{1} = \tfrac{1}{h}J_{1} + \tfrac{1}{h}J_{2},
\]
where
\[
J_{1}=\int_{I_{1}}(3x - h)\eta dx, \quad J_{2}=\int_{I_{2}}(3x-5h)\eta dx.
\]
Taylor's expansion of $\eta$ and our hypotheses on its boundary conditions yield
\begin{align*}
J_{1} & = \tfrac{h^{2}}{2}\eta(0) + \tfrac{5h^{4}}{24}\eta''(0) + O(h^{6}),\\
J_{2} & = -\tfrac{h^{2}}{2}\eta(0) - \tfrac{5h^{4}}{24}\eta''(0) + O(h^{6}),
\end{align*}
so that $r_{1}=O(h^{5})$. For $r_{N}$ we note that it is given by the expression for $r_{1}$ with $\eta$ replaced
by $w$, where $w(x):=\eta(1-x)$. Indeed, since $\phi_{N}(x)=\phi_{2}(1-x)$ and $\phi_{N+1}(x)=\phi_{1}(1-x)$,
it follows that $b_{N}=(\eta,\phi_{N-1})=(w,\phi_{2})$ and $b_{N+1}=(\eta,\phi_{N+1})=(w,\phi_{1})$. Therefore
\[
r_{N}=5\int_{I_{1}}w dx + \int_{I_{2}}w dx -3\bigl(2(w,\phi_{1}) + (w,\phi_{2})\bigr),
\]
so that $r_{N}=O(h^{5})$ in view of the estimate for $r_{1}$. The estimate \eqref{eq39} follows now from the
properties of the matrix $\Gamma$.
\end{proof}
\begin{lemma} Let $\eta\in C^{3}$ and $\rho=\eta - P\eta$. Then, there exists a constant $C=C(\|\eta'''\|_{\infty})$
such that
\begin{equation}
\max_{1\leq i\leq N}\bigl|\rho'(x_{i+1/2})\bigr| \leq Ch^{2}.
\label{eq310}
\end{equation}
\end{lemma}
\begin{proof} See Lemma 5.5 of \cite{adarxiv}.
\end{proof}
\begin{lemma} Let $w\in C^{2}$, $\eta\in C^{4}$ with $\eta'(0)=\eta'(1)=0$, $\eta'''(0)=\eta'''(1)=0$. If $\rho=\eta-P\eta$,
there exists a constant $C$ independent of $h$, such that
\begin{equation}
\max_{1\leq i\leq N}\biggl|\int_{I_{i}}w\rho dx\biggr| \leq Ch^{5}.
\label{eq311}
\end{equation}
\end{lemma}
\begin{proof} The proof is similar to that of Lemma 3.3 if \eqref{eq39} and \eqref{eq310} are taken into account.
\end{proof}
\begin{lemma} Consider the mass matrices $G_{ij}=(\phi_{j},\phi_{i})$, $1\leq i,j\leq N+1$, and
$G_{ij}^{0}=(\chi_{j},\chi_{i})$, $1\leq i,j\leq N-1$. \\
(i)\,\,There exist constants $c_{i}$, $1\leq i\leq 4$, independent of $h$, such that
\begin{align*}
c_{1}h\abs{\beta}^{2} & \leq <G\beta,\beta> \leq c_{2}h\abs{\beta}^{2}\quad \forall \beta\in \mathbb{R}^{N+1},\\
c_{3}h\abs{\beta}^{2} & \leq <G^{0}\beta,\beta> \leq c_{4}h\abs{\beta}^{2}\quad \forall \beta\in \mathbb{R}^{N-1}.
\end{align*}
(ii)\,\,Let $b\in\mathbb{R}^{N+1}$, $G\beta=b$, and $\zeta=\sum_{j=1}^{N+1}\beta_{j}\phi_{j}$. Then
\[
\|\zeta\|\leq (c_{1}h)^{-1/2}\abs{b}.
\]\indent
\,\,\,If $b\in\mathbb{R}^{N-1}$, $G^{0}\beta=b$, and $\zeta=\sum_{j=1}^{N-1}\beta_{j}\chi_{j}$, then
\[
\|\zeta\| \leq (c_{3}h)^{-1/2}\abs{b}.
\]
\end{lemma}
\begin{proof}
The proofs of (i) and (ii) are given in Dupont, \cite{d2}, when the elements of the finite element subspace satisfy
periodic boundary conditions. In our case, the proof of (i) follows again from Gerschgorin's Lemma, and (ii) is a
consequence of (i).
\end{proof}
\begin{lemma} Let $w\in C_{0}^{2}$, $v\in C^{2}$, $\eta\in C^{4}$ with $\eta'(0)=\eta'(1)=0$,
$\eta'''(0)=\eta'''(1)=0$, $u\in C_{0}^{4}$ with $u''(0)=u''(1)=0$, $\rho=\eta-P\eta$, $\sigma=u-P_{0}u$.
Then, for constants $C$ independent of $h$: \\
(i)\,\,If $\zeta_{1}\in S_{h,0}$ is defined by $(\zeta_{1},\chi)=(\rho',\chi)$, $\forall \chi\in S_{h,0}$, then
$\|\zeta_{1}\|\leq Ch^{3}$.\\
(ii)\,\, If $\zeta_{2}\in S_{h}$ is defined by $(\zeta_{2},\phi)=(\sigma',\phi)$, $\forall \phi\in S_{h}$, then
$\|\zeta_{2}\|\leq Ch^{3}$.\\
(iii)\,\,If $\zeta_{3}\in S_{h}$ is defined by $(\zeta_{3},\phi)=((w\rho)',\phi)$, $\forall \phi \in S_{h}$, then
$\|\zeta_{3}\|\leq Ch^{3}$.\\
(iv)\,\,If $\zeta_{4}\in S_{h}$ is defined by $(\zeta_{4},\phi)=((v\sigma)',\phi)$, $\forall \phi \in S_{h}$, then
$\|\zeta_{4}\|\leq Ch^{3}$.\\
(v)\,\,If $\zeta_{5}\in S_{h,0}$ is defined by $(\zeta_{5},\chi)=((v\sigma)',\chi)$, $\forall \chi \in S_{h,0}$, then
$\|\zeta_{5}\|\leq Ch^{3}$.\\
(vi)\,\,If $\zeta_{6}\in S_{h,0}$ is defined by $(\zeta_{6},\chi)=((v\rho)',\chi)$, $\forall \chi \in S_{h,0}$, then
$\|\zeta_{6}\|\leq Ch^{3}$.
\end{lemma}
\begin{proof}
(i)\,\,If $b_{i}=(\rho',\chi_{i})$, $1\leq i\leq N-1$, then $b_{i}=-(\rho,\chi_{i}')$, i.e.
$b_{i}=-\frac{1}{h}\int_{I_{i}}\rho dx + \frac{1}{h}\int_{I_{i+1}}\rho dx$, $1\leq i\leq N-1$. By \eqref{eq39},
$\abs{b_{i}}\leq Ch^{4}$. Hence $\abs{b}\leq Ch^{3.5}$ and (i) follows by Lemma 3.7(ii).\\
The proof of (ii) is similar and takes into account \eqref{eq31}.\\
(iii)\,\,If now $b_{i}=((w\rho)',\phi_{i})$, $1\leq i\leq N+1$, then $b_{i}=-(w\rho,\phi_{i}')$, i.e.
$b_{1}=\frac{1}{h}\int_{I_{1}}w\rho dx$, $b_{i}=-\frac{1}{h}\int_{I_{i-1}}w\rho dx + \frac{1}{h}\int_{I_{i}}w\rho dx$,
$2\leq i\leq N$, $b_{N+1}=-\frac{1}{h}\int_{I_{N}}w\rho dx$. By \eqref{eq311}
$\max_{1\leq i\leq N}\abs{b_{i}}\leq Ch^{4}$, so that $\abs{b}\leq Ch^{3.5}$ and (iii) follows from Lemma 3.7(ii).\\
The proofs of (iv) and (v) are similar to that of (iii) if we take into account \eqref{eq38}. Finally, if
$b_{i}=((v\rho)',\chi_{i})$, $1\leq i\leq N-1$, then
$b_{i}=-(v\rho,\chi_{i}')=-\frac{1}{h}\int_{I_{i}}v\rho dx + \frac{1}{h}\int_{I_{i+1}}v\rho dx$, $1\leq i\leq N-1$.
By \eqref{eq311}, $\abs{b_{i}}\leq Ch^{4}$. Hence, $\abs{b}\leq Ch^{3.5}$ and (vi) follows from Lemma 3.7(ii).
\end{proof}
\section{Semidiscretization with continuous, piecewise linear functions on uniform meshes}
In this section we will prove optimal-order $L^{2}$-error estimates for the solutions of the semidiscrete problems
\eqref{eq25}-\eqref{eq26} and \eqref{eq27}-\eqref{eq28} that approximate the ibvp's \eqref{eqsw} and \eqref{eqssw},
respectively, in the spaces $S_{h}=S_{h}^{0,2}$, $S_{h,0}=S_{h,0}^{0,2}$ of piecewise linear continuous functions on
a uniform spatial mesh, using the notation and results of Section 3.\par
The proof of optimality of the order of convergence in the error estimates uses, in addition to the superaccuracy
properties of the $L^{2}$ projection, {\it{compatibility conditions}} at the boundary $\partial I=\{0,1\}$ that smooth
solutions of \eqref{eqsw} and \eqref{eqssw} satisfy. \par
We will assume that the ibvp \eqref{eqsw} has a unique solution $(\eta,u)$ such that $\eta\in C(0,T;C^{4})$,
$u\in C(0,T;C_{0}^{4})$ for some $0<T<\infty$. We will also assume that for some $\alpha>0$,
$\min_{0\leq x\leq 1}(1+\eta_{0}(x))\geq \frac{\alpha}{2}$, so that by the theory of \cite{pt}|,
$\min_{0\leq x\leq 1}(1+\eta(x,t))\geq \alpha>0$, for all $t\in[0,T]$. In addition to the hypothesis
$(\eta_{0},u_{0})\in C^{4}\times C_{0}^{4}$, we assume that $\eta_{0}'\in C_{0}^{3}$, $\eta_{0}'''\in C_{0}^{1}$,
$u_{0}''\in C_{0}^{2}$. Then, from the second p.d.e. of \eqref{eqsw} and the b.c. $u|_{\partial I}=0$, it follows that
$\eta_{x}|_{\partial I}=0$ for $t\in [0,T]$. Differentiating the first p.d.e. with respect to $x$ and using the positivity
of $1+\eta$ we also conclude that $u_{xx}|_{\partial I}=0$ for $t\in [0,T]$. Finally, differentiating the second p.d.e.
twice with respect to $x$ we see that for $0\leq t\leq T$, $\eta_{xxx}|_{\partial I}=0$ as well. We will make the
same hypotheses, leading to the same compatibility conditions for the solution $(\eta,u)$ of \eqref{eqssw}, under
the assumption that $\min_{\substack{0\leq x\leq 1\\0\leq t\leq T}} (1+\frac{1}{2}\eta(x,t))\geq \beta$ for some positive
constant $\beta >0$, which may also be similarly justified, cf. Section 6.2. \par
We begin with the error estimate for the \eqref{eqssw} which is again simpler due to the symmetry of this system.
\begin{theorem} Let $(\eta,u)$ be the solution of \eqref{eqssw} and suppose that $\eta\in C(0,T;C^{4})$,
$u\in C(0,T;C_{0}^{4})$, $\eta_{0}'\in C_{0}^{3}$, $\eta_{0}'''\in C_{0}^{1}$, $u_{0}''\in C_{0}^{2}$ and that
$\min_{\substack{0\leq x\leq 1\\0\leq t\leq T}}(1+\frac{1}{2}\eta(x,t))\geq \beta >0$ for some constant $\beta$.
Let $x_{i}=(i-1)h$, $1\leq i\leq N+1$, $Nh=1$, and $(\eta_{h},u_{h})$ be the solution of \eqref{eq27}-\eqref{eq28} for
$t\in [0,T]$ in the space of piecewise linear continuous functions $S_{h}\times S_{h,0}$. Then
\begin{equation}
\max_{0\leq t\leq T}(\|\eta(t)-\eta_{h}(t)\| + \|u(t)-u_{h}(t)\|) \leq Ch^{2},
\label{eq41}
\end{equation}
and
\begin{equation}
\max_{0\leq t\leq T}(\|\eta(t)-\eta_{h}(t)\|_{\infty} + \|u(t)-u_{h}(t)\|_{\infty}) \leq Ch^{2}.
\tag{$4.1^{'}$}
\label{eq41a}
\end{equation}
\end{theorem}
\begin{proof}
We refer to the analogous proof (Proposition 2.1) in the quasiuniform mesh case for notation. We let again $\theta=P\eta-\eta_{h}$,
$\xi=P_{0}u-u_{h}$, $\rho=\eta-P\eta$, $\sigma=u-P_{0}u$. The identity \eqref{eq213} still holds and we write it,
using integration by parts, in the form
\begin{equation}
\tfrac{1}{2}\tfrac{d}{dt}\|\theta\|^{2} + \bigl([(1+\tfrac{\eta}{2})\xi]_{x},\theta\bigr)=
\tfrac{1}{4}(\xi_{x}\theta,\theta) + A_{1} + A_{2},
\label{eq42}
\end{equation}
where
\begin{equation}
A_{1} := - (\sigma_{x},\theta) -\tfrac{1}{2}((\eta\sigma)_{x},\theta) -\tfrac{1}{2}((u\rho)_{x},\theta),
\label{eq43}
\end{equation}
\begin{equation}
A_{2}:=  - \tfrac{1}{4}(u_{x}\theta,\theta) +\tfrac{1}{4}(\sigma_{x}\theta,\theta)
+ \tfrac{1}{2}((\rho\sigma)_{x},\theta)  +\tfrac{1}{2}((\rho\xi)_{x},\theta).
\label{eq44}
\end{equation}
We will estimate the terms of $A_{1}$ using the superaccuracy properties of Section 3, in view of the compatibility
conditions on $\eta$ and $u$ for $0\leq t\leq T$ implied by our hypotheses as was previously explained.\\
By Lemma 3.8(ii),(iv) with $v=\eta$, and (iii) with $w=u$, we have, respectively,
\begin{align*}
\abs{(\sigma_{x},\theta)} & \leq Ch^{3}\|\theta\|,\\
\abs{((\eta\sigma)_{x},\theta)} & \leq Ch^{3}\|\theta\|,\\
\abs{((u\rho)_{x},\theta)} & \leq Ch^{3}\|\theta\|.
\end{align*}
We conclude by \eqref{eq43} that
\begin{equation}
\abs{A_{1}} \leq Ch^{3}\|\theta\|.
\label{eq45}
\end{equation}
The terms of $A_{2}$ are estimated as in the proof of Proposition 2.1, immediately after \eqref{eq213}, in the case $r=2$.
As a result we have
\begin{equation}
\abs{A_{2}} \leq C(h^{3}\|\theta\| + \|\theta\|^{2} + \|\xi\|^{2}).
\label{eq46}
\end{equation}
Therefore, by \eqref{eq42}, \eqref{eq45}, and \eqref{eq46}, there holds for $t\in [0,T]$ that
\begin{equation}
\tfrac{1}{2}\tfrac{d}{dt}\|\theta\|^{2} + \bigl([(1+\tfrac{\eta}{2})\xi]_{x},\theta\bigr)\leq
\tfrac{1}{4}(\xi_{x}\theta,\theta) + C(h^{3}\|\theta\| + \|\theta\|^{2} + \|\xi\|^{2}).
\label{eq47}
\end{equation}
In addition, the identity \eqref{eq215} still holds. Using integration by parts we write it for $t\in [0,T]$ in the form
\begin{equation}
\tfrac{1}{2}\tfrac{d}{dt}\|\xi\|^{2} - \bigl([(1+\tfrac{\eta}{2})\xi]_{x},\theta\bigr)= - \tfrac{1}{4}(\xi_{x}\theta,\theta)
+ B_{1} + B_{2},
\label{eq48}
\end{equation}
where
\begin{equation}
B_{1}:= -(\rho_{x},\xi) - \tfrac{1}{2}((\eta\rho)_{x},\xi) - \tfrac{3}{2}((u\sigma)_{x},\xi),
\label{eq49}
\end{equation}
\begin{equation}
\begin{aligned}
B_{2}:= & -\tfrac{3}{2}((u\xi)_{x},\xi) + \tfrac{3}{2}((\sigma\xi)_{x},\xi) +\tfrac{3}{2}(\sigma\sigma_{x},\xi)
 - \tfrac{1}{2}(\eta_{x}\theta,\xi) \\
& +\tfrac{1}{2}((\rho\theta)_{x},\xi) + \tfrac{1}{2}(\rho\rho_{x},\xi).
\end{aligned}
\label{eq410}
\end{equation}
Using again the compatibility properties of $\eta$ and $u$ for $0\leq t\leq T$, by Lemma 3.8(i), (vi) with $v=\eta$,
and (v) with $v=u$ we have, respectively,
\begin{align*}
\abs{(\rho_{x},\xi)} & \leq Ch^{3}\|\xi\|,\\
\abs{((\eta\rho)_{x},\xi)} & \leq Ch^{3} \|\xi\|,\\
\abs{((u\sigma)_{x},\xi)} & \leq Ch^{3}\|\xi\|,
\end{align*}
so that by \eqref{eq49}
\begin{equation}
\abs{B_{1}}\leq Ch^{3}\|\xi\|.
\label{eq411}
\end{equation}
The terms of $B_{2}$ are estimated again as in the proof of Proposition 2.1, after \eqref{eq215}, in the case $r=2$.
We have therefore
\begin{equation}
\abs{B_{2}}\leq C(h^{3}\|\xi\| + \|\theta\|^{2} + \|\xi\|^{2}),
\label{eq412}
\end{equation}
and by \eqref{eq48}, \eqref{eq411}, and \eqref{eq412}, for $t\in [0,T]:$
\begin{equation}
\tfrac{1}{2}\tfrac{d}{dt}\|\xi\|^{2} - \bigl([(1+\tfrac{\eta}{2})\xi]_{x},\theta\bigr)= - \tfrac{1}{4}(\xi_{x}\theta,\theta)
+ C(h^{3}\|\xi\| + \|\theta\|^{2} + \|\xi\|^{2}).
\label{eq413}
\end{equation}
Adding \eqref{eq47} and \eqref{eq413} we get for $t\in [0,T]$
\[
\tfrac{d}{dt}(\|\theta\|^{2} + \|\xi\|^{2}) \leq Ch^{3}(\|\xi\| + \|\theta\|) + C(\|\theta\|^{2} + \|\xi\|^{2}).
\]
Therefore, since $\theta(0)=0$, $\xi(0)=0$, Gronwall's lemma gives the superaccurate estimate
\begin{equation}
\|\theta\| + \|\xi\| \leq Ch^{3}, \quad 0\leq t\leq T,
\label{eq414}
\end{equation}
from which \eqref{eq41} follows. In view of \eqref{eq24} and \eqref{eq22b} \eqref{eq414} implies the $L^{\infty}$ estimate
\eqref{eq41a} as well.
\end{proof}
We prove now the analogous optimal-order $L^{2}$ error estimate for the \eqref{eqsw}.
\begin{theorem} Let $(\eta,u)$ be the solution of \eqref{eqsw} and suppose that $\eta\in C(0,T;C^{4})$,
$u\in C(0,T;C_{0}^{4})$, $\eta_{0}'\in C_{0}^{3}$, $u_{0}''\in C_{0}^{2}$, and that
$\min_{\substack{0\leq x\leq 1\\0\leq t\leq T}} (1+\eta(x,t))\geq \alpha>0$ for some positive constant $\alpha$. Let
$x_{i}=(i-1)h$, $1\leq i\leq N+1$, $Nh=1$, and $(\eta_{h},u_{h})$ be the solution of \eqref{eq25}-\eqref{eq26} for
$t\in [0,T]$ in the space of piecewise linear continuous functions $S_{h}\times S_{h,0}$. Then
\begin{equation}
\max_{0\leq t\leq T}(\|\eta(t)-\eta_{h}(t)\| + \|u(t)-u_{h}(t)\|) \leq Ch^{2},
\label{eq415}
\end{equation}
and
\begin{equation}
\max_{0\leq t\leq T}(\|\eta(t)-\eta_{h}(t)\|_{\infty} + \|u(t)-u_{h}(t)\|_{\infty}) \leq Ch^{2}. \tag{$4.15^{'}$}
\label{eq415a}
\end{equation}
\end{theorem}
\begin{proof}
We refer again to the analogous proof (Proposition 2.2) in the quasiuniform case for notation. In particular
we let again $\theta=P\eta-\eta_{h}$, $\xi=P_{0}u-u_{h}$, $\rho=\eta-P\eta$, $\sigma=u-P_{0}u$.
The identity \eqref{eq220} still holds and we write it, using integration by parts, in the form
\begin{equation}
\tfrac{1}{2}\tfrac{d}{dt}\|\theta\|^{2} + \bigl([(1+\eta)\xi]_{x},\theta\bigr)=
\tfrac{1}{2}(\xi_{x}\theta,\theta) + A_{3} + A_{4},
\label{eq416}
\end{equation}
where
\begin{equation}
A_{3} = -(\sigma_{x},\theta) - ((\eta\sigma)_{x},\theta) - ((u\rho)_{x},\theta),
\label{eq417}
\end{equation}
\begin{equation}
A_{4} = -\tfrac{1}{2}(u_{x}\theta,\theta) + \tfrac{1}{2}(\sigma_{x}\theta,\theta) + ((\rho\sigma)_{x},\theta)
+ ((\rho\xi)_{x},\theta).
\label{eq418}
\end{equation}
Using the compatibility conditions on $\eta$ and $u$ implied by our hypotheses, we have, by Lemma 3.8 (ii), (iv)
with $v=\eta$, and (iii) with $w=u$, respectively, that
\begin{align*}
\abs{(\sigma_{x},\theta)} &\leq Ch^{3}\|\theta\|,\\
\abs{((\eta\sigma)_{x},\theta)} & \leq Ch^{3}\|\theta\|,\\
\abs{((u\rho)_{x},\theta)} & \leq Ch^{3}\|\theta\|.
\end{align*}
Hence, by \eqref{eq417}
\begin{equation}
\abs{A_{3}} \leq Ch^{3}\|\theta\|.
\label{eq419}
\end{equation}
The terms of $A_{4}$ are estimated as in the proof of Proposition 2.2 in the various inequalities after \eqref{eq220}
for $r=2$. As a result, we have
\begin{equation}
\abs{A_{4}} \leq C(h^{3}\|\theta\| + \|\theta\|^{2} + \|\xi\|^{2}).
\label{eq420}
\end{equation}
As in Proposition 2.2, we let $t_{h}$ be such that $\|\xi_{x}\|_{\infty}\leq 1$ for $t\leq t_{h}$ and suppose that
$t_{h}<T$. Then we have that $\abs{(\xi_{x}\theta,\theta)}\leq \|\theta\|^{2}$ and \eqref{eq416}, \eqref{eq418} and
\eqref{eq420} imply that for $0\leq t\leq t_{h}$
\begin{equation}
\tfrac{1}{2}\tfrac{d}{dt}\|\theta\|^{2} - (\gamma,\theta_{x}) \leq C(h^{3}\|\theta\| + \|\theta\|^{2} + \|\xi\|^{2}),
\label{eq421}
\end{equation}
where $\gamma=(1+\eta)\xi$.\\
The identity \eqref{eq222} still holds. We write it in the form
\begin{equation}
(\xi_{t},\gamma) + (\theta_{x},P_{0}\gamma)=(\xi\xi_{x},P_{0}\gamma) + B_{3} + B_{4},
\label{eq422}
\end{equation}
where
\begin{equation}
B_{3} = -(\rho_{x},P_{0}\gamma) - ((u\sigma)_{x},P_{0}\gamma),
\label{eq423}
\end{equation}
\begin{equation}
B_{4} = -((u\xi)_{x},P_{0}\gamma) + ((\sigma\xi)_{x},P_{0}\gamma) + (\sigma\sigma_{x},P_{0}\gamma).
\label{eq424}
\end{equation}
By Lemma 3.8(i), and (v) with $v=u$, we have, respectively,
\begin{align*}
\abs{(\rho_{x},P_{0}\gamma)} & \leq Ch^{3}\|P_{0}\gamma\| \leq Ch^{3}\|\gamma\| \leq Ch^{3}\|\xi\|,\\
\abs{((u\sigma)_{x},P_{0}\gamma)} & \leq Ch^{3}\|P_{0}\gamma\| \leq Ch^{3}\|\xi\|.
\end{align*}
Therefore,
\begin{equation}
\abs{B_{3}} \leq Ch^{3}\|\xi\|.
\label{eq425}
\end{equation}
Now, using the superapproximation property \eqref{eq223} we have, as in the proof of Proposition 2.2, that
\begin{align*}
\abs{((u\xi)_{x},P_{0}\gamma)} & \leq C\|\xi\|^{2},\\
\abs{((\sigma\xi)_{x},P_{0}\gamma)} &\leq C\|\xi\|^{2},\\
\abs{(\sigma\sigma_{x},P_{0}\gamma)} &\leq C\|\sigma\|_{\infty}\|\sigma_{x}\|\|\xi\|\leq Ch^{3}\|\xi\|.
\end{align*}
Hence,
\begin{equation}
\abs{B_{4}}\leq C(h^{3}\|\xi\| + \|\xi\|^{2}).
\label{eq426}
\end{equation}
Finally, as in the proof of Proposition 2.2, we have for $0\leq t\leq t_{h}$
\begin{equation}
\abs{(\xi\xi_{x},P_{0}\gamma)} \leq C\|\xi\|^{2}.
\label{eq427}
\end{equation}
From \eqref{eq422}, \eqref{eq425}, \eqref{eq426}, and \eqref{eq427} it follows that for $0\leq t\leq t_{h}$
\begin{equation}
(\xi_{t},\gamma) + (\theta_{x},P_{0}\gamma) \leq C(h^{3}\|\xi\| + \|\xi\|^{2}).
\label{eq428}
\end{equation}
From \eqref{eq421} and \eqref{eq428} we have, as in the proof of Proposition 2.2 that for $0\leq t\leq t_{h}$
\[
\tfrac{d}{dt}\bigl[\|\theta\|^{2} + ((1+\eta)\xi,\xi)\bigr] \leq C(h^{6} + \|\theta\|^{2} + \|\xi\|^{2}),
\]
from which, since $\theta(0)=\xi(0)=0$, we see from Gronwall's lemma that for a constant $C=C(T)$ it holds that
\begin{equation}
\|\theta\| + \|\xi\| \leq Ch^{3}, \quad 0\leq t\leq t_{h}.
\label{eq429}
\end{equation}
Hence $\|\xi_{x}\|_{\infty}\leq Ch^{3/2}$ for $0\leq t\leq t_{h}$ in view of \eqref{eq24}. It follows that $t_{h}$
is not maximal; thus we may take $t_{h}=T$ in \eqref{eq429}. The conclusion of the theorem follows.
\end{proof}
We close this section by presenting the results of some relevant numerical experiments. We solve the nonhomogeneous
\eqref{eqssw} and \eqref{eqsw} using the standard Galerkin method with piecewise linear continuous functions
on a uniform mesh on $[0,1]$ with $h=1/N$ using as exact solutions the functions $\eta=\exp(2t)(\cos(\pi x) + x + 2)$
and $u=\exp(-xt)\sin(\pi x)$. As in Section 2, the fourth-order explicit classical RK method was used for time-stepping
with $k=h/10$. Table \ref{tbl41} shows the $L^{2}$-errors at $t=1$ and their order of convergence for this problem for both
systems. As predicted by the theory of the present section the order of convergence is equal to 2.\par
In Table \ref{tbl42} we present the $L^{2}$ errors for the same problems for the analogous Galerkin method that uses cubic
splines on a uniform mesh for the spatial discretization. The convergence of this scheme was not analyzed here, but
its order of convergence appears to be equal to 4., i.e. optimal in $L^{2}$. (It should be noted that in the \eqref{eqssw}
case we had to take $k/h=1/80$ to achieve stability of the fully discrete scheme.)\vspace{10pt} \\
\scriptsize
\begin{table}[h]
\begin{tabular}[t]{ | c | c | c | c | c | c | c | c | c | }\hline
 & \multicolumn{4}{| c |}{$L^{2}-errors$: SW} & \multicolumn{4}{| c |}{$L^{2}-errors$: SSW} \\ \hline
$N$   &    $\eta$      &  $order$  &      $u$       &  $order$ &
$\eta$     &  $order$  &      $u$       & $order$    \\ \hline
$40$  & $0.4721(-2)$ & $      $ & $0.1859(-3)$ & $      $ & $0.2883(-2)$ & $      $ & $0.1772(-3)$ & $      $  \\ \hline
$80$  & $0.1179(-2)$ & $2.0011$ & $0.4627(-4)$ & $2.0060$ & $0.7203(-3)$ & $2.0011$ & $0.4415(-4)$ & $2.0052$  \\ \hline
$120$ & $0.5241(-3)$ & $2.0004$ & $0.2055(-4)$ & $2.0020$ & $0.3201(-3)$ & $2.0004$ & $0.1963(-4)$ & $1.9992$  \\ \hline
$160$ & $0.2948(-3)$ & $2.0002$ & $0.1155(-4)$ & $2.0012$ & $0.1800(-3)$ & $2.0002$ & $0.1105(-4)$ & $1.9986$  \\ \hline
$200$ & $0.1887(-3)$ & $2.0001$ & $0.7394(-5)$ & $2.0008$ & $0.1152(-3)$ & $2.0001$ & $0.7072(-5)$ & $1.9983$  \\ \hline
$240$ & $0.1310(-3)$ & $2.0001$ & $0.5134(-5)$ & $2.0003$ & $0.8001(-4)$ & $2.0001$ & $0.4911(-5)$ & $1.9997$  \\ \hline
$280$ & $0.9625(-4)$ & $2.0001$ & $0.3772(-5)$ & $2.0001$ & $0.5878(-4)$ & $2.0001$ & $0.3608(-5)$ & $2.0005$  \\ \hline
$320$ & $0.7369(-4)$ & $2.0000$ & $0.2888(-5)$ & $2.0000$ & $0.4501(-4)$ & $2.0000$ & $0.2762(-5)$ & $2.0004$  \\ \hline
$360$ & $0.5822(-4)$ & $2.0000$ & $0.2282(-5)$ & $1.9999$ & $0.3556(-4)$ & $2.0000$ & $0.2182(-5)$ & $2.0003$  \\ \hline
$400$ & $0.4716(-4)$ & $2.0000$ & $0.1848(-5)$ & $1.9998$ & $0.2880(-4)$ & $2.0000$ & $0.1768(-5)$ & $1.9998$  \\ \hline
$440$ & $0.3898(-4)$ & $2.0000$ & $0.1528(-5)$ & $1.9998$ & $0.2381(-4)$ & $2.0000$ & $0.1461(-5)$ & $1.9996$  \\ \hline
$480$ & $0.3275(-4)$ & $2.0000$ & $0.1284(-5)$ & $1.9998$ & $0.2000(-4)$ & $2.0000$ & $0.1228(-5)$ & $1.9998$  \\ \hline
$520$ & $0.2791(-4)$ & $2.0000$ & $0.1094(-5)$ & $1.9998$ & $0.1704(-4)$ & $2.0000$ & $0.1046(-5)$ & $2.0000$  \\ \hline
$560$ & $0.2406(-4)$ & $2.0000$ & $0.9431(-6)$ & $1.9998$ & $0.1470(-4)$ & $2.0000$ & $0.9019(-6)$ & $2.0002$  \\ \hline
$600$ & $0.2096(-4)$ & $2.0000$ & $0.8215(-6)$ & $1.9999$ & $0.1280(-4)$ & $2.0000$ & $0.7857(-6)$ & $2.0002$  \\ \hline
$640$ & $0.1842(-4)$ & $2.0000$ & $0.7221(-6)$ & $1.9999$ & $0.1125(-4)$ & $2.0000$ & $0.6905(-6)$ & $2.0000$  \\ \hline
\end{tabular}
\normalsize
\caption{
$L^{2}$-errors and orders of convergence, continuous, piecewise linear functions, uniform mesh.}
\label{tbl41}
\end{table}
\newpage
\scriptsize
\begin{table}[h]
\begin{tabular}[t]{ | c | c | c | c | c | c | c | c | c | }\hline
 & \multicolumn{4}{| c |}{$L^{2}-errors$: SW} & \multicolumn{4}{| c |}{$L^{2}-errors$: SSW} \\ \hline
$N$   &    $\eta$      &  $order$  &      $u$       &  $order$ &
$\eta$     &  $order$  &      $u$       & $order$    \\ \hline
$40$  & $0.4877(-6)$ & $      $ & $0.2307(-7)$ & $      $ & $0.3287(-6)$ & $      $ & $0.2280(-7)$ & $      $  \\ \hline
$80$  & $0.2938(-7)$ & $4.0530$ & $0.1444(-8)$ & $3.9979$ & $0.1997(-7)$ & $4.0405$ & $0.1412(-8)$ & $4.0134$  \\ \hline
$120$ & $0.5731(-8)$ & $4.0312$ & $0.2848(-9)$ & $4.0037$ & $0.3908(-8)$ & $4.0237$ & $0.2786(-9)$ & $4.0022$  \\ \hline
$160$ & $0.1802(-8)$ & $4.0224$ & $0.9015(-10)$ & $3.9982$ & $0.1230(-8)$ & $4.0172$ & $0.8821(-10)$ & $3.9977$  \\ \hline
$200$ & $0.7351(-9)$ & $4.0175$ & $0.3692(-10)$ & $4.0011$ & $0.5024(-9)$ & $4.0137$ & $0.3623(-10)$ & $3.9882$  \\ \hline
$240$ & $0.3536(-9)$ & $4.0143$ & $0.1781(-10)$ & $3.9977$ & $0.2418(-9)$ & $4.0110$ & $0.1746(-10)$ & $4.0038$  \\ \hline
$280$ & $0.1905(-9)$ & $4.0122$ & $0.9619(-11)$ & $3.9965$ & $0.1303(-9)$ & $4.0091$ & $0.9429(-11)$ & $3.9960$  \\ \hline
$320$ & $0.1115(-9)$ & $4.0104$ & $0.5640(-11)$ & $3.9975$ & $0.7632(-10)$ & $4.0081$ & $0.5523(-11)$ & $4.0057$  \\ \hline
$360$ & $0.6954(-10)$ & $4.0093$ & $0.3523(-11)$ & $3.9949$ & $0.4760(-10)$ & $4.0071$ & $0.3451(-11)$ & $3.9925$  \\ \hline
$400$ & $0.4559(-10)$ & $4.0080$ & $0.2312(-11)$ & $3.9969$ & $0.3121(-10)$ & $4.0068$ & $0.2265(-11)$ & $3.9955$  \\ \hline
$440$ & $0.3112(-10)$ & $4.0072$ & $0.1580(-11)$ & $3.9969$ & $0.2131(-10)$ & $4.0055$ & $0.1549(-11)$ & $3.9875$  \\ \hline
$480$ & $0.2196(-10)$ & $4.0067$ & $0.1116(-11)$ & $3.9946$ & $0.1504(-10)$ & $4.0038$ & $0.1096(-11)$ & $3.9745$  \\ \hline
$520$ & $0.1593(-10)$ & $4.0052$ & $0.8105(-12)$ & $3.9970$ & $0.1092(-10)$ & $3.9992$ & $0.8003(-12)$ & $3.9314$  \\ \hline
$560$ & $0.1184(-10)$ & $4.0062$ & $0.6029(-12)$ & $3.9931$ & $0.8113(-11)$ & $4.0093$ & $0.5997(-12)$ & $3.8922$  \\ \hline
$600$ & $0.8982(-11)$ & $4.0053$ & $0.4576(-12)$ & $3.9946$ & $0.6155(-11)$ & $4.0024$ & $0.4617(-12)$ & $3.7910$  \\ \hline
\end{tabular}
\normalsize
\caption{
$L^{2}$-errors and orders of convergence, cubic splines, uniform mesh.}
\label{tbl42}
\end{table}
\normalsize
\section{Full discretizations with the explicit Euler, the improved Euler, and the third-order Shu-Osher scheme}
In this section we will examine three temporal discretizations of the o.d.e. systems represented by the standard Galerkin
semidiscretizations that we analyzed in Sections 2 and 4. In \cite{d1} Dupont analyzed, in the case of a system similar to that
of the shallow water equations, the convergence of a linearized Crank-Nicolson scheme. In this paper we will analyze three
explicit Runge-Kutta schemes. The explicit Euler method, which will require for stability the restrictive mesh condition $k=O(h^{2})$,
the explicit, second-order accurate `improved' Euler method, which requires that $k=O(h^{4/3})$,
and an explicit third-order Runge-Kutta method due to Shu and Osher, \cite{so}, that needs the condition $k/h\leq \lambda_{0}$
for a small enough constant $\lambda_{0}$.\par
In order to simplify somewhat the proofs we will analyze the convergence of all methods only in the most straightforward to treat case
of the \eqref{eqssw} system and a semidiscretization based on a quasiuniform spatial mesh. Thus, the expected spatial order of convergence
(see Section 2) in $L^{2}$ is of $O(h^{r-1})$. A similar result holds for the \eqref{eqsw} system but we omit the proofs here. Our
proofs require $r\geq 2$ for the Euler method and $r\geq 3$ for the other two fully discrete schemes.
\subsection{The explicit Euler method}
We use the notation of Section 2 for the spatial discretization on a quasiuniform mesh, letting
$S_{h}=S_{h}^{k,r}$, $S_{h,0}=S_{h,0}^{k,r}$, for $r\geq 2$. We let $k=T/M$ be the time step, where $M$ is a positive
integer, and set $t^{n}=nk$ for $n=0,1,2,\dots,M$. The fully discrete approximations $H_{h}^{n}\in S_{h}$, $U_{h}^{n}\in S_{h,0}$
of $\eta(\cdot,t^{n})$, $u(\cdot,t^{n})$, respectively, where $(\eta,u)$ is the solution of the \eqref{eqssw}, are given for
$0\leq n\leq M-1$ by the equations:
\begin{equation}
\begin{aligned}
(H_{h}^{n+1},\phi) - (H_{h}^{n},\phi) +k(U_{hx}^{n},\phi) + \tfrac{k}{2}((H_{h}^{n}U_{h}^{n})_{x},\phi) & = 0,
\quad \forall \phi \in S_{h},\\
(U_{h}^{n+1},\chi) - (U_{h}^{n},\chi) + k(H_{hx}^{n},\chi) + \tfrac{k}{2}(H_{h}^{n}H_{hx}^{n},\chi)
+\tfrac{3k}{2}(U_{h}^{n}U_{hx}^{n},\chi) & = 0,\quad \forall \chi \in S_{h,0},
\end{aligned}
\label{eq51}
\end{equation}
with
\[
H_{h}^{0} = P\eta_{0},\qquad U_{h}^{0}=P_{0}u_{0}.
\]
The equations \eqref{eq51} are written in the form
\begin{equation}
H_{h}^{n+1} - H_{h}^{n} + kPU_{hx}^{n} + \tfrac{k}{2}P\bigl((H_{h}^{n}U_{h}^{n})_{x}\bigr) = 0,
\label{eq52}
\end{equation}
\begin{equation}
U_{h}^{n+1} - U_{h}^{n} + kP_{0}H_{hx}^{n} + \tfrac{k}{2}P_{0}(H_{h}^{n}H_{hx}^{n}) + \tfrac{3k}{2}P_{0}(U_{h}^{n}U_{hx}^{n})=0.
\label{eq53}
\end{equation}
We start with an estimate of the continuous-time truncation error of the $L^{2}$ projections.
\begin{lemma} Suppose that $(\eta,u)$ is the solution of \eqref{eqssw} on $[0,T]$. Let $H(t)=P\eta(t)$,
$U(t) = P_{0}u(t)$ and $\psi(t) \in S_{h}$, $\zeta(t)\in S_{h,0}$ for $0\leq t\leq T$ be such that
\begin{equation}
H_{t} + PU_{x} + \tfrac{1}{2}P\bigl((HU)_{x}\bigr) = \psi,
\label{eq54}
\end{equation}
\begin{equation}
U_{t} + P_{0}H_{x} + \tfrac{1}{2}P_{0}(HH_{x}) + \tfrac{3}{2}P_{0}(UU_{x})=\zeta.
\label{eq55}
\end{equation}
Then
\[
\|\psi\| + \|\psi_{t}\| + \|\zeta\| + \|\zeta_{t}\|\leq Ch^{r-1}.
\]
\end{lemma}
\begin{proof}
Subtracting \eqref{eq54} from the equation $P\bigl(\eta_{t}+u_{x} + \tfrac{1}{2}(\eta u)_{x}\bigr)=0$, and putting
$\rho:=\eta - H$, $\sigma:=u - U$, we obtain
\[
P\bigl([(1+\tfrac{1}{2}\eta)\sigma]_{x} + \tfrac{1}{2}(u\rho)_{x} - \tfrac{1}{2}(\rho\sigma)_{x}\bigr)=-\psi.
\]
Thus, from the approximation properties of $S_{h}$ and $S_{h,0}$ we have
\begin{align*}
\|\psi\| & \leq \|P[(1+\tfrac{1}{2}\eta)\sigma]_{x}\| + \tfrac{1}{2}\|P(u\rho)_{x}\|
+ \tfrac{1}{2}\|P(\rho\sigma)_{x}\|\\
&\leq C(h^{r-1} + h^{r-1} + h^{2r-1})\leq Ch^{r-1}.
\end{align*}
Similarly,
\[
\|\psi_{t}\|\leq \|P[(1+\tfrac{1}{2}\eta)\sigma]_{xt}\| + \tfrac{1}{2}\|P(u\rho)_{xt}\| + \tfrac{1}{2}\|P(\rho\sigma)_{xt}\|
\leq Ch^{r-1}.
\]
Subtracting \eqref{eq55} from the equation $P_{0}(u_{t}+\eta_{x} + \tfrac{1}{2}\eta\eta_{x} + \tfrac{3}{2}uu_{x})=0$
we also obtain
\[
P_{0}\bigl(\rho_{x} + \tfrac{1}{2}(\eta\rho)_{x} - \tfrac{1}{2}\rho\rho_{x} + \tfrac{3}{2}(u\sigma)_{x} -
\tfrac{3}{2}\sigma\sigma_{x}\bigr)=-\zeta.
\]
Therefore,
\begin{align*}
\|\zeta\|&\leq \|P_{0}[(1+\tfrac{1}{2}\eta)\rho]_{x}\| + \tfrac{1}{2}\|P_{0}(\rho\rho_{x})\|
+ \tfrac{3}{2}\|P_{0}(u\sigma)_{x}\| + \tfrac{3}{2}\|P_{0}(\sigma\sigma_{x})\|\\
& \leq C(h^{r-1} + h^{2r-1} + h^{r-1} + h^{2r-1})\leq Ch^{r-1},
\end{align*}
and
\begin{align*}
\|\zeta_{t}\| & \leq \|P_{0}[(1+\tfrac{1}{2}\eta)\rho]_{xt}\| + \tfrac{1}{2}\|P_{0}(\rho\rho_{x})_{t}\|
+\tfrac{3}{2}\|P_{0}(u\sigma)_{xt}\| + \tfrac{3}{2}\|P_{0}(\sigma\sigma_{x})_{t}\|\\
& \leq C(h^{r-1} + h^{2r-1} + h^{r-1} + h^{2r-1}) \leq Ch^{r-1}.
\end{align*}
\end{proof}
We now derive consistency estimates for the scheme \eqref{eq51}.
\begin{lemma} Let $H^{n}:=H(t^{n})=P\eta(t^{n})$, $U^{n}:=U(t^{n})=P_{0}u(t^{n})$ for $n=0,1,\dots,M$, and
$\delta_{1}$, $\delta_{2}$, $n=0,1,2,\dots,M-1$, be defined as
\begin{align*}
\delta_{1}^{n}:=H^{n+1} - H^{n} + kPU_{x} + \tfrac{k}{2}P(H^{n}U^{n})_{x},\\
\delta_{2}^{n}:=U^{n+1} - U^{n} + kP_{0}H_{x}^{n} + \tfrac{k}{2}P_{0}(H^{n}H_{x}^{n}) + \tfrac{3k}{2}P_{0}(U^{n}U_{x}^{n}).
\end{align*}
Then
\[
\max_{0\leq n\leq M-1}(\|\delta_{1}^{n}\| + \|\delta_{2}^{n}\|)\leq Ck(k+h^{r-1}).
\]
\end{lemma}
\begin{proof}
From \eqref{eq54}, \eqref{eq55} it follows
\begin{align*}
\delta_{1}^{n} & = H^{n+1} - H^{n} - kH_{t}^{n} + k\psi^{n},\\
\delta_{2}^{n} & = U^{n+1} - U^{n} - kU_{t}^{n} + k\zeta^{n},
\end{align*}
where $\psi^{n} = \psi(t^{n})$ and $\zeta^{n}=\zeta(t^{n})$. Therefore, for $0\leq n\leq M-1$, we have from
Taylor's theorem and Lemma 5.1
\begin{align*}
\|\delta_{1}^{n}\| + \|\delta_{2}^{n}\| & \leq \|H^{n+1} - H^{n} - kH_{t}^{n}\| + k\|\psi^{n}\|
+ \|U^{n+1} - U^{n} - kU_{t}^{n}\| + k\|\zeta^{n}\|\\
& \leq C(k^{2} + kh^{r-1} + k^{2} + kh^{r-1})\leq Ck(k+h^{r-1}).
\end{align*}
\end{proof}
Our stability and convergence result follows.
\begin{proposition} Let $H_{h}^{n}$, $U_{h}^{n}$ be the solution of \eqref{eq51} and $\mu=k/h^{2}$. Then there
exists a constant $C=C(\mu)$ such that
\[
\max_{0\leq n\leq M}(\|H_{h}^{n} - \eta(t^{n})\| + \|U_{h}^{n} - u(t^{n})\|\leq C(k + h^{r-1}).
\]
\end{proposition}
\begin{proof}
Let $\ve^{n}=H^{n} - H_{h}^{n}$ and $e^{n}=U^{n} - U_{h}^{n}$. From \eqref{eq52}, \eqref{eq53} and the definition of
$\delta_{1}^{n}$, $\delta_{2}^{n}$ in the previous lemma, we have, for $n=0,1,2,\dots,M-1$,
\begin{align*}
\ve^{n+1} & =\ve^{n} - kPe_{x}^{n} + \tfrac{k}{2}P(\ve^{n}e^{n})_{x} - \tfrac{k}{2}P(H^{n}e^{n} + U^{n}\ve^{n})_{x}
+ \delta_{1}^{n},\\
e^{n+1} & = e^{n} - kP_{0}\ve_{x}^{n} + \tfrac{k}{2}P_{0}(\ve^{n}\ve_{x}^{n}) - \tfrac{k}{2}P_{0}(H^{n}\ve^{n})_{x}
+\tfrac{3k}{2}P_{0}(e^{n}e_{x}^{n}) - \tfrac{3k}{2}P_{0}(U^{n}e^{n})_{x} + \delta_{2}^{n}
\end{align*}
Hence,
\begin{equation}
\begin{aligned}
\|\ve^{n+1}\|^{2} & = \|\ve^{n}\|^{2} - 2k(\ve^{n},e_{x}^{n}) + k(\ve^{n},(\ve^{n}e^{n})_{x})
-k(\ve^{n},(H^{n}e^{n} + U^{n}\ve^{n})_{x}) + F^{n},\\
\|e^{n+1}\|^{2} & = \|e^{n}\|^{2} - 2k(e^{n},\ve_{x}^{n}) + k(e^{n},\ve^{n}\ve_{x}^{n})
-k(e^{n},(H^{n}\ve^{n})_{x}) -3k(e^{n},(U^{n}e^{n})_{x}) + G^{n},
\end{aligned}
\label{eq56}
\end{equation}
where
\begin{align*}
F^{n} = & 2(\ve^{n},\delta_{1}^{n}) + k^{2}(Pe_{x}^{n},P[e^{n}-\ve^{n}e^{n} + H^{n}e^{n} + U^{n}\ve^{n}]_{x})
-2k(e_{x}^{n},\delta_{1}^{n})\\
& + \tfrac{k^{2}}{4}(P(\ve^{n}e^{n})_{x},P[\ve^{n}e^{n} - 2(H^{n}e^{n} + U^{n}\ve^{n})]_{x}) + k((\ve^{n}e^{n})_{x},\delta_{1}^{n})\\
& + \tfrac{k^{2}}{4}\|P(H^{n}e^{n} + U^{n}\ve^{n})_{x}\|^{2} - k((H^{n}e^{n} + U^{n}\ve^{n})_{x},\delta_{1}^{n})
+\|\delta_{1}^{n}\|^{2},
\end{align*}
and
\begin{align*}
G^{n} = & 2(e^{n},\delta_{2}^{n}) +
k^{2}(P_{0}\ve_{x}^{n},P_{0}[\ve_{x}^{n} - \ve^{n}\ve_{x}^{n}+ (H^{n}\ve^{n})_{x} - 3e^{n}e_{x}^{n} + 3(U^{n}e^{n})_{x}])
- 2k(\ve_{x}^{n},\delta_{2}^{n})\\
 & + \tfrac{k^{2}}{2}(P_{0}(\ve^{n}\ve_{x}^{n}),P_{0}[\tfrac{1}{2}\ve^{n}\ve_{x}^{n} - (H^{n}\ve^{n})_{x}
 + 3e^{n}e_{x}^{n} - 3(U^{n}e^{n})_{x}]) + k(\ve^{n}\ve_{x}^{n},\delta_{2}^{n})\\
 & + \tfrac{k^{2}}{2}(P_{0}(H^{n}\ve^{n})_{x},P_{0}[\tfrac{1}{2}(H^{n}\ve^{n})_{x} - 3e^{n}e_{x}^{n} + 3(U^{n}e^{n})_{x}])
 -k((H^{n}\ve^{n})_{x},\delta_{2}^{n})\\
& + \tfrac{9k^{2}}{2}(P_{0}(e^{n}e_{x}^{n}),P_{0}[\tfrac{1}{2}e^{n}e_{x}^{n} - (U^{n}e^{n})_{x}]) + 3k(e^{n}e_{x}^{n},\delta_{2}^{n})\\
 & + \tfrac{9k^{2}}{4}\|P_{0}(U^{n}e^{n})_{x}\|^{2} - 3k((U^{n}e^{n})_{x},\delta_{2}^{n}) + \|\delta_{2}^{n}\|^{2}.
\end{align*}
Let $0\leq n^{*}\leq M-1$ be the maximal integer such that
\[
\|\ve^{n}\|_{\infty} + \|e^{n}\|_{\infty} \leq 1, \quad 0\leq n\leq n^{*}.
\]
Then, by the approximation and inverse properties of $S_{h}$, $S_{h,0}$, \eqref{eq22b} and Lemma 5.2 we have for
$0\leq n\leq n^{*}$
\[
\|Pe_{x}^{n}\|\leq \tfrac{C}{h}\|e^{n}\|, \qquad \|P(\ve^{n}e^{n})_{x}\|\leq \tfrac{C}{h}(\|\ve^{n}\| + \|e^{n}\|),
\]
\begin{align*}
\|P(H^{n}e^{n} + U^{n}\ve^{n})_{x}\| & \leq \|(H^{n}e^{n})_{x}\| + \|(U^{n}\ve^{n})_{x}\|\\
& \leq \|(\rho^{n}e^{n})_{x}\| + \|(\eta^{n}e^{n})_{x}\| + \|(\sigma^{n}\ve^{n})_{x}\| + \|(u^{n}\ve^{n})_{x}\|\\
& \leq Ch^{r-1} + \tfrac{C}{h}(\|\ve^{n}\| + \|e^{n}\|),
\end{align*}
where $\rho^{n}=\eta^{n} - H^{n}=\eta(t^{n}) - H^{n}$, $\sigma^{n}=u^{n}-U^{n}=u(t^{n})-U^{n}$. Thus,
\begin{align*}
\|F^{n}\|\leq & \tfrac{Ck^{2}}{h^{2}}(\|e^{n}\|^{2} + \|\ve^{n}\|^{2})
+ \tfrac{Ck^{2}}{h}h^{r-1}(\|e^{n}\| + \|\ve^{n}\|) + Ck^{2}h^{2r-2}\\
& + \|\ve^{n}\|\|\delta_{1}^{n}\| + \tfrac{Ck}{h}(\|e^{n}\| + \|\ve^{n}\|)\|\delta_{1}^{n}\|
+ Ckh^{r-1}\|\delta_{1}^{n}\| + \|\delta_{1}^{n}\|^{2},
\end{align*}
and therefore
\begin{align*}
\|F^{n}\|\leq & \mu Ck(\|e^{n}\|^{2} + \|\ve^{n}\|^{2}) + \tfrac{Ck^{2}}{h}(h^{2r-1}
+ \tfrac{\|e^{n}\|^{2} + \|\ve^{n}\|^{2}}{h}) + Ck^{2}h^{2r-2}\\
& + Ck\|\ve^{n}\|(k+h^{r-1}) + Ckh^{r-1}\|\delta_{1}^{n}\| + C\|\delta_{1}^{n}\|^{2}\\
\leq & \mu Ck(\|e^{n}\|^{2} + \|\ve^{n}\|^{2}) + Ck^{2}h^{2r-2}\\
& + Ck(\|e^{n}\|^{2} + \|\ve^{n}\|^{2}) + Ck(k+h^{r-1})^{2} + Ck^{2}h^{r-1}(k+h^{r-1})\\
\leq & C_{\mu}k(\|e^{n}\|^{2} + \|\ve^{n}\|^{2}) + Ckh^{2r-2} + Ck(k+h^{r-1})^{2},
\end{align*}
where $C_{\mu}$ is a polynomial of $\mu$ of degree one and with positive coefficients. Hence, for $0\leq n\leq n^{*}$,
\begin{equation}
\|F^{n}\| \leq C_{\mu}k(\|e^{n}\|^{2} + \|\ve^{n}\|^{2}) + Ck(k+h^{r-1})^{2}.
\label{eq57}
\end{equation}
Similarly, for $0\leq n\leq n^{*}$
\begin{equation}
\|G^{n}\| \leq C_{\mu}k(\|e^{n}\|^{2} + \|\ve^{n}\|^{2}) + Ck(k+h^{r-1})^{2}.
\label{eq58}
\end{equation}
Adding now the two equations of \eqref{eq56} we obtain
\[
\|\ve^{n+1}\|^{2} + \|e^{n+1}\|^{2} = \|\ve^{n}\|^{2} + \|e^{n}\|^{2} - k(e^{n},H_{x}^{n}\ve^{n})
+ k(\ve_{x}^{n},U^{n}\ve^{n}) + 3k(e_{x}^{n},U^{n}e^{n}) + F^{n} + G^{n},
\]
or
\[
\|\ve^{n+1}\|^{2} + \|e^{n+1}\|^{2} = \|\ve^{n}\|^{2} + \|e^{n}\|^{2} - k(e^{n},H_{x}^{n}\ve^{n})
- \tfrac{k}{2}(\ve^{n},U_{x}^{n}\ve^{n}) - \tfrac{3k}{2}(e^{n},U_{x}^{n}e^{n}) + F^{n} + G^{n}.
\]
But, for $0\leq n\leq n^{*}$,
\[
\abs{(e^{n},H_{x}^{n}\ve^{n})} \leq \abs{(e^{n},\rho_{x}^{n}\ve^{n})} + \abs{(e^{n},\eta_{x}^{n}\ve^{n})}
\leq Ch^{r-1}\|e^{n}\| + C\|\ve^{n}\|\|e^{n}\|,
\]
and similarly
\[
\abs{(\ve^{n},U_{x}^{n}\ve^{n})} + \abs{(e^{n},U_{x}^{n}e^{n})}\leq Ch^{r-1}(\|\ve^{n}\| + \|e^{n}\|)
+ C(\|\ve^{n}\|^{2} + \|e^{n}\|^{2}).
\]
Therefore, taking into account \eqref{eq57} and \eqref{eq58}, we obtain for $0\leq n\leq n^{*}$
\[
\|\ve^{n+1}\|^{2} + \|e^{n+1}\|^{2} \leq (1+C_{\mu}k)(\|\ve^{n}\|^{2} + \|e^{n}\|^{2}) + Ck(k+h^{r-1})^{2}.
\]
Hence, from the discrete Gronwall's lemma we see that there exists a constant $C(\mu,T)$ such that
\[
\|\ve^{n}\|^{2} + \|e^{n}\|^{2}\leq C(\mu,T)\bigl[\|\ve^{0}\|^{2} + \|e^{0}\|^{2} + (k + h^{r-1})^{2}\bigr],
\quad \text{for} \quad 0\leq n\leq n^{*}+1.
\]
We conclude, since $\ve^{0}=e^{0}=0$, that
\begin{equation}
\|\ve^{n}\| + \|e^{n}\| \leq C(\mu,T)(k+h^{r-1}),\quad 0\leq n\leq n^{*}+1.
\label{eq59}
\end{equation}
Since $k=\mu h^{2}$, we see by \eqref{eq24}, for $h$ sufficiently small, that
$\|\ve^{n^{*}+1}\|_{\infty} + \|e^{n^{*}+1}\|_{\infty} \leq 1$. This contradicts the maximal property of $n^{*}$
and we may reach $n^{*}=M-1$ in \eqref{eq59}. Since $H_{h}^{n} - \eta(t^{n})=-(\rho^{n} + \ve^{n})$,
$U_{h}^{n} - u(t^{n}) = - (\sigma^{n} + e^{n})$, the conclusion of the proposition follows.
\end{proof}
\subsection{The improved Euler method}
We next study the temporal discretization of the semidiscrete problem \eqref{eq27}, \eqref{eq28} by the explicit, second-order accurate
`improved Euler' scheme (the explict midpoint method), which for the o.d.e. $y'=f(t,y)$ may be written in the form
\begin{align*}
y^{n,1} & = y^{n} + \tfrac{k}{2}f(t^{n},y^{n}),\\
y^{n+1} & = y^{n} + kf(t^{n}+\tfrac{k}{2},y^{n,1}).
\end{align*}
Using the notation introduced in the previous subsection, but assuming now that $r\geq 3$, we seek approximations
$H_{h}^{n}\in S_{h}$, $U_{h}^{n}\in S_{h,0}$ of $\eta(\cdot,t^{n})$, $u(\cdot,t^{n})$, respectively, where $(\eta,u)$ is the solution
of the \eqref{eqssw}, and $H_{h}^{n,1}\in S_{h}$, $U_{h}^{n,1}\in S_{h,0}$, that are given for $0\leq n\leq M-1$ by the equations
\begin{equation}
\begin{aligned}
(H^{n,1}_{h},\phi) - (H^{n}_{h},\phi) + \tfrac{k}{2}(U^{n}_{hx},\phi)
+ \tfrac{k}{4}((H^{n}_{h}U^{n}_{h})_{x},\phi) & = 0, \quad \forall \phi \in S_{h}^{r},
\\
(U^{n,1}_{h},\chi) - (U_{h}^{n},\chi) + \tfrac{k}{2}(H^{n}_{hx},\chi)
+ \tfrac{k}{4}(H^{n}_{h}H^{n}_{hx},\chi) +
\tfrac{3k}{4}(U^{n}_{h}U^{n}_{hx},\chi) & = 0, \quad \forall \chi \in S_{h,0}^{r}, \\
(H^{n+1}_{h},\phi) - (H^{n}_{h},\phi) + k(U^{n,1}_{hx},\phi)
+ \tfrac{k}{2}((H^{n,1}_{h}U^{n,1}_{h})_{x},\phi) & = 0, \quad \forall \phi \in S_{h}^{r},
\\
(U^{n+1}_{h},\chi) - (U_{h}^{n},\chi) + k(H^{n,1}_{hx},\chi)
+ \tfrac{k}{2}(H^{n,1}_{h}H^{n,1}_{hx},\chi) +
\tfrac{3k}{2}(U^{n,1}_{h}U^{n,1}_{hx},\chi) & = 0, \quad \forall \chi \in S_{h,0}^{r},
\end{aligned}
\label{equ510}
\end{equation}
with
\[
H^{0}_{h}=P\eta_{0}, \qquad U_{h}^{0}=P_{0}u_{0}.
\]
The equations \eqref{equ510} may be written in the form
\begin{equation}
\begin{aligned}
H_{h}^{n,1} - H_{h}^{n} + \tfrac{k}{2}PU_{hx}^{n} + \tfrac{k}{4}P(H^{n}_{h}U^{n}_{h})_{x}&=0,
\\
U_{h}^{n,1} - U_{h}^{n} + \tfrac{k}{2}P_{0}H_{hx}^{n} + \tfrac{k}{4}P_{0}\bigl(H^{n}_{h}H^{n}_{hx}\bigr)
+ \tfrac{3k}{4}P_{0}\bigl(U^{n}_{h}U^{n}_{hx}\bigr)&=0,\\
H_{h}^{n+1} - H_{h}^{n} + kPU_{hx}^{n,1} + \tfrac{k}{2}P(H^{n,1}_{h}U^{n,1}_{h})_{x}&=0,
\\
U_{h}^{n+1} - U_{h}^{n} + kP_{0}H_{hx}^{n,1} + \tfrac{k}{2}P_{0}\bigl(H^{n,1}_{h}H^{n,1}_{hx}\bigr)
+ \tfrac{3k}{2}P_{0}\bigl(U^{n,1}_{h}U^{n,1}_{hx}\bigr)&=0.
\end{aligned}
\label{equ511}
\end{equation}
We start again by estimating the continuous-time truncation error.
\begin{lemma} Suppose that $(\eta,u)$ is the solution of \eqref{eqssw} on $[0,T]$. Let $H(t)=P\eta(t)$, $U(t)=P_{0}u(t)$,
\,$\psi(t)$ $\in$ $S_{h}$, \,$\zeta(t)$ $\in$ $S_{h,0}$ for $0\leq t\leq T$ be such that
\begin{equation}
\begin{aligned}
H_{t} + PU_{x} + \tfrac{1}{2}P(HU)_{x}&=\psi,
\\
U_{t}+ P_{0}H_{x} + \tfrac{1}{2}P_{0}\bigl(HH_{x}\bigr)
+ \tfrac{3}{2}P_{0}\bigl(UU_{x}\bigr)&=\zeta.
\end{aligned}
\label{equ512}
\end{equation}
Then
\begin{equation}
\|\psi\|+ \|\psi_{t}\| + \|\zeta\|+\|\zeta_{t}\| \leq Ch^{r-1}.
\label{equ513}
\end{equation}
\end{lemma}
\begin{proof} Subtracting $\eqref{equ512}_{1}$ from the equation $P\bigl(\eta_{t}+u_{x}+\tfrac{1}{2}(\eta u)_{x}\bigr)=0$,
and setting $\rho:=\eta- H$, $\sigma:=u-U$, we get
\[
P\bigl([(1+\tfrac{1}{2}\eta)\sigma]_{x} + \tfrac{1}{2}(u\rho)_{x}-\tfrac{1}{2}(\rho\sigma)_{x}\bigr)=-\psi.
\]
Thus, from the approximation properties of $S_{h}$, $S_{h,0}$, we have
\[
\|\psi\| \leq \|P[(1+\tfrac{1}{2}\eta)\sigma]_{x}\|+\tfrac{1}{2}\|P(u\rho)_{x}\|
+ \tfrac{1}{2}\|P(\rho\sigma)_{x}\| \leq Ch^{r-1}.
\]
Similarly,
\[
\|\psi_{t}\|\leq \|P[(1+\tfrac{1}{2}\eta)\sigma]_{xt}\|+\tfrac{1}{2}\|P(u\rho)_{xt}\|
+ \tfrac{1}{2}\|P(\rho\sigma)_{xt}\| \leq Ch^{r-1}.
\]
Subtracting $\eqref{equ512}_{2}$ from the equation $P_{0}\bigl(u_{t}+\eta_{x} + \tfrac{1}{2}\eta\eta_{x}+\tfrac{3}{2}uu_{x}\bigr)=0$
we have
\[
P_{0}\bigl(\rho_{x} + \tfrac{1}{2}(\eta\rho)_{x}-\tfrac{1}{2}\rho\rho_{x}
+ \tfrac{3}{2}(u\sigma)_{x} - \tfrac{3}{2}\sigma\sigma_{x}\bigr)=-\zeta.
\]
Therefore
\[
\|\zeta\|\leq \|P_{0}[(1+\tfrac{1}{2}\eta)\rho]_{x}\| + \tfrac{1}{2}\|P_{0}(\rho\rho_{x})\|
+\tfrac{3}{2}\|P_{0}(u\sigma)_{x}\| + \tfrac{3}{2}\|P_{0}(\sigma\sigma_{x})\|\leq Ch^{r-1},
\]
and
\[
\|\zeta_{t}\|\leq \|P_{0}[(1+\tfrac{1}{2}\eta)\rho]_{xt}\| + \tfrac{1}{2}\|P_{0}(\rho\rho)_{xt}\|
+\tfrac{3}{2}\|P_{0}(u\sigma)_{xt}\| + \tfrac{3}{2}\|P_{0}(\sigma\sigma)_{xt}\|\leq Ch^{r-1}.
\]
\end{proof}
In order to estimate the local error of the scheme, we let $H^{n}=H(t^{n})=P\eta(t^{n})$, $U^{n}=U(t^{n})=P_{0}u(t^{n})$
and define the functions $(H^{n,1},U^{n,1}) \in S_{h}\times S_{h,0}$ for $0\leq n\leq M-1$ by the equations
\begin{equation}
\begin{aligned}
H^{n,1} - H^{n} & + \tfrac{k}{2}PU_{x}^{n} + \tfrac{k}{4}P(H^{n}U^{n})_{x}=0, \\
U^{n,1} - U^{n} & + \tfrac{k}{2}P_{0}H_{x}^{n} + \tfrac{k}{4}P_{0}(H^{n}H_{x}^{n}) + \tfrac{3k}{4}P_{0}(U^{n}U_{x}^{n})=0.
\end{aligned}
\label{equ514}
\end{equation}
\begin{lemma} Suppose that $(\eta,u)$ is the solution of \eqref{eqssw} on $[0,T]$ and let $\lambda=k/h$. Define
$(\delta_{1}^{n}, \delta_{2}^{n})\in S_{h}\times S_{h,0}$ for $0\leq n\leq M-1$ by the relations
\begin{equation}
\begin{aligned}
\delta_{1}^{n}& :=H^{n+1} - H^{n} + kPU_{x}^{n,1} + \tfrac{k}{2}P(H^{n,1}U^{n,1})_{x}, \\
\delta_{2}^{n}& :=U^{n+1} - U^{n} + kP_{0}H_{x}^{n,1} + \tfrac{k}{2}P_{0}(H^{n,1}H_{x}^{n,1})
+ \tfrac{3k}{2}P_{0}(U^{n,1}U_{x}^{n,1}).
\end{aligned}
\label{equ515}
\end{equation}
Then there exists a constant $C=C(\lambda)$, depending polynomially on $\lambda$, so that
\begin{equation}
\max_{0\leq n\leq M-1}(\|\delta_{1}^{n}\| + \|\delta_{2}^{n}\|) \leq Ck(k^{2} + h^{r-1}).
\label{equ516}
\end{equation}
\end{lemma}
\begin{proof} Let $0\leq n\leq M-1$. From \eqref{equ514}, \eqref{equ512} we obtain
\begin{equation}
H^{n,1} = H^{n} + \tfrac{k}{2}H_{t}^{n} - \tfrac{k}{2}\psi^{n}, \qquad
U^{n,1} =U^{n} +\tfrac{k}{2}U_{t}^{n} - \tfrac{k}{2}\zeta^{n},
\label{equ517}
\end{equation}
where $\psi^{n}=\psi(t^{n})$ and $\zeta^{n}=\zeta(t^{n})$. From these equations we get
\[
H^{n,1}U^{n,1} = H^{n}U^{n} + \tfrac{k}{2}(HU)_{t}^{n} + w_{1}^{n},
\]
where
\begin{equation}
w_{1}^{n} = \tfrac{k^{2}}{4}H_{t}^{n}U_{t}^{n} - \tfrac{k}{2}(U^{n} + \tfrac{k}{2}U_{t}^{n})\psi^{n}
-\tfrac{k}{2}(H^{n} + \tfrac{k}{2}H_{t}^{n})\zeta^{n} + \tfrac{k^{2}}{4}\psi^{n}\zeta^{n}.
\label{equ518}
\end{equation}
From $(\ref{equ515})_{1}$, \eqref{equ517}, $(\ref{equ512})_{1}$, and \eqref{equ518} we see that
\begin{equation}
\delta_{1}^{n}=H^{n+1} - H^{n} - kH_{t}^{n} - \tfrac{k^{2}}{2}H_{tt}^{n} + k\psi^{n} + \tfrac{k^{2}}{2}\psi_{t}^{n}
-\tfrac{k^{2}}{2}P\zeta_{x}^{n} + \tfrac{k}{2}Pw_{1x}^{n}.
\label{equ519}
\end{equation}
From \eqref{equ515}, taking into account \eqref{equ513} and the approximation and inverse properties of $S_{h}$, $S_{h,0}$,
we obtain
\begin{align*}
\|w_{1}^{n}\|_{1} & \leq C(k^{2}\|H_{t}^{n}\|_{1}\|U_{t}^{n}\|_{1} + k\|\psi^{n}\|_{1}(\|U^{n}\|+k\|U_{t}^{n}\|)\\
&\,\,\,\,\, + k\|\zeta^{n}\|_{1}(\|H^{n}\|_{1} + k\|H_{t}^{n}\|_{1}) + k^{2}\|\psi^{n}\|_{1}\|\zeta^{n}\|_{1})\\
& \leq C(k^{2} + (\lambda + \lambda^{2})h^{r-1}).
\end{align*}
Hence, from Taylor's theorem, \eqref{equ513} and \eqref{equ519} we get
\begin{equation}
\|\delta_{1}^{n}\| \leq C_{1}k(k^{2} + h^{r-1}),
\label{equ520}
\end{equation}
where $C_{1}$ is a polynomial of $\lambda$ with positive coefficients. To estimate $\|\delta_{2}^{n}\|$ we have
from \eqref{equ517}
\begin{equation}
H^{n,1}H_{x}^{n,1} = H^{n}H_{x}^{n} + \tfrac{k}{2}(HH_{x})_{t}^{n} + w_{2}^{n},
\label{equ521}
\end{equation}
where
\[
w_{2}^{n}:=\tfrac{k^{2}}{4}H_{t}^{n}H_{tx}^{n}-\tfrac{k}{2}\bigl((H^{n} + \tfrac{k}{2}H_{t}^{n})\psi^{n}\bigr)_{x}
+\tfrac{k^{2}}{4}\psi^{n}\psi_{x}^{n}.
\]
From \eqref{equ513} and the approximation and inverse properties of $S_{h}$ we obtain
\begin{equation}
\|w_{2}^{n}\| \leq C(k^{2} + \lambda h^{r-1}).
\label{equ522}
\end{equation}
Similarly,
\begin{equation}
U^{n,1}U_{x}^{n,1} = U^{n}U_{x}^{n} + \tfrac{k}{2}(UU_{x})_{t}^{n} + w_{3}^{n},
\label{equ523}
\end{equation}
where
\[
w_{3}^{n}:=\tfrac{k^{2}}{4}U_{t}^{n}U_{tx}^{n}-\tfrac{k}{2}\bigl((U^{n} + \tfrac{k}{2}U_{t}^{n})\zeta^{n}\bigr)_{x}
+\tfrac{k^{2}}{4}\zeta^{n}\zeta_{x}^{n}.
\]
With similar estimates we obtain
\begin{equation}
\|w_{3}^{n}\| \leq C(k^{2} + \lambda h^{r-1}).
\label{equ524}
\end{equation}
Now, from $(\ref{equ515})_{1}$, \eqref{equ517}, \eqref{equ521}, \eqref{equ523} we see that
\[
\delta_{2}^{n}=(U^{n+1} - U^{n} - kU_{t}^{n} - \tfrac{k^{2}}{2}U_{tt}^{n}) + k\zeta^{n} + \tfrac{k^{2}}{2}\zeta_{t}^{n}
- \tfrac{k^{2}}{2}P_{0}\psi_{x}^{n} + \tfrac{k}{2}P_{0}w_{2}^{n} + \tfrac{3k}{2}P_{0}w_{3}^{n},
\]
and therefore, from Taylor's theorem, \eqref{equ513}, \eqref{equ522}, \eqref{equ524} we get
\[
\|\delta_{2}^{n}\| \leq C_{1}k(k^{2}+h^{r-1}),
\]
for some constant $C_{1}$ which is a polynomial of $\lambda$ with positive coefficients. This inequality and \eqref{equ520}
conclude the proof.
\end{proof}
In order to prove the main error estimate for the scheme, we state and prove some preliminary results. Given $H^{n}$, $U^{n}$
defined as before we define the operators $A : S_{h}\times S_{h,0} \to S_{h}$, $B : S_{h}\times S_{h,0} \to S_{h,0}$
for $\phi \in S_{h}$, $\chi \in S_{h,0}$ by the equations
\begin{equation}
\begin{aligned}
A(\phi,\chi) &: = \tfrac{1}{4}P(U^{n}\phi)_{x} + \tfrac{1}{4}P(H^{n}\chi)_{x} + \tfrac{1}{2}P\chi' - \tfrac{1}{4}P(\phi\chi)',\\
B(\phi,\chi) &: = \tfrac{1}{4}P_{0}(H^{n}\phi)_{x} + \tfrac{3}{4}P_{0}(U^{n}\chi)_{x} + \tfrac{1}{2}P_{0}\phi'
-\tfrac{1}{4}P_{0}(\phi\phi') - \tfrac{3}{4}P_{0}(\chi\chi').
\end{aligned}
\label{equ525}
\end{equation}
\begin{lemma} If $\kp \in \mathbb{R}$ and $\phi \in S_{h}$, $\chi \in S_{h,0}$, then
\begin{equation}
\begin{aligned}
A(\phi_{1}-\kp\phi_{2},\chi_{1}-\kp\chi_{2}) & = A(\phi_{1},\chi_{1})-\kp A(\phi_{2},\chi_{2})
+ \tfrac{\kp}{4}P(\phi_{1}\chi_{2} + \phi_{2}\chi_{1})' - \tfrac{\kp + \kp^{2}}{4}P(\phi_{2}\chi_{2})', \\
B(\phi_{1}-\kp\phi_{2},\chi_{1}-\kp\chi_{2}) & = B(\phi_{1},\chi_{1})-\kp B(\phi_{2},\chi_{2})
+ \tfrac{\kp}{4}P_{0}(\phi_{1}\phi_{2})' + \tfrac{3\kp}{4}P_{0}(\chi_{1}\chi_{2})' \\
&\quad
- \tfrac{\kp+\kp^{2}}{4}P_{0}(\phi_{2}\phi_{2}')
-\tfrac{3(\kp+\kp^{2})}{4}P_{0}(\chi_{2}\chi_{2}').
\end{aligned}
\label{equ526}
\end{equation}
\end{lemma}
\begin{proof}
From $(\ref{equ525})_{1}$ we have
\begin{align*}
A(\phi_{1}-\kp\phi_{2},\chi_{1}-\kp\chi_{2}) & = \tfrac{1}{4}P\bigl(U^{n}(\phi_{1}-\kp\phi_{2})\bigr)_{x}
+\tfrac{1}{4}P\bigl(H^{n}(\chi_{1}-\kp\chi_{2})\bigr)_{x} + \tfrac{1}{2}P(\chi_{1}-\kp\chi_{2})'\\
& \quad
-\tfrac{1}{4}P\bigl((\phi_{1}-\kp\phi_{2})(\chi_{1}-\kp\chi_{2})\bigr)'\\
& = A(\phi_{1},\chi_{1})-\kp A(\phi_{2},\chi_{2}) + \tfrac{\kp}{4}P(\phi_{1}\chi_{2}+\phi_{2}\chi_{1})'
-\tfrac{\kp+\kp^{2}}{4}P(\phi_{2}\chi_{2})'.
\end{align*}
From $(\ref{equ525})_{2}$ we obtain
\begin{align*}
B(\phi_{1}-\kp\phi_{2},\chi_{1}-\kp\chi_{2}) & = \tfrac{1}{4}P_{0}\bigl(H^{n}(\phi_{1}-\kp\phi_{2})\bigr)_{x}
+\tfrac{3}{4}P_{0}\bigl(U^{n}(\chi_{1}-\kp\chi_{2})\bigr)_{x} + \tfrac{1}{2}P_{0}(\phi_{1}-\kp\phi_{2})'\\
& \quad
-\tfrac{1}{4}P_{0}\bigl((\phi_{1}-\kp\phi_{2})(\phi_{1}'-\kp\phi_{2}')\bigr)
-\tfrac{3}{4}P_{0}\bigl((\chi_{1}-\kp\chi_{2})(\chi_{1}'-\kp\chi_{2}')\bigr)\\
& = B(\phi_{1},\chi_{1}) - \kp B(\phi_{2},\chi_{2}) + \tfrac{\kp}{4}P_{0}(\phi_{1}\phi_{2})'
+ \tfrac{3\kp}{4}P_{0}(\chi_{1}\chi_{2})' \\
&\quad
-\tfrac{\kp+\kp^{2}}{4}P_{0}(\phi_{2}\phi_{2}') - \tfrac{3(\kp+\kp^{2})}{4}P_{0}(\chi_{2}\chi_{2}').
\end{align*}
\end{proof}
In the sequel we let
$\ve^{n}=H^{n}-H_{h}^{n}$, $e^{n}=U^{n}-U_{h}^{n}$, $\theta^{n}=H^{n,1}-H_{h}^{n,1}$, $\xi^{n}=U^{n,1} - U_{h}^{n,1}$, and let
again $\lambda=k/h$.
\begin{lemma} There holds that
\begin{equation}
\begin{aligned}
\theta^{n} & = \ve^{n} - kA(\ve^{n},e^{n}),\\
\xi^{n} & =e^{n} - kB(\ve^{n},e^{n}).
\end{aligned}
\label{equ527}
\end{equation}
\end{lemma}
\begin{proof} Subtracting $(\ref{equ511})_{1}$ from $(\ref{equ514})_{1}$, we obtain
\begin{equation}
\theta^{n}- \ve^{n} + \tfrac{k}{2}Pe_{x}^{n} + \tfrac{k}{4}P(H^{n}U^{n}-H_{h}^{n}U_{h}^{n})_{x}=0.
\label{equ528}
\end{equation}
Since
\[
H^{n}U^{n}-H_{h}^{n}U_{h}^{n} = H^{n}e^{n} + U^{n}\ve^{n} - e^{n}\ve^{n},
\]
we get from \eqref{equ528}
\[
\theta^{n}-\ve^{n} + \tfrac{k}{2}Pe_{x}^{n} + \tfrac{k}{4}P(H^{n}e^{n})_{x} + \tfrac{k}{4}P(U^{n}\ve^{n})_{x}
-\tfrac{k}{4}P(\ve^{n}e^{n})_{x}=0,
\]
and by the definition of $A(\ve^{n},e^{n})$ and $(\ref{equ525})_{1}$ we obtain $(\ref{equ527})_{1}$. Subtracting now $(\ref{equ511})_{2}$
from $(\ref{equ514})_{2}$ we get that
\begin{equation}
\begin{aligned}
\xi^{n} - e^{n} + \tfrac{k}{2}P_{0}\ve_{x}^{n} + \tfrac{k}{4}P_{0}(H^{n}H_{x}^{n}-H_{h}^{n}H_{hx}^{n})
+ \tfrac{3k}{4}P_{0}(U^{n}U_{x}^{n}-U_{h}^{n}U_{hx}^{n})=0.
\end{aligned}
\label{equ529}
\end{equation}
Since
\[
H^{n}H_{x}^{n}-H_{h}^{n}H_{hx}^{n} = (H^{n}\ve^{n})_{x} - \ve^{n}\ve_{x}^{n},
\]
and
\[
U^{n}U_{x}^{n}-U_{h}^{n}U_{hx}^{n} = (U^{n}e^{n})_{x} - e^{n}e_{x}^{n},
\]
we see from \eqref{equ529} that
\[
\xi^{n} - e^{n} + \tfrac{k}{2}P_{0}\ve_{x}^{n} + \tfrac{k}{4}P_{0}(H^{n}\ve^{n})_{x} + \tfrac{3k}{4}P_{0}(U^{n}e^{n})_{x}
-\tfrac{k}{4}P_{0}(\ve^{n}\ve_{x}^{n}) - \tfrac{3k}{4}P_{0}(e^{n}e_{x}^{n})=0.
\]
Taking into account the definition of $B(\ve^{n},e^{n})$ and $(\ref{equ525})_{2}$ we obtain now $(\ref{equ527})_{2}$.
\end{proof}
\begin{lemma} If $\|\ve^{n}\|_{1,\infty} + \|e^{n}\|_{1,\infty} \leq 1$, for some index $n$, then
\begin{equation}
\begin{aligned}
\|A(\ve^{n},e^{n})\| + \|B(\ve^{n},e^{n})\| & \leq Ch^{-1}(\|\ve^{n}\| + \|e^{n}\|),\\
\|A(\ve^{n},e^{n})\|_{\infty} + \|B(\ve^{n},e^{n})\|_{\infty} & \leq C h^{-1},\\
\|\theta^{n}\| + \|\xi^{n}\| & \leq C_{2}(\|\ve^{n}\| + \|e^{n}\|),\\
\|\theta^{n}\|_{\infty} + \|\xi^{n}\|_{\infty} & \leq C, \\
\|A(\theta^{n},\xi^{n})\| + \|B(\theta^{n},\xi^{n})\| & \leq Ch^{-1}(\|\ve^{n}\| + \|e^{n}\|),
\end{aligned}
\label{equ530}
\end{equation}
where $C_{2}$ a constant depending polynomially on $\lambda$.
\end{lemma}
\begin{proof} The estimates $(\ref{equ530})_{1}$ and $(\ref{equ530})_{2}$ follow from the definitions of $A(\ve^{n},e^{n})$ and
$B(\ve^{n},e^{n})$, the inverse and approximation properties of $S_{h}$, $S_{h,0}$ and the fact that
$\|H^{n}\|_{W_{\infty}^{1}} + \|U^{n}\|_{W_{\infty}^{1}}\leq C=C(\eta,u)$. The inequalities $(\ref{equ530})_{3}$, $(\ref{equ530})_{4}$
follow from \eqref{equ527} and $(\ref{equ530})_{1}$, $(\ref{equ530})_{2}$. Finally, $(\ref{equ530})_{5}$ follows from the definition
of $A(\theta^{n},\xi^{n})$, $B(\theta^{n},\xi^{n})$, the inverse and approximation properties of $S_{h}$, $S_{h,0}$, $(\ref{equ530})_{3}$
and the fact that $\|H^{n}\|_{W_{\infty}^{1}} + \|U^{n}\|_{W_{\infty}^{1}}\leq C=C(\eta,u)$.
\end{proof}
\begin{lemma} Let $f^{n}=H^{n,1} - H^{n}$, $g^{n}=U^{n,1} - U^{n}$, and
\begin{equation}
\begin{aligned}
\omega_{1}^{n} & = \tfrac{k}{2}P(f^{n}\xi^{n} + g^{n}\theta^{n})_{x} - \delta_{1}^{n},\\
\omega_{2}^{n} & = \tfrac{k}{2}P_{0}(f^{n}\theta^{n})_{x} + \tfrac{3k}{2}P_{0}(g^{n}\xi^{n})_{x} - \delta_{2}^{n}.
\end{aligned}
\label{equ531}
\end{equation}
If $\|\ve^{n}\|_{1,\infty} + \|e^{n}\|_{1,\infty} \leq 1$, then
\begin{equation}
\begin{aligned}
& \|f^{n}\|_{\infty} + \|g^{n}\|_{\infty}  \leq Ck,\\
& \|\omega_{1}^{n}\| \leq C_{3}k(\|\ve^{n}\| + \|e^{n}\|) + \|\delta_{1}^{n}\|,\\
& \|\omega_{2}^{n}\| \leq C_{4}k(\|\ve^{n}\| + \|e^{n}\|) + \|\delta_{2}^{n}\|,
\end{aligned}
\label{equ532}
\end{equation}
where $C_{3}$, $C_{4}$ are constants depending polynomially on $\lambda$.
\end{lemma}
\begin{proof} The inequality $(\ref{equ532})_{1}$ follows from \eqref{equ514} and the fact that
$\|H^{n}\|_{W_{\infty}^{1}} + \|U^{n}\|_{W_{\infty}^{1}}\leq C=C(\eta,u)$. From the inverse properties of $S_{h}$, $S_{h,0}$
we have
\begin{align*}
\|\omega_{1}^{n}\| & \leq Ck(\|f_{x}^{n}\|_{\infty}\|\xi^{n}\| + \|f^{n}\|_{\infty}\|\xi^{n}\|_{1})
+ Ck(\|g_{x}^{n}\|_{\infty}\|\theta^{n}\| + \|g^{n}\|_{\infty}\|\theta^{n}\|_{1}) + \|\delta_{1}^{n}\|\\
& \leq C\lambda\|f^{n}\|_{\infty}\|\xi^{n}\| + C\lambda\|g^{n}\|_{\infty}\|\theta^{n}\| + \|\delta_{1}^{n}\|.
\end{align*}
Hence, $(\ref{equ532})_{2}$ follows from $(\ref{equ532})_{1}$, $(\ref{equ532})_{3}$. The inequality
$(\ref{equ532})_{3}$ follows in a similar manner.
\end{proof}
\begin{proposition} Let $(\eta,u)$ be the solution of \eqref{eqssw}, $H_{h}^{n}$, $U_{h}^{n}$ the solution of
\eqref{equ510} and $\mu=k/h^{4/3}$. Then, there exists a constant $C=C(\mu)$ such that
\[
\max_{0\leq n\leq M}(\|H_{h}^{n} - \eta(t^{n})\| + \|U_{h}^{n}-U(t^{n})\|) \leq C(k^{2} + h^{r-1}).
\]
\end{proposition}
\begin{proof} We will show that
\begin{equation}
\max_{0\leq n\leq M}(\|\ve^{n}\| + \|e^{n}\|) \leq C(k^{2} + h^{r-1}),
\label{equ533}
\end{equation}
from which the result of the proposition follows. From $(\ref{equ511})_{3}$, $(\ref{equ511})_{4}$, \eqref{equ515} we have
\begin{equation}
\begin{aligned}
\ve^{n+1} & = \ve^{n} - kP\xi_{x}^{n} - \tfrac{k}{2}P(H^{n,1}U^{n,1} - H_{h}^{n,1}U_{h}^{n,1})_{x} + \delta_{1}^{n},\\
e^{n+1} & = e^{n} - kP_{0}\theta_{x}^{n}-\tfrac{k}{2}P_{0}(H^{n,1}H_{x}^{n,1}-H_{h}^{n,1}H_{hx}^{n,1})
-\tfrac{3k}{2}P_{0}(U^{n,1}U_{x}^{n,1}-U_{h}^{n,1}U_{hx}^{n,1}) + \delta_{2}^{n}.
\end{aligned}
\label{equ534}
\end{equation}
But
\[
H^{n,1}U^{n,1} - H_{h}^{n,1}U_{h}^{n,1}=H^{n,1}\xi^{n} + U^{n,1}\theta^{n} - \theta^{n}\xi^{n},
\]
and therefore
\begin{equation}
H^{n,1}U^{n,1} - H_{h}^{n,1}U_{h}^{n,1}=H^{n}\xi^{n} + U^{n}\theta^{n} - \theta^{n}\xi^{n} + f^{n}\xi^{n} + g^{n}\theta^{n},
\label{equ535}
\end{equation}
where $f^{n}$, $g^{n}$ as in Lemma 5.8. Also,
\begin{align*}
H^{n,1}H_{x}^{n,1} - H_{h}^{n,1}H_{hx}^{n,1} & = (H^{n,1}\theta^{n})_{x} - \theta^{n}\theta_{x}^{n},\\
U^{n,1}U_{x}^{n,1} - U_{h}^{n,1}U_{hx}^{n,1} & = (U^{n,1}\xi^{n})_{x} - \xi^{n}\xi_{x}^{n}.
\end{align*}
Hence
\begin{equation}
\begin{aligned}
H^{n,1}H_{x}^{n,1} - H_{h}^{n,1}H_{hx}^{n,1} & = (H^{n}\theta^{n})_{x} - \theta^{n}\theta_{x}^{n} + (f^{n}\theta^{n})_{x},\\
U^{n,1}U_{x}^{n,1} - U_{h}^{n,1}U_{hx}^{n,1} & = (U^{n}\xi^{n})_{x} - \xi^{n}\xi_{x}^{n} + (g^{n}\xi^{n})_{x}.
\end{aligned}
\label{equ536}
\end{equation}
Following \eqref{equ535}, \eqref{equ536}, \eqref{equ531} we write the equations \eqref{equ534} in the form
\begin{align*}
\ve^{n+1} & = \ve^{n} - kP\xi_{x}^{n} - \tfrac{k}{2}P(H^{n}\xi^{n})_{x} - \tfrac{k}{2}P(U^{n}\theta^{n})_{x}
+\tfrac{k}{2}P(\theta^{n}\xi^{n})_{x} - \omega_{1}^{n},\\
e^{n+1} & = e^{n} - kP_{0}\theta_{x}^{n} - \tfrac{k}{2}P_{0}(H^{n}\theta^{n})_{x} + \tfrac{k}{2}P_{0}(\theta^{n}\theta_{x}^{n})
-\tfrac{3k}{2}P_{0}(U^{n}\xi^{n})_{x} + \tfrac{3k}{2}P_{0}(\xi^{n}\xi_{x}^{n}) - \omega_{2}^{n}.
\end{align*}
Using the definition of $A(\theta^{n},\xi^{n})$, $B(\theta^{n},\xi^{n})$ (cf. \eqref{equ525}) we write the above as
\begin{align*}
\ve^{n+1} & = \ve^{n} - 2kA(\theta^{n},\xi^{n}) - \omega_{1}^{n}, \\
e^{n+1} & = e^{n} - 2kB(\theta^{n},\xi^{n}) - \omega_{2}^{n}.
\end{align*}
Then
\begin{equation}
\begin{aligned}
\|\ve^{n+1}\|^{2} & = \|\ve^{n}-2kA(\theta^{n},\xi^{n})\|^{2} + \|\omega_{1}^{n}\|^{2} - 2(\ve^{n}-2kA(\theta^{n},\xi^{n}),\omega_{1}^{n})\\
& \leq \|\ve^{n}-2kA(\theta^{n},\xi^{n})\|^{2} + \|\omega_{1}^{n}\|^{2} + 2\|\ve^{n}\|\|\omega_{1}^{n}\| +
4k\|A(\theta^{n},\xi^{n})\|\|\omega_{1}^{n}\|,
\end{aligned}
\label{equ537}
\end{equation}
and
\begin{equation}
\begin{aligned}
\|e^{n+1}\|^{2} & = \|e^{n}-2kB(\theta^{n},\xi^{n})\|^{2} + \|\omega_{2}^{n}\|^{2} - 2(e^{n}-2kB(\theta^{n},\xi^{n}),\omega_{2}^{n})\\
& \leq \|e^{n}-2kB(\theta^{n},\xi^{n})\|^{2} + \|\omega_{2}^{n}\|^{2} + 2\|e^{n}\|\|\omega_{2}^{n}\| +
4k\|B(\theta^{n},\xi^{n})\|\|\omega_{2}^{n}\|.
\end{aligned}
\label{equ538}
\end{equation}
Unlike the proof of convergence of the Euler scheme where the `temporary' hypothesis that
$\|\ve^{n}\|_{\infty} + \|e^{n}\|_{\infty}\leq 1$ up to some index $n^{*}$ was sufficient, the present proof requires a
stronger hypothesis, which necessitates that $r$ should be taken at least 3.
Let $0\leq n^{*}\leq M-1$ be the maximal index for which it holds that
\[
\|\ve^{n}\|_{1,\infty} + \|e^{n}\|_{1,\infty} \leq 1, \quad 0\leq n\leq n^{*}.
\]
Then, from $(\ref{equ530})_{5}$, having in mind $(\ref{equ532})_{2}$ as well, we obtain for $0\leq n\leq n^{*}$
\begin{align*}
\|\omega_{1}^{n}\|^{2} + 2\|\ve^{n}\|\|\omega_{1}^{n}\| & + 4k\|A(\theta^{n},\xi^{n}\|\|\omega_{1}^{n}\|
\leq \|\omega_{1}^{n}\|^{2} + C(\|\ve^{n}\| + \|e^{n}\|)\|\omega_{1}^{n}\|\\
& \leq \bigl(\|\omega_{1}^{n}\| + C(\|\ve^{n}\| + \|e^{n}\|)\bigr)\|\omega_{1}^{n}\|\\
& \leq \bigl(C(\|\ve^{n}\| + \|e^{n}\|) + \|\delta_{1}^{n}\|\bigr)\bigl(C_{5}k(\|\ve^{n}\| + \|e^{n}\|) + \|\delta_{1}^{n}\|\bigr)\\
& \leq C_{6}k(\|\ve^{n}\|^{2} + \|e^{n}\|^{2}) + C_{7}(\|\ve^{n}\| + \|e^{n}\|)\|\delta_{1}^{n}\| + \|\delta_{1}^{n}\|^{2},
\end{align*}
where $C_{5}$, $C_{6}$, $C_{7}$ are constants that depend on $\lambda=k/h$. From \eqref{equ516} we see that
\begin{align*}
(\|\ve^{n}\| + \|e^{n}\|)\|\delta_{1}^{n}\| & \leq C\sqrt{k}(\|\ve^{n}\| + \|e^{n}\|)\cdot \sqrt{k}(k^{2} + h^{r-1})\\
& \leq Ck(\|\ve^{n}\|^{2} + \|e^{n}\|^{2}) + Ck(k^{2} + h^{r-1})^{2},
\end{align*}
and therefore, for constants $C_{8}$, $C_{9}$ depending on $\lambda$, that
\[
\|\omega_{1}^{n}\|^{2} + 2\|\ve^{n}\|\|\omega_{1}^{n}\| + 4k\|A(\theta^{n},\xi^{n}\|\|\omega_{1}^{n}\|
\leq C_{8}k(\|\ve^{n}\|^{2} + \|e^{n}\|^{2}) + C_{9}k(k^{2} + h^{r-1})^{2}.
\]
(In the sequel we will not be mentioning the dependence of constants on $\lambda$.) Finally, for $0\leq n\leq n^{*}$ we have
from \eqref{equ537}
\[
\|\ve^{n+1}\|^{2} \leq \|\ve^{n}-2kA(\theta^{n},\xi^{n})\|^{2} + Ck(\|\ve^{n}\|^{2} + \|e^{n}\|^{2}) + Ck(k^{2} + h^{r-1})^{2}.
\]
Similarly, from $(\ref{equ530})_{5}$, $(\ref{equ532})_{3}$, \eqref{equ516} and for $0\leq n\leq n^{*}$ we get from \eqref{equ538}
\[
\|e^{n+1}\|^{2} \leq \|e^{n}-2kB(\theta^{n},\xi^{n})\|^{2} + Ck(\|\ve^{n}\|^{2} + \|e^{n}\|^{2}) + Ck(k^{2} + h^{r-1})^{2}.
\]
Adding the last two inequalities yields
\begin{equation}
\begin{aligned}
\|\ve^{n+1}\|^{2} + \|e^{n+1}\|^{2} & \leq \|\ve^{n}-2kA(\theta^{n},\xi^{n})\|^{2} + \|e^{n}-2kB(\theta^{n},\xi^{n})\|^{2}\\
& \quad + Ck(\|\ve^{n}\|^{2} + \|e^{n}\|^{2}) + Ck(k^{2} + h^{r-1})^{2},
\end{aligned}
\label{equ539}
\end{equation}
for $0\leq n\leq n^{*}$. In what follows we will show that if $k=O(h^{4/3})$, then for $0\leq n\leq n^{*}$
\[
\|\ve^{n}-2kA(\theta^{n},\xi^{n})\|^{2} + \|e^{n}-2kB(\theta^{n},\xi^{n})\|^{2}
\leq (1 + Ck)(\|\ve^{n}\|^{2} + \|e^{n}\|^{2}),
\]
for some constant $C$ depending on $\mu$. \par
We note first that
\begin{equation}
\begin{aligned}
\|\ve^{n} & -2kA(\theta^{n},\xi^{n})\|^{2} + \|e^{n}-2kB(\theta^{n},\xi^{n})\|^{2} =
\|\ve^{n}\|^{2} + \|e^{n}\|^{2}\\
& -4k\bigl((\ve^{n},A(\theta^{n},\xi^{n})) + (e^{n},B(\theta^{n},\xi^{n}))\bigr)
+ 4k^{2}\bigl(\|A(\theta^{n},\xi^{n})\|^{2} + \|B(\theta^{n},\xi^{n})\|^{2}\bigr).
\end{aligned}
\label{equ540}
\end{equation}
In the sequel instead of $A(\ve^{n},e^{n})$ and $B(\ve^{n},e^{n})$ we will write $A$ and $B$, respectively. From
\eqref{equ527}, \eqref{equ526} it follows that
\begin{align*}
A(\theta^{n},\xi^{n}) = A(\ve^{n}-kA,e^{n}-kB) & = A - kA(A,B) +\tfrac{k}{4}P(\ve^{n}B+e^{n}A)_{x}
-\tfrac{k+k^{2}}{4}P(AB)_{x},\\
B(\theta^{n},\xi^{n}) = B(\ve^{n}-kA,e^{n}-kB) & = B - kB(A,B) + \tfrac{k}{4}P_{0}(\ve^{n}A)_{x} + \tfrac{3k}{4}P_{0}(e^{n}B)_{x}\\
&\quad -\tfrac{k+k^{2}}{4}P_{0}(AA_{x}) -\tfrac{3(k+k^{2})}{4}P_{0}(BB_{x}).
\end{align*}
Hence
\begin{equation}
\begin{aligned}
4k\bigl((\ve^{n},A(\theta^{n},\xi^{n})+(e^{n},B(\theta^{n},\xi^{n})\bigr) & = 4k\bigl((\ve^{n},A) + (e^{n},B)\bigr) \\
&\quad - 4k^{2}\bigl((\ve^{n},A(A,B)) + (e^{n},B(A,B))\bigr)\\
&\quad -k^{2}(\ve_{x}^{n},\ve^{n}B+e^{n}A)+ (k^{2}+k^{3})(\ve_{x}^{n},AB) \\
&\quad - k^{2}(e_{x}^{n},\ve^{n}A) - 3k^{2}(e_{x}^{n},e^{n}B)-(k^{2}+k^{3})(e^{n},AA_{x}) \\
&\quad - 3(k^{2}+k^{3})(e^{n},BB_{x}).
\end{aligned}
\label{equ541}
\end{equation}
Therefore for the $O(k)$ terms in \eqref{equ540} we finally have (cf. \eqref{equ525}) that
\begin{align*}
4k\bigl((\ve^{n},A) + (e^{n},B)\bigr) & = -k(\ve_{x}^{n},U^{n}\ve^{n}) - k(\ve_{x}^{n},H^{n}e^{n})-2k(\ve_{x}^{n},e^{n})
+ k(\ve_{x}^{n},\ve^{n}e^{n}) \\
& \quad  -k(e_{x}^{n},H^{n}\ve^{n}) - 3k(e_{x}^{n},U^{n}e^{n})-2k(e_{x}^{n},\ve^{n})-k(e^{n},\ve^{n}\ve_{x}^{n}),
\end{align*}
and thus
\begin{equation}
\begin{aligned}
-4k\bigl((\ve^{n},A) + (e^{n},B)\bigr) & = -\tfrac{k}{2}(\ve^{n},U_{x}^{n}\ve^{n}) - k(\ve^{n},H_{x}^{n}e^{n})
- \tfrac{3k}{2}(e^{n},U_{x}^{n}e^{n})\\
& \leq Ck(\|\ve^{n}\|^{2} + \|e^{n}\|^{2}).
\end{aligned}
\label{equ542}
\end{equation}
Gathering the $O(k^{2})$ terms of \eqref{equ540} we obtain
\begin{align*}
K2 & :=4k^{2}\|A\|^{2} + 4k^{2}\|B\|^{2} + 4k^{2}(\ve^{n},A(A,B)) + 4k^{2}(e^{n},B(A,B))
+ k^{2}(\ve_{x}^{n},\ve^{n}B+e^{n}A)\\
& \quad   - k^{2}(\ve_{x}^{n},AB) + k^{2}(e_{x}^{n},\ve^{n}A) + 3k^{2}(e_{x}^{n},e^{n}B) + k^{2}(e^{n},AA_{x})
+ 3k^{2}(e^{n},BB_{x}).
\end{align*}
If $K2_{1}:=4k^{2}\|A\|^{2} + 4k^{2}\|B\|^{2}=4k^{2}(A,A) + 4k^{2}(B,B)$, then, from \eqref{equ525}
\begin{align*}
K2_{1}& =k^{2}((U^{n}\ve^{n})_{x},A) + k^{2}((H^{n}e^{n})_{x},A)
+ 2k^{2}(e_{x}^{n},A) -k^{2}((\ve^{n}e^{n})_{x},A) \\
& \quad  + k^{2}((H^{n}\ve^{n})_{x},B) + 3k^{2}((U^{n}e^{n})_{x},B) + 2k^{2}(\ve_{x}^{n},B)
- k^{2}(\ve^{n}\ve_{x}^{n},B)-3k^{2}(e^{n}e_{x}^{n},B).
\end{align*}
If $K2_{2}:=4k^{2}(\ve^{n},A(A,B)) + 4k^{2}(e^{n},B(A,B))$, then, from \eqref{equ525} again
\begin{align*}
K2_{2} & =-k^{2}(\ve_{x}^{n},U^{n}A) - k^{2}(\ve_{x}^{n},H^{n}B) - 2k^{2}(\ve_{x}^{n},B) + k^{2}(\ve_{x}^{n},AB)\\
& \quad -k^{2}(e_{x}^{n},H^{n}A) - 3k^{2}(e_{x}^{n},U^{n}B) - 2k^{2}(e_{x}^{n},A) - k^{2}(e^{n},AA_{x})
-3k^{2}(e^{n},BB_{x}).
\end{align*}
Adding the last two equations yields
\begin{align*}
K2_{1} + K2_{2}& =k^{2}(U_{x}^{n}\ve^{n},A) + k^{2}(H_{x}^{n}e^{n},A) - k^{2}((\ve^{n}e^{n})_{x},A)
+ k^{2}(H_{x}^{n}\ve^{n},B) +3k^{2}(U_{x}^{n}e^{n},B) \\
& \quad  - k^{2}(\ve^{n}\ve_{x}^{n},B) - 3k^{2}(e^{n}e_{x}^{n},B) + k^{2}(\ve_{x}^{n},AB)-k^{2}(e^{n},AA_{x})
-3k^{2}(e^{n},BB_{x}),
\end{align*}
and so
\begin{align*}
K2&=K2_{1} + K2_{2} + k^{2}(\ve_{x}^{n},\ve^{n}B) + k^{2}((\ve^{n}e^{n})_{x},A)-k^{2}(\ve_{x}^{n},AB)
+ 3k^{2}(e_{x}^{n},e^{n}B)\\
&\quad  + k^{2}(e^{n},AA_{x}) + 3k^{2}(e^{n},BB_{x}),
\end{align*}
i.e.
\[
K2=k^{2}(U_{x}^{n}\ve^{n},A) + k^{2}(H_{x}^{n}e^{n},A) + k^{2}(H_{x}^{n}\ve^{n},B) + 3k^{2}(U_{x}^{n}e^{n},B),
\]
and consequently, according to $(\ref{equ530})_{1}$, for $0\leq n\leq n^{*}$, there holds that
\begin{equation}
K2 \leq Ck^{2}(\|\ve^{n}\|+\|e^{n}\|)(\|A\| + \|B\|) \leq Ck(\|\ve^{n}\|^{2} + \|e^{n}\|^{2}).
\label{equ543}
\end{equation}
The $O(k^{3})$ terms of \eqref{equ540} are
\begin{align*}
K3& :=-k^{3}(\ve_{x}^{n},AB) + k^{3}(e^{n},AA_{x}) + 3k^{3}(e^{n},BB_{x}) - 8k^{3}(A,A(A,B)) + 2k^{3}(A,(\ve^{n}B+e^{n}A)_{x})\\
&\quad -2k^{3}(A,(AB)_{x}) -8k^{3}(B,B(A,B)) +2k^{3}(B,(\ve^{n}A)_{x}) +6k^{3}(B,(e^{n}B)_{x})-2k^{3}(B,AA_{x}).
\end{align*}
Hence
\[
K3=k^{3}(\ve_{x}^{n},AB) - k^{3}(e^{n},AA_{x}) - 3k^{3}(e^{n},BB_{x}) - 8k^{3}(A,A(A,B)) - 8k^{3}(B,B(A,B)).
\]
If $K3_{1}=k^{3}(\ve_{x}^{n},AB) - k^{3}(e^{n},AA_{x}) - 3k^{3}(e^{n},BB_{x})$ then
\[
K3_{1}\leq k^{3}\|\ve^{n}\|_{1}\|A\|_{\infty}\|B\| + k^{3}\|e^{n}\|(\|A\|_{\infty}\|A\|_{1} + 3\|B\|_{\infty}\|B\|_{1}),
\]
and, consequently, from $(\ref{equ530})_{2}$, the inverse properties of $S_{h}$, $S_{h,0}$ and $(\ref{equ530})_{1}$ it follows that
\[
K3_{1} \leq Ck(\|\ve^{n}\|^{2} + \|e^{n}\|^{2}),
\]
If $K3_{2}=- 8k^{3}(A,A(A,B)) - 8k^{3}(B,B(A,B))$ then, from \eqref{equ525}
\begin{align*}
K3_{2}&=2k^{3}(A_{x},U^{n}A) + 2k^{3}(A_{x},H^{n}B) + 4k^{3}(A_{x},B) - 2k^{3}(A_{x},AB)\\
&\quad +2k^{3}(B_{x},H^{n}A) + 6k^{3}(B_{x},U^{n}B) + 4k^{3}(B_{x},A) + 2k^{3}(B,AA_{x}),
\end{align*}
from which
\[
K3_{2}=-k^{3}(U_{x}^{n}A,A)-2k^{3}(H_{x}^{n}A,B)-3k^{3}(U_{x}^{n}B,B).
\]
Thus, from $(\ref{equ530})_{1}$
\[
K3_{2}\leq Ck(\|\ve^{n}\|^{2} + \|e^{n}\|^{2}).
\]
Hence, for $0\leq n\leq n^{*}$
\begin{equation}
K3 \leq C_{12}k(\|\ve^{n}\|^{2} + \|e^{n}\|^{2}).
\label{equ544}
\end{equation}
The $O(k^{4})$ terms of \eqref{equ540} are
\begin{align*}
K4&:=4k^{4}\|A(A,B)\|^{2} + \tfrac{k^{4}}{4}\|P(\ve^{n}B+e^{n}A)_{x}\|^{2} + \tfrac{k^{4}}{4}\|P(AB)_{x}\|^{2}
-2k^{4}(A,(AB)_{x})\\
&\quad -2k^{4}(A(A,B),(\ve^{n}B+e^{n}A)_{x})
+ 2k^{4}(A(A,B),(AB)_{x}) - \tfrac{k^{4}}{2}((\ve^{n}B+e^{n}A)_{x},P(AB)_{x})\\
&\quad + 4k^{4}\|B(A,B)\|^{2}
+ \tfrac{k^{4}}{4}\|P_{0}(\ve^{n}A)_{x}\|^{2} + \tfrac{9k^{4}}{4}\|P_{0}(e^{n}B)_{x}\|^{2}
+ \tfrac{k^{4}}{4}\|P_{0}(AA_{x})\|^{2} + \tfrac{9k^{4}}{4}\|P_{0}(BB_{x})\|^{2}\\
&\quad - 2k^{4}(B(A,B),(\ve^{n}A)_{x}) - 6k^{4}(B(A,B),(e^{n}B)_{x}) + 2k^{4}(B(A,B),AA_{x}) + 6k^{4}(B(A,B),BB_{x})\\
&\quad + \tfrac{3k^{4}}{2}((\ve^{n}A)_{x},P_{0}(e^{n}B)_{x})-\tfrac{k^{4}}{2}((\ve^{n}A)_{x},P_{0}(AA_{x}))
-\tfrac{3k^{4}}{2}((\ve^{n}A)_{x},P_{0}(BB_{x}))\\
&\quad  -\tfrac{3k^{4}}{2}((e^{n}B)_{x},P_{0}(AA_{x}))- \tfrac{9k^{4}}{2}((e^{n}B)_{x},P_{0}(BB_{x}))
+ \tfrac{3k^{4}}{2}(AA_{x},P_{0}(BB_{x})).
\end{align*}
From Lemma 5.7 and the inverse properties of $S_{h}$, $S_{h,0}$ we have
\begin{equation}
\|A(A,B)\| + \|B(A,B)\| \leq C'h^{-1}(\|A\| + \|B\|)\leq Ch^{-2}(\|\ve^{n}\| + \|e^{n}\|).
\label{equ545}
\end{equation}
Similarly it follows that
\begin{equation}
\begin{aligned}
\|(\ve^{n}B)_{x}\| &\leq \|\ve_{x}^{n}B\| + \|\ve^{n}B_{x}\|\leq \|\ve^{n}\|_{1}\|B\|_{\infty}+\|\ve^{n}\|_{\infty}\|B\|_{1}\\
& \leq Ch^{-1}(\|\ve^{n}\| + \|B\|)\leq Ch^{-2}(\|\ve^{n}\| + \|e^{n}\|),
\end{aligned}
\label{equ546}
\end{equation}
and
\begin{equation}
\|(e^{n}B)_{x}\| + \|(\ve^{n}A)_{x}\| + \|(e^{n}A)_{x}\|\leq Ch^{-2}(\|\ve^{n}\| + \|e^{n}\|).
\label{equ547}
\end{equation}
Finally
\begin{equation}
\begin{aligned}
\|(AB)_{x}\| & \leq \|A_{x}B\| + \|AB_{x}\| \leq \|A\|_{1}\|B\|_{\infty} + \|A\|_{\infty}\|B\|_{1}\\
&\leq Ch^{-1}(\|A\| + \|B\|)\leq Ch^{-2}(\|\ve^{n}\| + \|e^{n}\|),
\end{aligned}
\label{equ548}
\end{equation}
and
\begin{equation}
\|AA_{x}\| + \|BB_{x}\|\leq Ch^{-1}(\|A\| + \|B\|)\leq Ch^{-2}(\|\ve^{n}\| + \|e^{n}\|).
\label{equ549}
\end{equation}
Hence
\begin{equation}
K4 \leq C'k^{4}h^{-4}(\|\ve^{n}\|^{2} + \|e^{n}\|^{2}) \leq C'k(\|\ve^{n}\|^{2} + \|e^{n}\|^{2}),
\label{equ550}
\end{equation}
where $C'$ a constant that depends polynomially on $\mu=k/h^{4/3}$. For the $O(k^{5})$ terms of \eqref{equ540} we get
\begin{align*}
K5&:=2k^{5}(A(A,B),(AB)_{x}) - \tfrac{k^{5}}{2}((\ve^{n}B+e^{n}A)_{x},P(AB)_{x}) + \tfrac{k^{5}}{2}\|P(AB)_{x}\|^{2}\\
&\quad + 2k^{5}(B(A,B),AA_{x}) - \tfrac{k^{5}}{2}((\ve^{n}A)_{x},P_{0}(AA_{x})) - \tfrac{3k^{5}}{2}((e^{n}B)_{x},P_{0}(AA_{x}))\\
&\quad +6k^{5}(B(A,B),BB_{x}) - \tfrac{3k^{5}}{2}((\ve^{n}A)_{x},P_{0}(BB_{x}))-\tfrac{9k^{5}}{2}((e^{n}B)_{x},P_{0}(BB_{x}))\\
&\quad + 3k^{5}(AA_{x},P_{0}(BB_{x})) + \tfrac{k^{5}}{2}\|P_{0}(AA_{x})\|^{2} + \tfrac{9k^{5}}{2}\|P_{0}(BB_{x})\|^{2}.
\end{align*}
From \eqref{equ540}-\eqref{equ549} we obtain
\begin{equation}
K5 \leq Ck(\|\ve^{n}\|^{2} + \|e^{n}\|^{2}).
\label{equ551}
\end{equation}
The $O(k^{6})$ terms of \eqref{equ540} are:
\[
K6:=\tfrac{k^{6}}{4}\|P(AB)_{x}\|^{2} + \tfrac{k^{6}}{4}\|P_{0}(AA_{x})\|^{2} + \tfrac{9k^{6}}{4}\|P_{0}(BB_{x})\|^{2}
+ \tfrac{3k^{6}}{4}(AA_{x},P_{0}(BB_{x})).
\]
Thus, from \eqref{equ545}-\eqref{equ549} we obtain
\[
K6 \leq Ck^{2}(\|\ve^{n}\|^{2} + \|e^{n}\|^{2}).
\]
From this inequality and \eqref{equ543}, \eqref{equ544}, \eqref{equ550} and \eqref{equ551}, we have, in view of \eqref{equ540}
\[
\|\ve^{n}-2kA(\theta^{n},\xi^{n})\|^{2} + \|e^{n}-2kB(\theta^{n},\xi^{n})\|^{2} \leq (1+C'k)(\|\ve^{n}\|^{2} + \|e^{n}\|^{2}),
\]
where $C'$ depends polynomially on $\mu$. Hence, from \eqref{equ539} and the above it follows that for $0\leq n\leq n^{*}$
\[
\|\ve^{n+1}\|^{2} + \|e^{n+1}\|^{2} \leq (1+C'k)(\|\ve^{n}\|^{2} + \|e^{n}\|^{2}) + Ck(k^{2} + h^{r-1})^{2},
\]
where $C'$ a constant depending polynomially on $\mu$. Using Gronwall's lemma in the above and the inverse properties of $S_{h}$,
$S_{h,0}$, we conclude that, since $r\geq 3$, if $h$ is taken sufficiently small, $n^{*}$ is not maximal and the argument
may be repeated so that eventually $n^{*}$ can be taken equal to $M-1$. Therefore
\[
\|\ve^{n}\|^{2} + \|e^{n}\|^{2} \leq C'(k^{2} + h^{r-1})^{2},\quad 0\leq n\leq M,
\]
for some constant $C'$ depending continuously on $\mu$. The conclusion of the proposition follows.
\end{proof}
In Table \ref{tbl51} we present the results of computations ($L^{2}$ errors for $\eta$) of the numerical solution of the nonhomogeneous
versions of both systems \eqref{eqsw} and \eqref{eqssw} with exact solution given by $\eta(x,t)=\exp(2t)(\cos(\pi x) + x + 2)$,
$u(x,t)=\exp(xt)(\sin(\pi x) + x^{3}-x^{2})$, $(x,t)\in [0,1]\times(0,1]$. We use piecewise linear continuous functions for
discretizing on a uniform mesh in space with $N=1/h=400$ and the improved Euler scheme with $k=h/10$ and $k=h^{4/3}/10$.
For both systems it was seen that the computations with $k=h/10$ were unstable. The fully discrete method for both systems
appeared to be stable when we took $k=h^{4/3}/10$.
\scriptsize
\captionsetup[subtable]{labelformat=empty,position=top,margin=1pt,singlelinecheck=false}
\begin{table}[h]
\subfloat[Shallow Water errors]{
\begin{tabular}[h]{ | c | c || c | c | }\hline
$t^{n}$  &  $k=h/10$      & $t^{n}$  &  $k=h^{4/3}/10$ \\ \hline
$0.05$   &  $0.3096(-5)$  & $0.0500$ &  $0.3098(-5)$ \\ \hline
$0.1$    &  $0.3309(-5)$  & $0.1000$ &  $0.3302(-5)$ \\ \hline
$0.3$    &  $0.1045(-4)$  & $0.3001$ &  $0.1036(-4)$ \\ \hline
$0.5$    &  $0.1487(-3)$  & $0.5001$ &  $0.1176(-4)$ \\ \hline
$0.59$   &  $3.2147$      & $0.59$   &  $0.1221(-4)$ \\ \hline
$0.7$    &  $overflow$    & $0.7002$ &  $0.1358(-4)$ \\ \hline
$0.8$    &  $$            & $0.8002$ &  $0.1289(-4)$ \\ \hline
$0.9$    &  $$            & $0.9002$ &  $0.1107(-4)$ \\ \hline
$1.0$    &  $$            & $1.0000$ &  $0.1682(-4)$ \\ \hline
\end{tabular}
}
\qquad
\subfloat[Symmetric Shallow Water errors]{
\begin{tabular}[h]{ | c | c || c | c | }\hline
$t^{n}$ &  $k=h/10$      & $t^{n}$   &  $k=h^{4/3}/10$ \\ \hline
$0.05$  &  $0.3066(-5)$  & $0.0500$  &  $0.3068(-5)$ \\ \hline
$0.1$   &  $0.2479(-5)$  & $0.1000$  &  $0.2474(-5)$ \\ \hline
$0.3$   &  $0.1043(-2)$  & $0.3001$  &  $0.6657(-5)$ \\ \hline
$0.35$  &  $1.4689$      & $0.35$    &  $0.6729(-5)$ \\ \hline
$0.5$   &  $overflow$    & $0.5001$  &  $0.7390(-5)$ \\ \hline
$0.7$   &  $$            & $0.7002$  &  $0.9463(-5)$ \\ \hline
$0.8$   &  $$            & $0.8002$  &  $0.1002(-4)$ \\ \hline
$0.9$   &  $$            & $0.9001$  &  $0.1118(-4)$ \\ \hline
$1.0$   &  $$            & $1.0000$  &  $0.1297(-4)$ \\ \hline
\end{tabular}
}
\caption{$L^{2}$-errors $\|H_{h}^{n} - \eta(t^{n})\|$, for \eqref{eqsw} and \eqref{eqssw}, piecewise linear continuous
functions on a uniform mesh $h=1/400$ and improved Euler time stepping.}
\label{tbl51}
\end{table}
\normalsize
\subsection{The third-order Shu-Osher scheme}
We now examine a third-order accurate explicit Runge-Kutta scheme due to Shu and Osher, \cite{so}. Written in the
standard Butcher notation, it is a three-stage scheme corresponding to the tableau below.\newpage
\begin{table}[h]
  \begin{equation*}
  \begin{array}{c|c}
  \rule{0pt}{2,3ex} A & \tau\\ \hline
  \rule{0pt}{2.3ex} b^{T} &
  \end{array} \,=\,
    \begin{array}{c c c | c}
      \rule{0pt}{2,3ex}  0   & \quad 0   & \quad 0   & 0   \\
      \rule{0pt}{2,3ex}  1   & \quad 0   & \quad 0   & 1   \\
      \rule{0pt}{2,3ex}  1/4 & \quad 1/4 & \quad 0   & 1/2 \\ \hline
      \rule{0pt}{2,3ex}  1/6 & \quad 1/6 & \quad 2/3 &
    \end{array}
  \end{equation*}
\end{table} \normalsize
Due to the special structure of this tableau one may simplify the scheme and write it as a two-stage
method approximating the o.d.e. $y'=f(t,y)$ in the form
\begin{align*}
y^{n,1} & = y^{n} + kf(t^{n},y^{n}),\\
y^{n,2} & = y^{n} + \tfrac{k}{4}f(t^{n},y^{n}) + \tfrac{k}{4}f(t^{n+1},y^{n,1}),\\
y^{n+1} & = y^{n} + \tfrac{k}{6}f(t^{n},y^{n}) + \tfrac{k}{6}f(t^{n+1},y^{n,1}) + \tfrac{2k}{3}f(t^{n}+k/2,y^{n,2});
\end{align*}
this is precisely the explicit scheme $(2.19)$ in \cite{so}. It is easy to check that the absolute stability interval of this scheme
on the imaginary axis is $[-\sqrt{3},\sqrt{3}]$; thus it is suitable for integrating in time semidiscretizations of e.g.
linear, first-order hyperbolic problems, such as the periodic initial-value problem for $u_{t} + u_{x} = 0$, under a Courant number
restriction. It is also well known, \cite{so}, that it has good nonlinear stability properties such as the TVD property, and has been
extensively used as a time-stepping scheme for the numerical approximation of hyperbolic systems in conservation law form.
In the rest of this section we will use it to discretize in time the semidiscrete \eqref{eqssw} initial-value problem
\eqref{eq27}-\eqref{eq28}.\par
Using the notation of Section 2, we let $S_{h}=S_{h}^{k,r}$, $S_{h,0}=S_{h,0}^{k,r}$, for $r\geq 3$. We put again
$k=T/M$, $t^{n}=nk$, $0\leq n\leq M$. We let as in Lemma 5.1 $H(t)=P\eta(t)$, $U(t)=P_{0}u(t)$, $H^{n}=H(t^{n})$,
$U^{n}=U(t^{n})$, and define
\begin{align}
\Phi & = U + \tfrac{1}{2}(HU)\,, \qquad \qquad \qquad \Phi^{n} = \Phi(t^{n}),
\label{equ552}\\
F & = H_{x} + \tfrac{1}{2}HH_{x} + \tfrac{3}{2}UU_{x}\,, \qquad F^{n}=F(t^{n}).
\label{equ553}
\end{align}
\normalsize
The Shu-Osher time-stepping scheme for \eqref{eq27}-\eqref{eq28} is the following: We seek $H_{h}^{n}\in S_{h}$,
$U_{h}^{n}\in S_{h,0}$ for $0\leq n\leq M$ and $H_{h}^{n,1}$, $H_{h}^{n,2}\in S_{h}$, $U_{h}^{n,1}$, $U_{h}^{n,2}\in S_{h,0}$
for $0\leq n\leq M-1$ such that for $0\leq n\leq M-1$,
\begin{equation}
\begin{aligned}
H_{h}^{n,1} - H_{h}^{n} + kP\Phi_{hx}^{n}& = 0,\\
U_{h}^{n,1} - U_{h}^{n} + kP_{0}F_{h}^{n} & = 0,\\
H_{h}^{n,2} - H_{h}^{n} + \tfrac{k}{4}P\Phi_{hx}^{n}+ \tfrac{k}{4}P\Phi_{hx}^{n,1} & = 0,\\
U_{h}^{n,2} - U_{h}^{n} + \tfrac{k}{4}P_{0}F_{h}^{n} + \tfrac{k}{4}P_{0}F_{h}^{n,1} & = 0,\\
H_{h}^{n+1} - H_{h}^{n} +\tfrac{k}{6}P\Phi_{hx}^{n} + \tfrac{k}{6}P\Phi_{hx}^{n,1} + \tfrac{2k}{3}P\Phi_{hx}^{n,2} & = 0,\\
U_{h}^{n+1} - U_{h}^{n} +\tfrac{k}{6}P_{0}F_{h}^{n} + \tfrac{k}{6}P_{0}F_{h}^{n,1} + \tfrac{2k}{3}P_{0}F_{h}^{n,2} & = 0,
\end{aligned}
\label{equ554}
\end{equation}
\normalsize
and $H_{h}^{0} = \eta_{h}(0)$, $U_{h}^{0} = u_{h}(0)$, where
\begin{align}
\Phi_{h}^{n} & = U_{h}^{n} + \tfrac{1}{2}(H_{h}^{n}U_{h}^{n}),
\label{equ555}\\
F_{h}^{n} & = H_{hx}^{n} + \tfrac{1}{2}H_{h}^{n}H_{hx}^{n} + \tfrac{3}{2}U_{h}^{n}U_{hx}^{n},
\label{equ556}
\end{align}
\normalsize and, for $j=1,2$,
\begin{align}
\Phi_{h}^{n,j} & = U_{h}^{n,j} + \tfrac{1}{2}(H_{h}^{n,j}U_{h}^{n,j}),
\label{equ557}\\
F_{h}^{n,j} & = H_{hx}^{n,j} + \tfrac{1}{2}H_{h}^{n,j}H_{hx}^{n,j} + \tfrac{3}{2}U_{h}^{n,j}U_{hx}^{n,j}.
\label{equ558}
\end{align} \normalsize
The intermediate stages $V^{n,j}\in S_{h}$, $W^{n,j}\in S_{h,0}$ for $j=1,2$ are defined, for $0\leq n\leq M-1$,
by the equations
\begin{align}
V^{n,1} - H^{n} + kP\Phi_{x}^{n} & = 0,
\label{equ559} \\
W^{n,1} - U^{n} + kP_{0}F^{n} & = 0,
\label{equ560}\\
V^{n,2} - H^{n} + \tfrac{k}{4}P\Phi_{x}^{n} + \tfrac{k}{4}P\Phi_{x}^{n,1} & = 0,
\label{equ561}\\
W^{n,2} - U^{n} + \tfrac{k}{4}P_{0}F^{n} + \tfrac{k}{4}P_{0}F^{n,1} & = 0,
\label{equ562}
\end{align}
where
\begin{align}
\Phi^{n,j} &=W^{n,j} + \tfrac{1}{2}(V^{n,j}W^{n,j}),
\label{equ563}\\
F^{n,j} & = V_{x}^{n,j} + \tfrac{1}{2}V^{n,j}V_{x}^{n,j} + \tfrac{3}{2}W^{n,j}W_{x}^{n,j},
\label{equ564}
\end{align}
for $j=1,2$.\\
We estimate first the continuous time truncation error using $L^{2}$ projections.
\begin{lemma} Let $(\eta,u)$ be the solution of  \eqref{eqssw} on $[0,T]$. If $H(t)=P\eta(t)$, $U(t)=P_{0}u(t)$,
and $\psi=\psi(t)\in S_{h}$ and $\zeta=\zeta(t)\in S_{h,0}$ are such that
\begin{align}
H_{t} + P\Phi_{x} = \psi,
\label{equ565}\\
U_{t} + P_{0}F = \zeta,
\label{equ566}
\end{align}
for $0\leq t\leq T$, then there exists a constant $C$ such that for $j=0,1,2$, it holds that
\[
\max_{0\leq t\leq T}(\|\partial_{t}^{j}\psi\| + \|\partial_{t}^{j}\zeta\|) \leq Ch^{r-1}\,.
\]
\end{lemma}
\begin{proof}
Subtracting both sides of the equations
\begin{align*}
P\eta_{t} + Pu_{x} + \tfrac{1}{2}P(\eta u)_{x} & = 0, \\
H_{t} + PU_{x} + \tfrac{1}{2}P(HU)_{x} &= \psi,
\end{align*}
and putting $\rho=\eta - H$, $\sigma=u-U$, we have
\[
P\sigma_{x} + \tfrac{1}{2}P(\eta u - HU)_{x}=-\psi.
\]
Since
\[
\eta u - HU = \eta\sigma + u\rho - \rho\sigma,
\]
it follows that
\[
P\sigma_{x} + \tfrac{1}{2}P(\eta\sigma)_{x} + \tfrac{1}{2}P(u\rho)_{x} - \tfrac{1}{2}P(\rho\sigma)_{x}=-\psi,
\]
and, as a consequence of the approximation properties of $S_{h}$ and $S_{h,0}$, for $j=0,1,2$
\[
\|\partial_{t}^{j}\psi\| \leq C h^{r-1}.
\]
Subtracting now both sides of the equations
\begin{align*}
P_{0}u_{t} + P_{0}\eta_{x} + \tfrac{1}{2}P_{0}(\eta\eta_{x}) + \tfrac{3}{2}P_{0}(uu_{x}) &= 0,\\
U_{t} + P_{0}H_{x} + \tfrac{1}{2}P_{0}(HH_{x}) + \tfrac{3}{2}P_{0}(UU_{x}) = \zeta,
\end{align*}
we obtain
\[
P_{0}\rho_{x} +\tfrac{1}{2}P_{0}(\eta\eta_{x} - HH_{x}) + \tfrac{3}{2}P_{0}(uu_{x}-UU_{x})=-\zeta.
\]
Since
\[
\eta\eta_{x} - HH_{x} = (\eta\rho)_{x} - \rho\rho_{x}\,, \qquad uu_{x}-UU_{x} = (u\sigma)_{x} - \sigma\sigma_{x},
\]
it follows that
\[
P_{0}\rho_{x} + \tfrac{1}{2}P_{0}((\eta\rho)_{x} - \rho\rho_{x})
+ \tfrac{3}{2}(P_{0}(u\sigma)_{x}-P_{0}(\sigma\sigma_{x})) = -\zeta,
\]
and, as in the $\psi$ case, we see that for $j=0,1,2$
\[
\|\partial_{t}^{j}\zeta\| \leq Ch^{r-1}.
\]
\end{proof}
We prove now consistency estimates for the scheme \eqref{equ554}.
\begin{lemma} Let $\lambda=k/h$. If
\begin{align}
\delta_{1}^{n} & = H^{n+1} - H^{n} +\tfrac{k}{6}P\Phi_{x}^{n} +\tfrac{k}{6}P\Phi_{x}^{n,1}+\tfrac{2k}{3}P\Phi_{x}^{n,2},
\label{equ567}\\
\delta_{2}^{n} & = U^{n+1} - U^{n} +\tfrac{k}{6}P_{0}F^{n} +\tfrac{k}{6}P_{0}F^{n,1}+\tfrac{2k}{3}P_{0}F^{n,2},
\label{equ568}
\end{align}
then there exists a constant $C_{\lambda}$, which is a polynomial of $\lambda$ with positive coefficients, such that
\[
\max_{0\leq n\leq M-1}(\|\delta_{1}^{n}\| + \|\delta_{2}^{n}\|)\leq C_{\lambda}k(k^{3} + h^{r-1}).
\]
\end{lemma}
\begin{proof}
From \eqref{equ559}, \eqref{equ565} and \eqref{equ560}, \eqref{equ565} we see that
\begin{align*}
V^{n,1} & = H^{n} + k H_{t}^{n} - k\psi^{n},\\
W^{n,1} & = U^{n} + kU_{t}^{n} - k\zeta^{n},
\end{align*}
and hence that
\[
V^{n,1}W^{n,1} = H^{n}U^{n} + k(HU)_{t}^{n} + k^{2}H_{t}^{n}U_{t}^{n} + v_{1}^{n}.
\]
Using the $L^{\infty}$ stability of the $L^{2}$ projection $P$, cf. \eqref{eq22a}, we obtain
\[
\|v_{1}^{n}\|\leq Ckh^{r-1}\,, \qquad \|v_{1}^{n}\|_{1} \leq C_{\lambda}h^{r-1}\,.
\]
Thus
\begin{equation}
\Phi^{n,1} = W^{n,1} + \tfrac{1}{2}(V^{n,1}W^{n,1})=\Phi^{n}+k\Phi_{t}^{n}+\tfrac{k^{2}}{2}H_{t}^{n}U_{t}^{n}+v_{2}^{n},
\label{equ569}
\end{equation}
with
\[
\|v_{2}^{n}\|\leq Ckh^{r-1}\,, \qquad \|v_{2}^{n}\|_{1}\leq C_{\lambda}h^{r-1}.
\]
Also
\[
V^{n,1}V_{x}^{n,1} = H^{n}H_{x}^{n} + k(HH_{x})_{t}^{n} + k^{2}H_{t}^{n}H_{tx}^{n}+\omega_{1}^{n},
\]
with
\[
\|\omega_{1}^{n}\| \leq C_{\lambda}h^{r-1}, \qquad  k\|\omega_{1}^{n}\|_{1}\leq Ck^{3} + C_{\lambda}h^{r-1}.
\]
In addition, since
\[
W^{n,1}W_{x}^{n,1} = U^{n}U_{x}^{n} + k(UU_{x})_{t}^{n} + k^{2}U_{t}^{n}U_{tx}^{n} + \omega_{2}^{n},
\]
with
\[
\|\omega_{2}^{n}\| \leq C_{\lambda}h^{r-1}, \qquad k\|\omega_{2}^{n}\|_{1}\leq Ck^{3} + C_{\lambda}h^{r-1},
\]
we will have
\begin{equation}
F^{n,1} = F^{n} + kF_{t}^{n} + \tfrac{k^{2}}{2}(H_{t}^{n}H_{tx}^{n} + 3U_{t}^{n}U_{tx}^{n}) + \omega_{3}^{n},
\label{equ570}
\end{equation}
with
\[
\|\omega_{3}^{n}\|\leq C_{\lambda}h^{r-1}, \qquad k\|\omega_{3}^{n}\|_{1} \leq Ck^{3} + C_{\lambda}h^{r-1}.
\]
Now
\begin{align*}
V^{n,2} & = H^{n} - \tfrac{k}{2}P\Phi_{x}^{n} - \tfrac{k^{2}}{2}P\Phi_{tx}^{n}-\tfrac{k^{3}}{8}P(H_{t}^{n}U_{t}^{n})_{x}
-\tfrac{k}{4}Pv_{2x}^{n}\\
& = H^{n} + \tfrac{k}{2}H_{t}^{n} - \tfrac{k}{2}\psi^{n} + \tfrac{k^{2}}{4}H_{tt}^{n}-\tfrac{k^{2}}{4}\psi_{t}^{n}
-\tfrac{k^{3}}{8}P(H_{t}^{n}U_{t}^{n})_{x} - \tfrac{k}{4}Pv_{2x}^{n},
\end{align*}
and finally
\begin{equation}
V^{n,2} = H^{n} + \tfrac{k}{2}H_{t}^{n} + \tfrac{k^{2}}{4}H_{tt}^{n} + \psi_{1}^{n},
\label{equ571}
\end{equation}
with
\[
\|\psi_{1}^{n}\| \leq Ck^{3} + C_{\lambda}h^{r-1}, \qquad  \|\psi_{1}^{n}\|_{1}\leq Ck^{3}+C_{\lambda}h^{r-1},
\]
where we used the stability of the $L^{2}$ projection in the $H^{1}$ norm, cf. \cite{twa}, and the inverse and approximation properties
of $S_{h}$.  Now
\begin{align*}
W^{n,2} & = U^{n} - \tfrac{k}{2}P_{0}F^{n} - \tfrac{k^{2}}{4}P_{0}F_{t}^{n}
-\tfrac{k^{3}}{8}P_{0}(H_{t}^{n}H_{tx}^{n}+3U_{t}^{n}U_{tx}^{n}) - \tfrac{k}{4}P_{0}\omega_{3}^{n}\\
& = U^{n} + \tfrac{k}{2}U_{t}^{n} - \tfrac{k}{2}\zeta^{n} + \tfrac{k^{2}}{4}U_{tt}^{n} - \tfrac{k^{2}}{4}\zeta_{t}^{n}
-\tfrac{k^{3}}{8}P_{0}(H_{t}^{n}H_{tx}^{n} + 3U_{t}^{n}U_{tx}^{n}) - \tfrac{k}{4}P_{0}\omega_{3}^{n},
\end{align*}
and therefore
\begin{equation}
W^{n,2} = U^{n} + \tfrac{k}{2}U_{t}^{n} + \tfrac{k^{2}}{4}U_{tt}^{n} + \zeta_{1}^{n},
\label{equ572}
\end{equation}
with
\[
\|\zeta_{1}^{n}\| \leq Ck^{3} + C_{\lambda}kh^{r-1}, \qquad \|\zeta_{1}^{n}\|_{1} \leq Ck^{3} + C_{\lambda}h^{r-1},
\]
where  we took into account the approximation and inverse properties of $S_{h,0}$, the stability of the $L^{2}$
projection in $H^{1}$, and the fact that $\eta_{x}\in H_{0}^{1}(0,1)$. It follows that
\[
V^{n,2}W^{n,2} = H^{n}U^{n} + \tfrac{k}{2}(HU)_{t}^{n} + \tfrac{k^{2}}{4}(HU)_{tt}^{n} - \tfrac{k^{2}}{4}H_{t}^{n}U_{t}^{n}
+ v_{3}^{n},
\]
with
\[
\|v_{3}^{n}\| \leq Ck^{3} + C_{\lambda}kh^{r-1}, \qquad \|v_{3}^{n}\|_{1} \leq Ck^{3} + C_{\lambda}h^{r-1}.
\]
Thus
\begin{equation}
\Phi^{n,2} = \Phi^{n} + \tfrac{k}{2} \Phi_{t}^{n} + \tfrac{k^{2}}{4}\Phi_{tt}^{n} - \tfrac{k^{2}}{8}H_{t}^{n}U_{t}^{n} + v_{4}^{n},
\label{equ573}
\end{equation}
with
\[
\|v_{4}^{n}\| \leq Ck^{3} + C_{\lambda}kh^{r-1}, \qquad \|v_{4}^{n}\|_{1} \leq Ck^{3} + C_{\lambda}h^{r-1}.
\]
From \eqref{equ569}, \eqref{equ573} we conclude that
\begin{align*}
\tfrac{1}{6}\Phi^{n} + \tfrac{1}{6}\Phi^{n,1} + \tfrac{2}{3}\Phi^{n,2} & =
\tfrac{1}{6}\Phi^{n} + \tfrac{1}{6}(\Phi^{n} + k\Phi_{t}^{n} + \tfrac{k^{2}}{2}H_{t}^{n}U_{t}^{n} + v_{2}^{n}) \\
\,\,\,& \,\,\,\,\,\,\, + \tfrac{2}{3}(\Phi^{n} + \tfrac{k}{2}\Phi_{t}^{n}
+ \tfrac{k^{2}}{4}\Phi_{tt}^{n} - \tfrac{k^{2}}{8}H_{t}^{n}U_{t}^{n} + v_{4}^{n}) \\
& = \Phi^{n} + \tfrac{k}{2}\Phi_{t}^{n} + \tfrac{k^{2}}{6}\Phi_{tt}^{n} + \tfrac{1}{6}v_{2}^{n} + \tfrac{2}{3}v_{4}^{n}.
\end{align*}
Hence
\begin{align*}
\delta_{1}^{n} & = H^{n+1} - H^{n} + kP\Phi_{x}^{n} + \tfrac{k^{2}}{2}P\Phi_{tx}^{n} + \tfrac{k^{3}}{6}P\Phi_{ttx}^{n}
+ \tfrac{k}{2}Pv_{2x}^{n} + \tfrac{2k}{3}Pv_{4x}^{n}\\
& = H^{n+1} - H^{n} - kH_{t}^{n} + k\psi^{n} - \tfrac{k^{2}}{2}H_{tt}^{n}
+ \tfrac{k^{2}}{2}\psi_{t}^{n} - \tfrac{k^{3}}{6}H_{ttt}^{n}
+\tfrac{k^{3}}{6}\psi_{tt}^{n} + \tfrac{k}{6}Pv_{2x}^{n} + \tfrac{2k}{3}Pv_{4x}^{n}\\
& = H^{n+1} - H^{n} - kH_{t}^{n} - \tfrac{k^{2}}{2}H_{tt}^{n} - \tfrac{k^{3}}{6}H_{ttt}^{n} + \alpha^{n},
\end{align*}
with
\[
\|\alpha^{n}\| \leq Ckh^{r-1} + C_{\lambda}kh^{r-1} + C_{\lambda}k(k^{3} + h^{r-1})\leq C_{\lambda}k(k^{3}+h^{r-1}).
\]
Therefore
\begin{equation}
\|\delta_{1}^{n}\| \leq C_{\lambda}k(k^{3} + h^{r-1}).
\label{equ574}
\end{equation}
From \eqref{equ571} we have
\[
V^{n,2}V_{x}^{n,2} = H^{n}H_{x}^{n} + \tfrac{k}{2}(HH_{x})_{t}^{n} + \tfrac{k^{2}}{4}(HH_{x})_{tt}^{n}
- \tfrac{k^{2}}{4}H_{t}^{n}H_{tx}^{n} + \omega_{4}^{n},
\]
with
\[
\|\omega_{4}^{n}\| \leq Ck^{3} + C_{\lambda}h^{r-1}.
\]
From \eqref{equ572} we obtain
\[
W^{n,2}W_{x}^{n,2} = U^{n}U_{x}^{n} + \tfrac{k}{2}(UU_{x})_{t}^{n} + \tfrac{k^{2}}{4}(UU_{x})_{tt}^{n}
- \tfrac{k^{2}}{4}U_{t}^{n}U_{tx}^{n} + \omega_{5}^{n},
\]
with
\[
\|\omega_{5}^{n}\|\leq Ck^{3} + C_{\lambda}h^{r-1}.
\]
Thus
\[
F^{n,2} = F^{n} + \tfrac{k}{2}F_{t}^{n} + \tfrac{k^{2}}{4}F_{tt}^{n} + \psi_{1x}^{n} - \tfrac{k^{2}}{8}H_{t}^{n}H_{tx}^{n}
 + \tfrac{1}{2}\omega_{4}^{n}
- \tfrac{3k^{2}}{8}U_{t}^{n}U_{tx}^{n} + \tfrac{3}{2}\omega_{5}^{n},
\]
i.e.
\begin{equation}
F^{n,2} = F^{n} + \tfrac{k}{2}F_{t}^{n} + \tfrac{k^{2}}{4}F_{tt}^{n} - \tfrac{k^{2}}{8}(H_{t}^{n}H_{tx}^{n}+3U_{t}^{n}U_{tx}^{n})
+ \omega_{6}^{n},
\label{equ575}
\end{equation}
where
\[
\|\omega_{6}^{n}\| \leq Ck^{3} + C_{\lambda}h^{r-1}.
\]
From \eqref{equ570}, \eqref{equ575} we now obtain
\begin{align*}
\tfrac{1}{6}F^{n} + \tfrac{1}{6}F^{n,1} + \tfrac{2}{3}F^{n,2} & = F^{n} + \tfrac{1}{6}(F^{n} + kF_{t}^{n}
+ \tfrac{k^{2}}{2}(H_{t}^{n}H_{tx}^{n} + 3U_{t}^{n}U_{tx}^{n}) + \omega_{3}^{n})\\
\,\,\,\,\,\,&\,\,\,\,\,\,\, + \tfrac{2}{3}(F^{n} + \tfrac{k}{2}F_{t}^{n}
+ \tfrac{k^{2}}{4}F_{tt}^{n} - \tfrac{k^{2}}{8}(H_{t}^{n}H_{tx}^{n}+3U_{t}^{n}U_{tx}^{n})
+ \omega_{6}^{n})\\
& = F^{n} + \tfrac{k}{2}F_{t}^{n} + \tfrac{k^{2}}{6}F_{tt}^{n} + \tfrac{1}{6}\omega_{3}^{n}+\tfrac{2}{3}\omega_{6}^{n},
\end{align*}
and therefore
\begin{align*}
\delta_{2}^{n} & = U^{n+1} - U^{n} + kP_{0}F^{n} + \tfrac{k^{2}}{2}P_{0}F_{t}^{n} + \tfrac{k^{3}}{6}P_{0}F_{tt}^{n}
+ \tfrac{k}{6}P_{0}\omega_{3}^{n} + \tfrac{2k}{3}P_{0}\omega_{6}^{n}\\
& = U^{n+1} - U^{n} - kU_{t}^{n} + k\zeta^{n} -\tfrac{k^{2}}{2}U_{tt}^{n} - \tfrac{k^{2}}{2}\zeta_{t}^{n}
-\tfrac{k^{3}}{6}U_{ttt}^{n} -\tfrac{k^{3}}{6}\zeta_{tt}^{n} + \tfrac{k}{6}P_{0}\omega_{3}^{n} + \tfrac{2k}{3}P_{0}\omega_{6}^{n}\\
& = U^{n+1} - U^{n} - kU_{t}^{n}  -\tfrac{k^{2}}{2}U_{tt}^{n} -\tfrac{k^{3}}{6}U_{ttt}^{n} + k\beta^{n},
\end{align*}
where
\[
\|\beta^{n}\| \leq Ckh^{r-1} + Ck^{3} + C_{\lambda}h^{r-1}\leq Ck^{3} + C_{\lambda}h^{r-1}.
\]
Hence
Επομένως
\[
\|\delta_{2}^{n}\| \leq C_{\lambda}k(k^{3} + h^{r-1}).
\]
From this estimate and from \eqref{equ574} the result of the Lemma follows.
\end{proof}
We now proceed with the proof of convergence of the scheme.
\begin{proposition} Let $(H_{h}^{n}$, $U_{h}^{n})$ be the solution of \eqref{equ554}. If $\lambda = k/h$ then there exists
a constant $\lambda_{0}$ and a constant $C$ independent of $k$, $h$ such that for $\lambda\leq \lambda_{0}$,
\[
\max_{0\leq n\leq M}(\|\eta(t^{n}) - H_{h}^{n}\| + \|u(t^{n}) - U_{h}^{n}\|) \leq C(k^{3} + h^{r-1}).
\]
\end{proposition}
\begin{proof}
It suffices to show that
\[
\max_{0\leq n\leq M}(\|H^{n} - H_{h}^{n}\| + \|U^{n} - U_{h}^{n}\|) \leq C(k^{3} + h^{r-1}).
\]
We let
\[
\ve^{n}=H^{n}-H_{h}^{n}, \,\,e^{n}=U^{n}-U_{h}^{n}, \,\,\ve^{n,j}=V^{n,j}-H_{h}^{n,j}, \,\,e^{n,j}=W^{n,j}-U_{h}^{n,j},
\,\,j=1,2.
\]
Then from \eqref{equ554}, \eqref{equ559}-\eqref{equ562} if follows that
\begin{align}
\ve^{n,1} & = \ve^{n} - kP(\Phi^{n}-\Phi_{h}^{n})_{x},
\label{equ576}\\
e^{n,1} & = e^{n} - kP_{0}(F^{n}-F_{h}^{n}),
\label{equ577} \\
\ve^{n,2} & = \ve^{n} - \tfrac{k}{4}P(\Phi^{n}-\Phi_{h}^{n})_{x} - \tfrac{k}{4}P(\Phi^{n,1}-\Phi_{h}^{n,1})_{x},
\label{equ578} \\
e^{n,2} & = e^{n} - \tfrac{k}{4}P_{0}(F^{n}-F_{h}^{n}) - \tfrac{k}{4}P_{0}(F^{n,1}-F_{h}^{n,1}),
\label{equ579}
\end{align}
so that from the two last equations of \eqref{equ554} and also from \eqref{equ567}, \eqref{equ568} we have
\begin{align}
\ve^{n+1} &=\ve^{n} - \tfrac{k}{6}P(\Phi^{n}-\Phi_{h}^{n})_{x}-\tfrac{k}{6}P(\Phi^{n,1}-\Phi_{h}^{n,1})_{x}
-\tfrac{2k}{3}P(\Phi^{n,2}-\Phi_{h}^{n,2})_{x} + \delta_{1}^{n},
\label{equ580}\\
e^{n+1} &=e^{n} - \tfrac{k}{6}P_{0}(F^{n}-F_{h}^{n})-\tfrac{k}{6}P_{0}(F^{n,1}-F_{h}^{n,1})
-\tfrac{2k}{3}P_{0}(F^{n,2}-F_{h}^{n,2}) + \delta_{2}^{n}.
\label{equ581}
\end{align}
From \eqref{equ552}, \eqref{equ555} it follows that
\[
\Phi^{n} - \Phi_{h}^{n} = U^{n} + \tfrac{1}{2}(H^{n}U^{n}) - U_{h}^{n}-\tfrac{1}{2}(H_{h}^{n}U_{h}^{n})=
e^{n} + \tfrac{1}{2}(H^{n}U^{n} - H_{h}^{n}U_{h}^{n}),
\]
and since
\begin{align*}
H^{n}U^{n} - H_{h}^{n}U_{h}^{n}& = H^{n}(U^{n}-U_{h}^{n}) + U^{n}(H^{n}-H_{h}^{n} - (H^{n}-H_{h}^{n})(U^{n}-U_{h}^{n})\\
& = H^{n}e^{n} + U^{n}\ve^{n} - \ve^{n}e^{n},
\end{align*}
we see that
\[
\Phi^{n} - \Phi_{h}^{n} = e^{n} + \tfrac{1}{2}H^{n}e^{n} + \tfrac{1}{2}U^{n}\ve^{n} - \tfrac{1}{2}\ve^{n}e^{n},
\]
or
\begin{equation}
\Phi^{n} - \Phi_{h}^{n} = \rho^{n} + \rho_{1}^{n},
\label{equ582}
\end{equation}
where
\begin{equation}
\rho^{n} = e^{n} + \tfrac{1}{2}H^{n}e^{n} + \tfrac{1}{2}U^{n}\ve^{n},
\label{equ583}
\end{equation}
and
\[
\rho_{1}^{n} = - \tfrac{1}{2}\ve^{n}e^{n}.
\]
Hence
\begin{align}
\|\rho^{n}\|\leq C(\|\ve^{n}\| + \|e^{n}\|),
\label{equ584} \\
\|\rho_{x}^{n}\|\leq \tfrac{C}{h}(\|\ve^{n}\| + \|e^{n}\|).
\label{equ585}
\end{align}
Let now $0\leq n^{*}\leq M-1$ be the maximal integer for which
\[
\|\ve^{n}\|_{1,\infty} + \|e^{n}\|_{1,\infty} \leq 1\,, \qquad 0\leq n\leq n^{*}.
\]
Then, for $0\leq n\leq n^{*}$,
\begin{align}
\|\rho_{x}^{n}\|_{\infty} & \leq C\,, \qquad \|\rho_{1x}^{n}\|_{\infty}\leq C,
\label{equ586} \\
\|\rho_{1x}^{n}\| & \leq C(\|\ve^{n}\| + \|e^{n}\|).
\label{equ587}
\end{align}
Now, from \eqref{equ586} and \eqref{equ582} we have
\[
\ve^{n,1} = \ve^{n} -kP\rho_{x}^{n} - kP\rho_{1x}^{n},
\]
and therefore, for $0\leq n\leq n^{*}$, we have
\begin{align}
\|\ve^{n,1}\| & \leq C_{\lambda}(\|\ve^{n}\| + \|e^{n}\|),
\label{equ588} \\
\|\ve^{n,1}\|_{1,\infty} & \leq C_{\lambda},
\label{equ589}
\end{align}
where we used the inverse properties of $S_{h}$ and the stability of $P$ in $L^{\infty}$ norm. Since now
\[
F^{n} - F_{h}^{n}=\ve_{x}^{n} + \tfrac{1}{2}(H^{n}H_{x}^{n}-H_{h}^{n}H_{hx}^{n})+\tfrac{3}{2}(U^{n}U_{x}^{n}-U_{h}^{n}U_{hx}^{n}),
\]
and
\[
H^{n}H_{x}^{n}-H_{h}^{n}H_{hx}^{n}=(H^{n}\ve^{n})_{x}-\ve^{n}\ve_{x}^{n}\,,\qquad
U^{n}U_{x}^{n}-U_{h}^{n}U_{hx}^{n}=(U^{n}e^{n})_{x}-e^{n}e_{x}^{n},
\]
we will have
\[
F^{n} - F_{h}^{n}=\ve_{x}^{n} + \tfrac{1}{2}(H^{n}\ve^{n})_{x} + \tfrac{3}{2}(U^{n}e^{n})_{x}
- \tfrac{1}{2}\ve^{n}\ve_{x}^{n} - \tfrac{3}{2}e^{n}e_{x}^{n},
\]
or
\begin{equation}
F^{n} - F_{h}^{n}= r_{x}^{n} + r_{1}^{n},
\label{equ590}
\end{equation}
where
\begin{equation}
r^{n} = \ve^{n} + \tfrac{1}{2}(H^{n}\ve^{n}) + \tfrac{3}{2}(U^{n}e^{n}),
\label{equ591}
\end{equation}
\[
r_{1}^{n} = - \tfrac{1}{2}\ve^{n}\ve_{x}^{n} - \tfrac{3}{2}e^{n}e_{x}^{n},
\]
with
\begin{align}
\|r^{n}\| & \leq C(\|\ve^{n}\| + \|e^{n}\|),
\label{equ592}\\
 \|r_{x}^{n}\| & \leq \tfrac{C}{h}(\|\ve^{n}\| + \|e^{n}\|),
\label{equ593}
\end{align}
and, for $0\leq n\leq n^{*}$,
\begin{align}
\|r_{x}^{n}\|_{\infty} & \leq C\,, \qquad \|r_{1}^{n}\|_{\infty} \leq C,
\label{equ594}\\
\|r_{1}^{n}\| & \leq C(\|\ve^{n}\| + \|e^{n}\|).
\label{equ595}
\end{align}
From \eqref{equ577}, \eqref{equ590} it follows that
\[
e^{n,1} = e^{n} - k P_{0}r_{x}^{n} - kP_{0}r_{1}^{n},
\]
and therefore, for $0\leq n\leq n^{*}$,
\begin{align}
\|e^{n,1}\| & \leq C_{\lambda}(\|\ve^{n}\| + \|e^{n}\|),
\label{equ596}\\
\|e^{n,1}\|_{1,\infty} & \leq C_{\lambda}.
\label{equ597}
\end{align}
Now, from \eqref{equ563}, \eqref{equ557} (for $j=1$) we obtain
\[
\Phi^{n,1} - \Phi_{h}^{n,1} = e^{n,1} + \tfrac{1}{2}V^{n,1}e^{n,1} + \tfrac{1}{2}W^{n,1}\ve^{n,1}
- \tfrac{1}{2}\ve^{n,1}e^{n,1},
\]
and
\begin{align*}
V^{n,1}e^{n,1} & = (H^{n} - kP\Phi_{x}^{n})(e^{n}-kP_{0}r_{x}^{n}-kP_{0}r_{1}^{n})\\
& = H^{n}e^{n}-kH^{n}P_{0}r_{x}^{n} - k H^{n}P_{0}r_{1}^{n} - kP\Phi_{x}^{n}\cdot e^{n,1}\,,\\
W^{n,1}\ve^{n,1} & = (U^{n} - kP_{0}F^{n})(\ve^{n}-kP\rho_{x}^{n}-kP\rho_{1x}^{n})\\
& = U^{n}\ve^{n} - kU^{n}P\rho_{x}^{n} - kU^{n}P\rho_{1x}^{n}-kP_{0}F^{n}\cdot\ve^{n,1}.
\end{align*}
Thus,
\begin{align*}
\Phi^{n,1} - \Phi_{h}^{n,1} & = (e^{n}+\tfrac{1}{2}H^{n}e^{n} + \tfrac{1}{2}U^{n}\ve^{n}) -
k(P_{0}r_{x}^{n} +\tfrac{1}{2}H^{n}P_{0}r_{x}^{n} + \tfrac{1}{2}U^{n}P\rho_{x}^{n}) \\
\,\,\,\,\,& \,\,\, - kP_{0}r_{1}^{n}-\tfrac{k}{2}(U^{n}P\rho_{1x}^{n}+P_{0}F^{n}\cdot\ve^{n,1}+H^{n}P_{0}r_{1}^{n}
+ P\Phi_{x}^{n}\cdot e^{n,1}) - \tfrac{1}{2}\ve^{n,1}e^{n,1},
\end{align*}
or
\begin{equation}
\Phi^{n,1} - \Phi_{h}^{n,1} = \rho^{n} - kf^{n} + \rho_{1}^{n,1},
\label{equ598}
\end{equation}
where
\begin{equation}
f^{n} = P_{0}r_{x}^{n} +\tfrac{1}{2}H^{n}P_{0}r_{x}^{n} + \tfrac{1}{2}U^{n}P\rho_{x}^{n},
\label{equ599}
\end{equation}
and
\[
\rho_{1}^{n,1} = - kP_{0}r_{1}^{n}-\tfrac{k}{2}(U^{n}P\rho_{1x}^{n}+P_{0}F^{n}\cdot\ve^{n,1}+H^{n}P_{0}r_{1}^{n}
+ P\Phi_{x}^{n}\cdot e^{n,1}) - \tfrac{1}{2}\ve^{n,1}e^{n,1}.
\]
From the inverse properties of $S_{h}$, $S_{h,0}$ the estimates \eqref{equ553}, \eqref{equ587}, the stability of the
$L^{2}$ projection in the $L^{\infty}$ norm, the fact that $\|F^{n}\|_{\infty}\leq C$, and that $\|\Phi_{x}^{n}\|\leq C$,
and the inequalities \eqref{equ596}, \eqref{equ589}, \eqref{equ597} and \eqref{equ594}, \eqref{equ586} we obtain
\begin{align}
\|\rho_{1x}^{n,1}\| & \leq C_{\lambda}(\|\ve^{n}\| + \|e^{n}\|),
\label{equ5100}\\
\|\rho_{1x}^{n,1}\|_{\infty} & \leq C_{\lambda}.
\label{equ5101}
\end{align}
In addition, from the inverse properties of $S_{h}$, $S_{h,0}$ and \eqref{equ586}, \eqref{equ585}, \eqref{equ594} we see that
\begin{align}
\|f_{x}^{n}\| & \leq \tfrac{C}{h^{2}}(\|\ve^{n}\| + \|e^{n}\|),
\label{equ5102}\\
\|f_{x}^{n}\|_{\infty} & \leq \tfrac{C}{h}.
\label{equ5103}
\end{align}
Now, from \eqref{equ578}, \eqref{equ582}, \eqref{equ598} if follows that
\begin{align*}
\ve^{n,2} & = \ve^{n} - \tfrac{k}{4}[P(\Phi^{n} - \Phi_{h}^{n})_{x} + P(\Phi^{n,1} - \Phi_{h}^{n,1})_{x}]\\
& = \ve^{n} - \tfrac{k}{4}[2P\rho_{x}^{n} - kPf_{x}^{n} + P\rho_{1x}^{n} + P\rho_{1x}^{n,1}],
\end{align*}
so that
\[
\ve^{n,2} = \ve^{n} - \tfrac{k}{2}P\rho_{x}^{n} + \tfrac{k^{2}}{4}Pf_{x}^{n} - \tfrac{k}{4}(P\rho_{1x}^{n}+P\rho_{1x}^{n,1}).
\]
From \eqref{equ585}, \eqref{equ5102}, \eqref{equ587}, \eqref{equ5100}, \eqref{equ586}, \eqref{equ5103} \eqref{equ5101}, and the
inverse properties of $S_{h}$, we obtain now, for $0\leq n\leq n^{*}$,
\begin{align}
\|\ve^{n,2}\| & \leq C_{\lambda}(\|\ve^{n}\| + \|e^{n}\|),
\label{equ5104}\\
\|\ve^{n,2}\|_{1,\infty} & \leq C_{\lambda}.
\label{equ5105}
\end{align}
We also have
\[
F^{n,1} - F_{h}^{n,1} = \ve_{x}^{n,1} + \tfrac{1}{2}(V^{n,1}\ve^{n,1})_{x} + \tfrac{3}{2}(W^{n,1}e^{n,1})_{x}
-\tfrac{1}{2}\ve^{n,1}\ve_{x}^{n,1} - \tfrac{3}{2}e^{n,1}e_{x}^{n,1},
\]
and
\begin{align*}
V^{n,1}\ve^{n,1} & = (H^{n} - kP\Phi_{x}^{n})(\ve^{n} - kP\rho_{x}^{n} - kP\rho_{1x}^{n})\\
& = H^{n}\ve^{n} - kH^{n}P\rho_{x}^{n} - kH^{n}P\rho_{1x}^{n} - kP\Phi_{x}^{n}\cdot \ve^{n,1},
\end{align*}
and also
\begin{align*}
W^{n,1}e^{n,1} & = (U^{n}-kP_{0}F^{n})(e^{n} - kP_{0}r_{x}^{n}-kP_{0}r_{1}^{n})\\
& = U^{n}e^{n} - kU^{n}P_{0}r_{x}^{n} - kU^{n}P_{0}r_{1}^{n} - kP_{0}F^{n}\cdot e^{n,1}.
\end{align*}
Hence,
\begin{align*}
F^{n,1} - F_{h}^{n,1} & = (\ve^{n} +\tfrac{1}{2}H^{n}\ve^{n} + \tfrac{3}{2}U^{n}e^{n})_{x}
- k(P\rho_{x}^{n} + \tfrac{1}{2}H^{n}P\rho_{x}^{n} + \tfrac{3}{2}U^{n}P_{0}r_{x}^{n})_{x} \\
\,\,\, & \,\,\,\,\,\,\, - k(P\rho_{1x}^{n} + \tfrac{1}{2}H^{n}P\rho_{1x}^{n} + \tfrac{1}{2}P\Phi_{x}^{n}\cdot\ve^{n,1}
+ \tfrac{3}{2}U^{n}P_{0}r_{1}^{n} + \tfrac{3}{2}P_{0}F^{n}\cdot e^{n,1})_{x} \\
\,\,\, & \,\,\,\,\,\,\, - \tfrac{1}{2}\ve^{n,1}\ve_{x}^{n,1} - \tfrac{3}{2}e^{n,1}e_{x}^{n,1},
\end{align*}
i.e.
\begin{equation}
F^{n,1} - F_{h}^{n,1} = r_{x}^{n} - kg_{x}^{n} + r_{1}^{n,1},
\label{equ5106}
\end{equation}
where
\begin{equation}
g^{n} = P\rho_{x}^{n} + \tfrac{1}{2}H^{n}P\rho_{x}^{n} + \tfrac{3}{2}U^{n}P_{0}r_{x}^{n},
\label{equ5107}
\end{equation}
and
\begin{align*}
r_{1}^{n,1} & = - k(P\rho_{1x}^{n} + \tfrac{1}{2}H^{n}P\rho_{1x}^{n} + \tfrac{1}{2}P\Phi_{x}^{n}\cdot\ve^{n,1}
 + \tfrac{3}{2}U^{n}P_{0}r_{1}^{n} + \tfrac{3}{2}P_{0}F^{n}\cdot e^{n,1})_{x} \\
\,\,\, & \,\,\,\,\,\,	- \tfrac{1}{2}\ve^{n,1}\ve_{x}^{n,1} -\tfrac{3}{2}e^{n,1}e_{x}^{n,1}.
\end{align*}
From the inverse property of $S_{h}$, $S_{h,0}$ \eqref{equ589}, \eqref{equ597}, \eqref{equ586}, \eqref{equ588}, \eqref{equ595},
\eqref{equ596}, \eqref{equ594}, we obtain, for $0\leq n\leq n^{*}$,
\begin{align}
\|r_{1}^{n,1}\| & \leq C_{\lambda}(\|\ve^{n}\| + \|e^{n}\|),
\label{equ5108} \\
\|r_{1}^{n,1}\|_{\infty} & \leq C_{\lambda}.
\label{equ5109}
\end{align}
In addition, from the inverse property of $S_{h}$, $S_{h,0}$ \eqref{equ585}, \eqref{equ593} and \eqref{equ586},
\eqref{equ594}, we obtain, for $0\leq n\leq n^{*}$,
\begin{align}
\|g_{x}^{n}\| & \leq \tfrac{C}{h^{2}}(\|\ve^{n}\| + \|e^{n}\|),
\label{equ5110} \\
\|g_{x}^{n}\|_{\infty} & \leq \tfrac{C}{h}.
\label{equ5111}
\end{align}
Now from \eqref{equ579}, \eqref{equ590}, \eqref{equ5106}, it follows that
\begin{align*}
e^{n,2} & = e^{n} - \tfrac{k}{4}[P_{0}(F^{n}-F_{h}^{n}) + P_{0}(F^{n,1}-F_{h}^{n,1})]\\
& = e^{n} - \tfrac{k}{4}[2P_{0}r_{x}^{n} - kP_{0}g_{x}^{n} + P_{0}r_{1}^{n} + P_{0}r_{1}^{n,1}],
\end{align*}
i.e.
\[
e^{n,2} = e^{n} - \tfrac{k}{2}P_{0}r_{x}^{n} + \tfrac{k^{2}}{4}P_{0}g_{x}^{n} - \tfrac{k}{4}(P_{0}r_{1}^{n}+P_{0}r_{1}^{n,1}).
\]
Hence, from \eqref{equ593}, \eqref{equ5107}, \eqref{equ595}, \eqref{equ5108}, the inverse properties of $S_{h,0}$, and \eqref{equ594},
\eqref{equ5111}, \eqref{equ5109}, we obtain, for $0\leq n\leq n^{*}$,
\begin{align}
\|e^{n,2}\| & \leq C_{\lambda}(\|\ve^{n}\| + \|e^{n}\|),
\label{equ5112} \\
\|e^{n,2}\|_{1,\infty} & \leq C_{\lambda}.
\label{equ5113}
\end{align}
In order to derive expressions for $\Phi^{n}-\Phi_{h}^{n}$, $F^{n}-F_{h}^{n}$, we note from \eqref{equ561}, \eqref{equ562}
that
\begin{align*}
H^{n} - V^{n,2} & = \tfrac{k}{4}(P\Phi_{x}^{n} + P\Phi_{x}^{n,1}), \\
U^{n} - W^{n,2} & = \tfrac{k}{4}(P_{0}F^{n} + P_{0}F^{n,1}),
\end{align*}
and therefore
\begin{align}
\|H^{n} - V^{n,2}\|_{\infty} & \leq C_{\lambda}k\,, \qquad \|H^{n}-V^{n,2}\|_{1,\infty} \leq C_{\lambda},
\label{equ5114} \\
\|U^{n} - W^{n,2}\|_{\infty} & \leq C_{\lambda}k\,, \qquad \|U^{n} - W^{n,2}\|_{1,\infty} \leq C_{\lambda}.
\label{equ5115}
\end{align}
We now have
\[
\Phi^{n,2} - \Phi_{h}^{n,2} = e^{n,2} + \tfrac{1}{2}V^{n,2}e^{n,2} + \tfrac{1}{2}W^{n,2}\ve^{n,2}-\tfrac{1}{2}\ve^{n,2}e^{n,2},
\]
and
\begin{align*}
V^{n,2}e^{n,2} & = [H^{n} - (H^{n}-V^{n,2})][e^{n}-\tfrac{k}{2}P_{0}r_{x}^{n}+\tfrac{k^{2}}{4}P_{0}g_{x}^{n}
-\tfrac{k}{4}(P_{0}r_{1}^{n} + P_{0}r_{1}^{n,1})]\\
& = H^{n}e^{n} - \tfrac{k}{2}H^{n}P_{0}r_{x}^{n}+\tfrac{k^{2}}{4}H^{n}P_{0}g_{x}^{n}-\tfrac{k}{4}H^{n}P_{0}(r_{1}^{n}+r_{1}^{n,1})
-(H^{n}-V^{n,2})e^{n,2},
\end{align*}
\begin{align*}
W^{n,2}\ve^{n,2} & = [U^{n}-(U^{n}-W^{n,2})][\ve^{n}-\tfrac{k}{2}P\rho_{x}^{n}+\tfrac{k^{2}}{4}Pf_{x}^{n}
-\tfrac{k}{4}(P\rho_{1x}^{n} + P\rho_{1x}^{n,1})]\\
& = U^{n}\ve^{n} - \tfrac{k}{2}U^{n}P\rho_{x}^{n} + \tfrac{k^{2}}{4}U^{n}Pf_{x}^{n}-\tfrac{k}{4}U^{n}P(\rho_{1x}^{n}+\rho_{1x}^{n,1})
-(U^{n}-W^{n,2})\ve^{n,2}.
\end{align*}
Hence, according to \eqref{equ583}, \eqref{equ599}
\begin{align*}
\Phi^{n,2} - \Phi_{h}^{n,2} & = \rho^{n} - \tfrac{k}{2}f^{n} + \tfrac{k^{2}}{4}[P_{0}g_{x}^{n} + \tfrac{1}{2}H^{n}P_{0}g_{x}^{n}
+ \tfrac{1}{2}U^{n}Pf_{x}^{n}] \\
\,\,\,&\,\,\,\,\,\, -\tfrac{k}{4}[P_{0}(r_{1}^{n} + r_{1}^{n,1}) + \tfrac{1}{2}H^{n}P_{0}(r_{1}^{n}+r_{1}^{n,1})
+ \tfrac{1}{2}U^{n}P(\rho_{1x}^{n} + \rho_{1x}^{n,1})]\\
\,\,\,\,\,\,\,\,&\,\,\,\,\,\,\,\,\, - \tfrac{1}{2}(H^{n}-V^{n,2})e^{n,2}-\tfrac{1}{2}(U^{n}-W^{n,2})\ve^{n,2}
-\tfrac{1}{2}\ve^{n,2}e^{n,2},
\end{align*}
i.e.
\begin{equation}
\Phi^{n,2} - \Phi_{h}^{n,2} = \rho^{n} - \tfrac{k}{2}f^{n} + \tfrac{k^{2}}{4}[P_{0}g_{x}^{n} + \tfrac{1}{2}H^{n}P_{0}g_{x}^{n}
+ \tfrac{1}{2}U^{n}Pf_{x}^{n}] + \rho_{1}^{n,2},
\label{equ5116}
\end{equation}
where
\begin{align*}
\rho_{1}^{n,2} & = -\tfrac{k}{4}[P_{0}(r_{1}^{n} + r_{1}^{n,1}) + \tfrac{1}{2}H^{n}P_{0}(r_{1}^{n}+r_{1}^{n,1})
+ \tfrac{1}{2}U^{n}P(\rho_{1x}^{n} + \rho_{1x}^{n,1})]\\
\,\,\,&\,\,\,\,\,\,\,\,\,\,- \tfrac{1}{2}(H^{n}-V^{n,2})e^{n,2}-\tfrac{1}{2}(U^{n}-W^{n,2})\ve^{n,2}
-\tfrac{1}{2}\ve^{n,2}e^{n,2}.
\end{align*}
From the inverse properties of $S_{h}$, $S_{h,0}$ and \eqref{equ595}, \eqref{equ5108}, \eqref{equ587}, \eqref{equ5100},
\eqref{equ5114}, \eqref{equ5112}, \eqref{equ5115}, \eqref{equ5104}, \eqref{equ5105}, \eqref{equ5113}, we obtain for
$0\leq n\leq n^{*}$,
\begin{equation}
\|\rho_{1x}^{n,2}\| \leq C_{\lambda}(\|\ve^{n}\| + \|e^{n}\|).
\label{equ5117}
\end{equation}
From \eqref{equ580} and \eqref{equ582}, \eqref{equ598}, \eqref{equ5116} we see that
\begin{equation}
\ve^{n+1} = \ve^{n} - kP\rho_{x}^{n} + \tfrac{k^{2}}{2}Pf_{x}^{n}
-\tfrac{k^{3}}{6}P[P_{0}g_{x}^{n} + \tfrac{1}{2}H^{n}P_{0}g_{x}^{n} + \tfrac{1}{2}U^{n}Pf_{x}^{n}]_{x}
+ \tfrac{k}{6}\omega^{n} + \delta_{1}^{n},
\label{equ5118}
\end{equation}
where
\[
\omega^{n} = - P(\rho_{1x}^{n} + \rho_{1x}^{n,1} + 4\rho_{1x}^{n,2}),
\]
for which it holds that
\begin{equation}
\|\omega^{n}\| \leq C_{\lambda}(\|\ve^{n}\| + \|e^{n}\|),
\label{equ5119}
\end{equation}
for $0\leq n\leq n^{*}$, by \eqref{equ587}, \eqref{equ5100}, \eqref{equ5117}. Also,
\[
F^{n,2} - F_{h}^{n,2} = \ve_{x}^{n,2} + \tfrac{1}{2}(V^{n,2}\ve^{n,2})_{x} + \tfrac{3}{2}(W^{n,2}e^{n,2})_{x}
-\tfrac{1}{2}\ve^{n,2}\ve_{x}^{n,2} - \tfrac{3}{2}e^{n,2}e_{x}^{n,2},
\]
and
\begin{align*}
V^{n,2}\ve^{n,2} & = [H^{n} - (H^{n} - V^{n,2})][\ve^{n} - \tfrac{k}{2}P\rho_{x}^{n}+\tfrac{k^{2}}{4}Pf_{x}^{n}
-\tfrac{k}{4}(P\rho_{1x}^{n} + P\rho_{1x}^{n,1})]\\
& = H^{n}\ve^{n} - \tfrac{k}{2}H^{n}P\rho_{x}^{n} + \tfrac{k^{2}}{4}H^{n}Pf_{x}^{n} - \tfrac{k}{4}H^{n}P(\rho_{1x}^{n}+\rho_{1x}^{n})
-(H^{n}- V^{n,2})\ve^{n,2},
\end{align*}
\begin{align*}
W^{n,2}e^{n,2} & = [U^{n} - (U^{n} - W^{n,2})][e^{n} - \tfrac{k}{2}P_{0}r_{x}^{n} + \tfrac{k^{2}}{4}P_{0}g_{x}^{n}
-\tfrac{k}{4}(P_{0}r_{1}^{n} + P_{0}r_{1}^{n,1})]\\
& = U^{n}e^{n} - \tfrac{k}{2}U^{n}P_{0}r_{x}^{n} + \tfrac{k^{2}}{4}U^{n}P_{0}g_{x}^{n}-\tfrac{k}{4}U^{n}P_{0}(r_{1}^{n}+r_{1}^{n,1})
-(U^{n} - W^{n,2})e^{n,2}.
\end{align*}
Hence, from \eqref{equ591}, \eqref{equ5107} it follows that
\begin{align*}
F^{n,2} - F_{h}^{n,2} & = r_{x}^{n} - \tfrac{k}{2}g_{x}^{n} + \tfrac{k^{2}}{4}[Pf_{x}^{n} + \tfrac{1}{2}H^{n}Pf_{x}^{n}
+\tfrac{3}{2}U^{n}P_{0}g_{x}^{n}]_{x}\\
\,\,\,&\,\,\,\,\,\, -\tfrac{k}{4}[P(\rho_{1x}^{n}+\rho_{1x}^{n,1}) + \tfrac{1}{2}H^{n}P(\rho_{1x}^{n}+\rho_{1x}^{n,1})
+\tfrac{3}{2}U^{n}P_{0}(r_{1}^{n} + r_{1}^{n,1})]_{x}\\
\,\,\,&\,\,\,\,\,\, -\tfrac{1}{2}[(H^{n} - V^{n,2})\ve^{n,2}]_{x} - \tfrac{3}{2}[(U^{n}-W^{n,2})e^{n,2}]_{x}
-\tfrac{1}{2}\ve^{n,2}\ve_{x}^{n,2} - \tfrac{3}{2}e^{n,2}e_{x}^{n,2},
\end{align*}
or
\begin{equation}
F^{n,2} - F_{h}^{n,2} = r_{x}^{n} - \tfrac{k}{2}g_{x}^{n} + \tfrac{k^{2}}{4}[Pf_{x}^{n} + \tfrac{1}{2}H^{n}Pf_{x}^{n}
+\tfrac{3}{2}U^{n}P_{0}g_{x}^{n}]_{x} + r_{1}^{n,2},
\label{equ5120}
\end{equation}
where
\begin{align*}
r_{1}^{n,2} & = -\tfrac{k}{4}[P(\rho_{1x}^{n}+\rho_{1x}^{n,1}) + \tfrac{1}{2}H^{n}P(\rho_{1x}^{n}+\rho_{1x}^{n,1})
+\tfrac{3}{2}U^{n}P_{0}(r_{1}^{n} + r_{1}^{n,1})]_{x}\\
\,\,\,&\,\,\,\,\,\, -\tfrac{1}{2}[(H^{n} - V^{n,2})\ve^{n,2}]_{x} - \tfrac{3}{2}[(U^{n}-W^{n,2})e^{n,2}]_{x}
-\tfrac{1}{2}\ve^{n,2}\ve_{x}^{n,2} - \tfrac{3}{2}e^{n,2}e_{x}^{n,2}.
\end{align*}
From the inverse properties of $S_{h}$, $S_{h,0}$, \eqref{equ587}, \eqref{equ5100}, \eqref{equ595}, \eqref{equ5108}, \eqref{equ5104},
\eqref{equ5115}, \eqref{equ5104}, \eqref{equ5112}, \eqref{equ587}, we have, for $0\leq n\leq n^{*}$,
\begin{equation}
\|r_{1}^{n,2}\| \leq C_{\lambda}(\|\ve^{n}\| + \|e^{n}\|).
\label{equ5121}
\end{equation}
Also, from \eqref{equ581} and \eqref{equ590}, \eqref{equ5106}, \eqref{equ5120} we see that
\begin{equation}
e^{n+1} = e^{n} - kP_{0}r_{x}^{n} + \tfrac{k^{2}}{2}P_{0}g_{x}^{n}
- \tfrac{k^{3}}{6}P_{0}[Pf_{x}^{n} + \tfrac{1}{2}H^{n}Pf_{x}^{n}
+\tfrac{3}{2}U^{n}P_{0}g_{x}^{n}]_{x} + \tfrac{k}{6}w^{n} + \delta_{2}^{n},
\label{equ5122}
\end{equation}
where
\[
w^{n} = -P_{0}(r_{1}^{n} + r_{1}^{n,1} + 4 r_{1}^{n,2})
\]
satisfies
\begin{equation}
\|w^{n}\| \leq C_{\lambda}(\|\ve^{n}\| + \|e^{n}\|),
\label{equ5123}
\end{equation}
for $0\leq n\leq n^{*}$, in view of \eqref{equ595}, \eqref{equ5108}, \eqref{equ5121}. We now write \eqref{equ5118}, \eqref{equ5122}
in the form
\begin{align}
\ve^{n+1} & = \gamma^{n} + \tfrac{k}{6}\omega^{n} + \delta_{1}^{n},
\label{equ5124}\\
e^{n+1} & = \sigma^{n} + \tfrac{k}{6}w^{n} + \delta_{2}^{n},
\label{equ5125}
\end{align}
where
\begin{align*}
\gamma^{n} & = \ve^{n} - kP\rho_{x}^{n} + \tfrac{k^{2}}{2}Pf_{x}^{n}
-\tfrac{k^{3}}{6}P\widetilde{f}_{x}^{n}, \\
\sigma^{n} & = e^{n} - kP_{0}r_{x}^{n} + \tfrac{k^{2}}{2}P_{0}g_{x}^{n}
- \tfrac{k^{3}}{6}P_{0}\widetilde{g}_{x}^{n},
\end{align*}
and
\begin{align*}
\widetilde{f}^{n} & = P_{0}g_{x}^{n} + \tfrac{1}{2}H^{n}P_{0}g_{x}^{n} + \tfrac{1}{2}U^{n}Pf_{x}^{n},\\
\widetilde{g}^{n} & = Pf_{x}^{n} + \tfrac{1}{2}H^{n}Pf_{x}^{n} + \tfrac{3}{2}U^{n}P_{0}g_{x}^{n}.
\end{align*}
Taking squares of norms we see that
\begin{equation}
\|\gamma^{n}\|^{2} + \|\sigma^{n}\|^{2} = \|\ve^{n}\|^{2} + \|e^{n}\|^{2} + ka_{1}^{n} + k^{2}a_{2}^{n}
+k^{3}a_{3}^{n} + k^{4}a_{4}^{n} + k^{5}a_{5}^{n} + k^{6}a_{6}^{n},
\label{equ5126}
\end{equation}
where $a_{j}^{n}$, $j=1,2,\dots,6$, are quantities with no explicit dependence on $k$, which we will calculate and estimate below.
For $a_{1}^{n}$ we obtain
\[
a_{1}^{n} = -2(\ve^{n},P\rho_{x}^{n}) - 2(e^{n},P_{0}r_{x}^{n}) = 2(\ve_{x}^{n},\rho^{n})+2(e_{x}^{n},r^{n}),
\]
and, using the definitions of $\rho^{n}$, $r^{n}$ in \eqref{equ583}, \eqref{equ591},
\begin{align*}
a_{1}^{n} & = (\ve_{x}^{n},H^{n}e^{n} + U^{n}\ve^{n}) + (e_{x}^{n},H^{n}\ve^{n}+3U^{n}e^{n})\\
& = -(\ve^{n},H_{x}^{n}e^{n}) -\tfrac{1}{2}(U_{x}^{n}\ve^{n},\ve^{n}) - \tfrac{3}{2}(U_{x}^{n}e^{n},e^{n}).
\end{align*}
Therefore
\begin{equation}
\abs{a_{1}^{n}} \leq C(\|\ve^{n}\|^{2} + \|e^{n}\|^{2}).
\label{equ5127}
\end{equation}
For $a_{2}^{n}$ we obtain
\begin{align*}
a_{2}^{n} & = (\ve^{n},Pf_{x}^{n}) + (e^{n},P_{0}g_{x}^{n}) + \|P\rho_{x}^{n}\|^{2} + \|P_{0}r_{x}^{n}\|^{2}\\
& = -(\ve_{x}^{n},f^{n}) - (e_{x}^{n},g^{n}) + \|P\rho_{x}^{n}\|^{2} + \|P_{0}r_{x}^{n}\|^{2}.
\end{align*}
Using the definitions of $f^{n}$, $g^{n}$ in \eqref{equ599}, \eqref{equ5107}, we see that
\begin{align*}
a_{2}^{n} &= -(\ve_{x}^{n},P_{0}r_{x}^{n}) - \tfrac{1}{2}(\ve_{x}^{n},H^{n}P_{0}r_{x}^{n})- \tfrac{1}{2}(\ve_{x}^{n},U^{n}P\rho_{x}^{n})
-(e_{x}^{n},P\rho_{x}^{n}) \\
\,\,\,\,&\,\,\,\,\,\,- \tfrac{1}{2}(e_{x}^{n},H^{n}P\rho_{x}^{n})-\tfrac{3}{2}(e_{x}^{n},U^{n}P_{0}r_{x}^{n})
+ \|P\rho_{x}^{n}\|^{2} + \|P_{0}r_{x}^{n}\|^{2}\\
& = -(\ve_{x}^{n},P_{0}r_{x}^{n}) - \tfrac{1}{2}((H^{n}\ve^{n})_{x}-H_{x}^{n}\ve^{n},P_{0}r_{x}^{n})
-\tfrac{1}{2}((U^{n}\ve^{n})_{x}-U_{x}^{n}\ve^{n},P\rho_{x}^{n})
-(e_{x}^{n},P\rho_{x}^{n}) \\
\,\,\,\,&\,\,\,\,\,\,- \tfrac{1}{2}((H^{n}e^{n})_{x}-H_{x}^{n}e^{n},P\rho_{x}^{n})
-\tfrac{3}{2}((U^{n}e^{n})_{x}-U_{x}^{n}e^{n},P_{0}r_{x}^{n}) + \|P\rho_{x}^{n}\|^{2} + \|P_{0}r_{x}^{n}\|^{2}\\
& = -(r_{x}^{n},P_{0}r_{x}^{n}) - (\rho_{x}^{n},P\rho_{x}^{n}) + \tfrac{1}{2}(H_{x}^{n}\ve^{n},P_{0}r_{x}^{n})
+ \tfrac{1}{2}(U_{x}^{n}\ve^{n},P\rho_{x}^{n}) + \tfrac{1}{2}(H_{x}^{n}e^{n},P\rho_{x}^{n})\\
\,\,\,\,&\,\,\,\,\,\,+\tfrac{3}{2}(U_{x}^{n}e^{n},P_{0}r_{x}^{n}) + \|P\rho_{x}^{n}\|^{2} + \|P_{0}r_{x}^{n}\|^{2}.
\end{align*}
Hence,
\[
a_{2}^{n} = \tfrac{1}{2}(H_{x}^{n}\ve^{n},P_{0}r_{x}^{n})
+ \tfrac{1}{2}(U_{x}^{n}\ve^{n},P\rho_{x}^{n}) + \tfrac{1}{2}(H_{x}^{n}e^{n},P\rho_{x}^{n})
+\tfrac{3}{2}(U_{x}^{n}e^{n},P_{0}r_{x}^{n}),
\]
and, from \eqref{equ585}, \eqref{equ593},
\begin{equation}
\abs{a_{2}^{n}} \leq \tfrac{C}{h}(\|\ve^{n}\|^{2} + \|e^{n}\|^{2}).
\label{equ5128}
\end{equation}
For $a_{3}^{n}$ we have
\begin{align*}
a_{3}^{n} & = \tfrac{1}{3}(\ve_{x}^{n},P_{0}g_{x}^{n})
+ \tfrac{1}{3}(\tfrac{1}{2}(H^{n}\ve^{n})_{x}-\tfrac{1}{2}H_{x}^{n}\ve^{n},P_{0}g_{x}^{n})
+ \tfrac{1}{3}(\tfrac{1}{2}(U^{n}\ve^{n})_{x} - \tfrac{1}{2}U_{x}^{n}\ve^{n},Pf_{x}^{n})\\
\,\,\,\,&\,\,\,\,\,\,+\tfrac{1}{3}(e_{x}^{n},Pf_{x}^{n})
+ \tfrac{1}{3}(\tfrac{1}{2}(H^{n}e^{n})_{x} - \tfrac{1}{2}H_{x}^{n}e^{n},Pf_{x}^{n})
+ \tfrac{1}{3}(\tfrac{3}{2}(U^{n}e^{n})_{x}-\tfrac{3}{2}U_{x}^{n}e^{n},P_{0}g_{x}^{n})\\
\,\,\,\,&\,\,\,\,\,\, - (P\rho_{x}^{n},Pf_{x}^{n}) - (P_{0}r_{x}^{n},P_{0}g_{x}^{n}),
\end{align*}
whence
\[
a_{3}^{n} = -\tfrac{2}{3}(P\rho_{x}^{n},Pf_{x}^{n}) - \tfrac{2}{3}(P_{0}r_{x}^{n},P_{0}g_{x}^{n})
-\tfrac{1}{6}(H_{x}^{n}\ve^{n}+3U_{x}^{n}e^{n},P_{0}g_{x}^{n}) - \tfrac{1}{6}(U_{x}^{n}\ve^{n}+H_{x}^{n}e^{n},Pf_{x}^{n}).
\]
Since
\begin{align*}
(P\rho_{x}^{n},Pf_{x}^{n}) + (P_{0}r_{x}^{n},P_{0}g_{x}^{n})
& = (P\rho_{x}^{n},(P_{0}r_{x}^{n})_{x} + \tfrac{1}{2}(H^{n}P_{0}r_{x}^{n})_{x} + \tfrac{1}{2}(U^{n}P\rho_{x}^{n})_{x})\\
\,\,\,\,&\,\,\,\,\,\,+ (P_{0}r_{x}^{n},(P\rho_{x}^{n})_{x} + \tfrac{1}{2}(H^{n}P\rho_{x}^{n})_{x} + \tfrac{3}{2}(U^{n}P_{0}r_{x}^{n})_{x})\\
& = -\tfrac{1}{2}(H^{n}(P\rho_{x}^{n})_{x},P_{0}r_{x}^{n}) - \tfrac{1}{2}(U^{n}(P\rho_{x}^{n})_{x},P\rho_{x}^{n})\\
\,\,\,\,&\,\,\,\,\,\,+\tfrac{1}{2}(P_{0}r_{x}^{n},(H^{n}P\rho_{x}^{n})_{x})
- \tfrac{3}{2}(U^{n}(P_{0}r_{x}^{n})_{x}, P_{0}r_{x}^{n}),
\end{align*}
we obtain
\begin{align*}
(P\rho_{x}^{n},Pf_{x}^{n}) + (P_{0}r_{x}^{n},P_{0}g_{x}^{n})
& = -\tfrac{1}{2}((H^{n}P\rho_{x}^{n})_{x} - H_{x}^{n}P\rho_{x}^{n},P_{0}r_{x}^{n})
+ \tfrac{1}{4}(U_{x}^{n}P\rho_{x}^{n},P\rho_{x}^{n})\\
\,\,\,\,&\,\,\,\,\,\,+\tfrac{1}{2}(P_{0}r_{x}^{n},(H^{n}P\rho_{x}^{n})_{x}) + \tfrac{3}{4}(U_{x}^{n}P_{0}r_{x}^{n},P_{0}r_{x}^{n}),
\end{align*}
and therefore
\[
(P\rho_{x}^{n},Pf_{x}^{n}) + (P_{0}r_{x}^{n},P_{0}g_{x}^{n}) =
\tfrac{1}{2}(H_{x}^{n}P\rho_{x}^{n},P_{0}r_{x}^{n}) + \tfrac{1}{4}(U_{x}^{n}P\rho_{x}^{n},P\rho_{x}^{n})
+ \tfrac{3}{4}(U_{x}^{n}P_{0}r_{x}^{n},P_{0}r_{x}^{n}).
\]
Thus, in view of \eqref{equ585}, \eqref{equ593}, \eqref{equ5110}, \eqref{equ5102} we see for $0\leq n\leq n^{*}$, that
\begin{equation}
\abs{a_{3}^{n}} \leq \tfrac{C}{h^{2}}(\|\ve^{n}\|^{2} + \|e^{n}\|^{2}).
\label{equ5129}
\end{equation}
For $a_{4}^{n}$  we obtain
\[
a_{4}^{n} = \tfrac{1}{3}(P\rho_{x}^{n},\widetilde{f}_{x}^{n})
+ \tfrac{1}{3}(P_{0}r_{x}^{n},\widetilde{g}_{x}^{n})
+ \tfrac{1}{4}\|Pf_{x}^{n}\|^{2} + \tfrac{1}{4}\|P_{0}g_{x}^{n}\|^{2},
\]
and
\begin{align*}
(P\rho_{x}^{n},\widetilde{f}_{x}^{n}) + (P_{0}r_{x}^{n},\widetilde{g}_{x}^{n}) & =
- ((P\rho_{x}^{n})_{x},P_{0}g_{x}^{n}) - \tfrac{1}{2}(H^{n}(P\rho_{x}^{n})_{x},P_{0}g_{x}^{n})
-\tfrac{1}{2}(U^{n}(P\rho_{x}^{n})_{x},Pf_{x}^{n}) \\
\,\,\,&\,\,\,\,\,\,-\!((P_{0}r_{x}^{n})_{x},Pf_{x}^{n})
-\tfrac{1}{2}(H^{n}(P_{0}r_{x}^{n})_{x},Pf_{x}^{n}) - \tfrac{3}{2}(U^{n}(P_{0}r_{x}^{n})_{x},P_{0}g_{x}^{n}),
\end{align*}
so that
\begin{align*}
(P\rho_{x}^{n},\widetilde{f}_{x}^{n}) + (P_{0}r_{x}^{n},\widetilde{g}_{x}^{n})& =
-(g_{x}^{n},P_{0}g_{x}^{n}) + \tfrac{1}{2}(H_{x}^{n}P\rho_{x}^{n},P_{0}g_{x}^{n}) + \tfrac{3}{2}(U_{x}^{n}P_{0}r_{x}^{n},P_{0}g_{x}^{n})\\
\,\,\,&\,\,\,\,\,\,-\!(f_{x}^{n},Pf_{x}^{n})
+\tfrac{1}{2}(H_{x}^{n}P_{0}r_{x}^{n},Pf_{x}^{n}) + \tfrac{1}{2}(U_{x}^{n}P\rho_{x}^{n},Pf_{x}^{n}).
\end{align*}
Therefore
\begin{equation}
a_{4}^{n} = \tfrac{-1}{12}(\|Pf_{x}^{n}\|^{2} + \|P_{0}g_{x}^{n}\|^{2}) + \widetilde{a}_{4}^{n},
\label{equ5130}
\end{equation}
where
\begin{equation}
\abs{\widetilde{a}_{4}^{n}} \leq \tfrac{C}{h^{3}}(\|\ve^{n}\|^{2} + \|e^{n}\|^{2}),
\label{equ5131}
\end{equation}
for $0\leq n\leq n^{*}$. \\
For $a_{5}^{n}$ it holds that
\[
a_{5}^{n} = - \tfrac{1}{6}(Pf_{x}^{n},\widetilde{f}_{x}^{n}) - \tfrac{1}{6}(P_{0}g_{x}^{n},\widetilde{g}_{x}^{n}),
\]
and
\begin{align*}
(Pf_{x}^{n},\widetilde{f}_{x}^{n}) + (P_{0}g_{x}^{n},\widetilde{g}_{x}^{n}) & =
(Pf_{x}^{n}, (P_{0}g_{x}^{n})_{x} + \tfrac{1}{2}(H^{n}P_{0}g_{x}^{n})_{x} + \tfrac{1}{2}(U^{n}Pf_{x}^{n})_{x})\\
\,\,\,&\,\,\,\,\,\,+ (P_{0}g_{x}^{n}, (Pf_{x}^{n})_{x} + \tfrac{1}{2}(H^{n}Pf_{x}^{n})_{x} + \tfrac{3}{2}(U^{n}P_{0}g_{x}^{n})_{x})\\
& = \tfrac{-1}{2}(H^{n}(Pf_{x}^{n})_{x},P_{0}g_{x}^{n}) + \tfrac{1}{4}(U_{x}^{n}Pf_{x}^{n},Pf_{x}^{n})
+ \tfrac{1}{2}(P_{0}g_{x}^{n},(H^{n}Pf_{x}^{n})_{x})\\
\,\,\,&\,\,\,\,\,\,+\tfrac{3}{4}(U_{x}^{n}P_{0}g_{x}^{n},P_{0}g_{x}^{n})\\
& = \tfrac{1}{2}(H_{x}^{n}Pf_{x}^{n},P_{0}g_{x}^{n}) + \tfrac{1}{4}(U_{x}^{n}Pf_{x}^{n},Pf_{x}^{n})
+ \tfrac{3}{4}(U_{x}^{n}P_{0}g_{x}^{n},P_{0}g_{x}^{n}).
\end{align*}
So, from \eqref{equ5102}, \eqref{equ5110},
\begin{equation}
\abs{a_{5}^{n}} \leq \tfrac{C}{h^{4}}(\|\ve^{n}\|^{2} + \|e^{n}\|^{2}),,
\label{equ5132}
\end{equation}
for $0\leq n\leq n^{*}$. \par
Finally, for $a_{6}^{n}$ we see that
\[
a_{6}^{n} = \tfrac{1}{36}\|P\widetilde{f}_{x}^{n}\|^{2} + \tfrac{1}{36}\|P_{0}\widetilde{g}_{x}^{n}\|^{2}.
\]
But since
\[
\widetilde{f}_{x}^{n} = (P_{0}g_{x}^{n})_{x} + \tfrac{1}{2}(H^{n}P_{0}g_{x}^{n})_{x} + \tfrac{1}{2}(U^{n}Pf_{x}^{n})_{x},
\]
we have
\[
\|\widetilde{f}_{x}^{n}\| \leq \tfrac{C}{h}(\|P_{0}g_{x}^{n}\| + \|Pf_{x}^{n}\|).
\]
Similarly, since
\[
\widetilde{g}_{x}^{n} = (Pf_{x}^{n})_{x} + \tfrac{1}{2}(H^{n}Pf_{x}^{n})_{x} + \tfrac{3}{2}(U^{n}P_{0}g_{x}^{n})_{x},
\]
it follows that
\[
\|\widetilde{g}_{x}^{n}\| \leq \tfrac{C}{h}(\|Pf_{x}^{n}\| + \|P_{0}g_{x}^{n}\|).
\]
Therefore
\begin{equation}
\abs{a_{6}^{n}} \leq \tfrac{C_{0}}{h^{2}}(\|P_{0}g_{x}^{n}\|^{2} + \|Pf_{x}^{n}\|^{2}),
\label{equ5133}
\end{equation}
for $0\leq n\leq n^{*}$ and for some constant $C_{0}$ independent of $h$ and $k$. Hence, from \eqref{equ5126}-\eqref{equ5133}
we obtain
\[
\|\gamma^{n}\|^{2} + \|\sigma^{n}\|^{2} \leq (1+C_{\lambda}k)(\|\ve^{n}\|^{2} + \|e^{n}\|^{2})
+ k^{4}(C_{0}\lambda^{2} - \tfrac{1}{12})(\|P_{0}g_{x}^{n}\|^{2} + \|Pf_{x}^{n}\|^{2}),
\]
and therefore, for $\lambda\leq \lambda_{0}=\sqrt{1/(12C_{0})}$ it holds that
\begin{equation}
\|\gamma^{n}\|^{2} + \|\sigma^{n}\|^{2} \leq (1+C_{\lambda}k)(\|\ve^{n}\|^{2} + \|e^{n}\|^{2}),
\label{equ5134}
\end{equation}
and
\begin{equation}
\|\gamma^{n}\| + \|\sigma^{n}\| \leq C(\|\ve^{n}\| + \|e^{n}\|),
\label{equ5135}
\end{equation}
for $0\leq n\leq n^{*}$. From \eqref{equ5124}, \eqref{equ5125} we obtain
\begin{align*}
\|\ve^{n+1}\|^{2} + \|e^{n+1}\|^{2}&=\|\gamma^{n}\|^{2} + \|\sigma^{n}\|^{2}
+ \tfrac{k}{3}[(\gamma^{n},\omega^{n}) + (\sigma^{n},w^{n}) + (\omega^{n},\delta_{1}^{n}) + (w^{n},\delta_{2}^{n})]\\
\,\,\,&\,\,\,\,\, + 2(\gamma^{n},\delta_{1}^{n}) + 2(\sigma^{n},\delta_{2}^{n})
+ \tfrac{k^{2}}{36}[\|\omega^{n}\|^{2} + \|w^{n}\|^{2}]
+ \|\delta_{1}^{n}\|^{2} + \|\delta_{2}^{n}\|^{2}.
\end{align*}
But from \eqref{equ5119}, \eqref{equ5123}, \eqref{equ5135} and Lemma 5.10 we see that
\[
k(\|\gamma^{n}\|\|\omega^{n}\| + \|\sigma^{n}\|\|w^{n}\|) \leq Ck(\|\ve^{n}\|^{2} + \|e^{n}\|^{2}),
\]
and
\begin{align*}
(k\|\omega^{n}\| + \|\gamma^{n}\|)\|\delta_{1}^{n}\| + (k\|w^{n}\| + \|\sigma^{n}\|)\|\delta_{2}^{n}\|
& \leq Ck(k^{3} + h^{r-1})(\|\ve^{n}\| + \|e^{n}\|) \\
& \leq Ck(\|\ve^{n}\|^{2} + \|e^{n}\|^{2}) + Ck(k^{3} + h^{r-1})^{2}.
\end{align*}
Hence, finally
\[
\|\ve^{n+1}\|^{2} + \|e^{n+1}\|^{2} \leq (1 + C_{\lambda}k)(\|\ve^{n}\|^{2} + \|e^{n}\|^{2}) + Ck(k^{3} + h^{r-1})^{2}.
\]
Therefore, from Gronwall's lemma it follows that
\[
\|\ve^{n}\|^{2} + \|e^{n}\|^{2} \leq \widetilde{C}_{1}(\|\ve^{0}\|^{2} + \|e^{0}\|^{2}) + \widetilde{C}_{2}(k^{3}+h^{r-1})^{2},
\]
i.e.
\[
\|\ve^{n}\| + \|e^{n}\| \leq  C(k^{3} + h^{r-1}),
\]
for $0\leq n\leq n^{*}+1$. Using the inverse properties of the spaces $S_{h}$, $S_{h,0}$ and the fact that $r\geq 3$
we conclude that $n^{*}$ was not maximal. Hence we may go up to $n^{*}=M-1$ and the conclusion of the proposition follows.
\end{proof}
We close this section by presenting the results of a relevant numerical experiment. We solve the nonhomogeneous SSW system
with exact solutions given by the functions $\eta(x,t)=\exp (2t)(\cos (\pi x) + x + 2)$,
$u(x,t)=\exp(xt)(\sin(\pi x) + 5x^{2}(x-1))$, for $0\leq x\leq 1$, $t\geq 0$, using cubic splines on a uniform mesh
on $[0,1]$ with $h=1/N$ for the spatial discretization and the Shu-Osher scheme with $k=h/10$ for time stepping.
\captionsetup[subtable]{labelformat=empty,position=top,margin=1pt,singlelinecheck=false}
\scriptsize
\begin{table}[h]
\subfloat[$L^{2}$-errors]{
\begin{tabular}[h]{ | c | c | c | c | c | }\hline
$N$   &   $\eta$         &  $order$  &  $u$               &  $order$  \\ \hline
$40$  &  $0.1094(-6)$   &            &  $0.1672(-6)$    &          \\ \hline
$80$  &  $0.8004(-8)$   &  $3.7733$  &  $0.1961(-7)$    &  $3.0915$ \\ \hline
$120$ &  $0.1886(-8)$   &  $3.5645$  &  $0.5738(-8)$    &  $3.0310$  \\ \hline
$160$ &  $0.7094(-9)$   &  $3.3994$  &  $0.2410(-8)$    &  $3.0152$  \\ \hline
$200$ &  $0.3407(-9)$   &  $3.2872$  &  $0.1232(-8)$    &  $3.0090$  \\ \hline
$240$ &  $0.1897(-9)$   &  $3.2127$  &  $0.7120(-9)$    &  $3.0058$  \\ \hline
$280$ &  $0.1165(-9)$   &  $3.1619$  &  $0.4481(-9)$    &  $3.0042$  \\ \hline
$320$ &  $0.7674(-10)$  &  $3.1260$  &  $0.3000(-9)$    &  $3.0032$  \\ \hline
$360$ &  $0.5325(-10)$  &  $3.1016$  &  $0.2107(-9)$    &  $3.0023$  \\ \hline
$400$ &  $0.3849(-10)$  &  $3.0826$  &  $0.1535(-9)$    &  $3.0024$  \\ \hline
$440$ &  $0.2873(-10)$  &  $3.0659$  &  $0.1153(-9)$    &  $3.0018$  \\ \hline
$480$ &  $0.2201(-10)$  &  $3.0631$  &  $0.8883(-10)$   &  $3.0010$  \\ \hline
$520$ &  $0.1725(-10)$  &  $3.0463$  &  $0.6986(-10)$   &  $3.0017$  \\ \hline
$560$ &  $0.1378(-10)$  &  $3.0279$  &  $0.5592(-10)$   &  $3.0030$  \\ \hline
$600$ &  $0.1119(-10)$  &  $3.0247$  &  $0.4546(-10)$   &  $3.0022$  \\ \hline
$640$ &  $0.9226(-11)$  &  $2.9846$  &  $0.3744(-10)$   &  $3.0074$  \\ \hline
\end{tabular}
}\,\,\,
\subfloat[$L^{\infty}$-errors]{
\begin{tabular}[h]{ | c | c | c | c | c | }\hline
$N$   &    $\eta$       &  $order$   &      $u$        &   $order$    \\ \hline
$40$  &  $0.3242(-6)$   &            &  $0.3471(-6)$   &              \\ \hline
$80$  &  $0.2553(-7)$   &  $3.6665$  &  $0.3666(-7)$   &  $3.2450$   \\ \hline
$120$ &  $0.6012(-8)$   &  $3.5670$  &  $0.1028(-7)$   &  $3.1332$   \\ \hline
$160$ &	 $0.2226(-8)$   &  $3.4532$  &  $0.4221(-8)$   &  $3.0938$   \\ \hline
$200$ &  $0.1045(-8)$   &  $3.3880$  &  $0.2121(-8)$   &  $3.0834$   \\ \hline
$240$ &  $0.5671(-9)$   &  $3.3546$  &  $0.1211(-8)$   &  $3.0767$   \\ \hline
$280$ &  $0.3406(-9)$   &  $3.3063$  &  $0.7570(-9)$   &  $3.0456$   \\ \hline
$320$ &  $0.2201(-9)$   &  $3.2701$  &  $0.5044(-9)$   &  $3.0402$   \\ \hline
$360$ &  $0.1502(-9)$   &  $3.2474$  &  $0.3525(-9)$   &  $3.0418$   \\ \hline
$400$ &  $0.1070(-9)$   &  $3.2137$  &  $0.2557(-9)$   &  $3.0491$   \\ \hline
$440$ &  $0.7901(-10)$  &  $3.1842$  &  $0.1913(-9)$   &  $3.0407$   \\ \hline
$480$ &  $0.6011(-10)$  &  $3.1424$  &  $0.1471(-9)$   &  $3.0207$   \\ \hline
$520$ &  $0.4684(-10)$  &  $3.1155$  &  $0.1155(-9)$   &  $3.0202$   \\ \hline
$560$ &  $0.3725(-10)$  &  $3.0913$  &  $0.9230(-10)$  &  $3.0280$   \\ \hline
$600$ &  $0.3019(-10)$  &  $3.0458$  &  $0.7485(-10)$  &  $3.0370$   \\ \hline
$640$ &  $0.2470(-10)$  &  $3.1086$  &  $0.6152(-10)$  &  $3.0396$   \\ \hline
\end{tabular}
}\\
\subfloat[$H^{1}$-errors]{
\begin{tabular}[h]{ | c | c | c | c | c | }\hline
$N$   &    $\eta$      &  $order$   &      $u$       &   $order$   \\ \hline
$40$  &  $0.2022(-4)$  &            &  $0.1713(-4)$  &            \\ \hline
$80$  &  $0.2515(-5)$  &  $3.0076$  &  $0.2159(-5)$  &  $2.9885$   \\ \hline
$120$ &	 $0.7451(-6)$  &  $2.9998$  &  $0.6428(-6)$	 &  $2.9880$   \\ \hline
$160$ &	 $0.3150(-6)$  &  $2.9926$  &  $0.2723(-6)$	 &  $2.9858$   \\ \hline
$200$ &  $0.1617(-6)$  &  $2.9883$  &  $0.1400(-6)$  &  $2.9826$  \\ \hline
$240$ &  $0.9388(-7)$  &  $2.9830$  &  $0.8129(-7)$  &  $2.9797$  \\ \hline
$280$ &  $0.5932(-7)$  &  $2.9783$  &  $0.5138(-7)$  &  $2.9768$  \\ \hline
$320$ &  $0.3987(-7)$  &  $2.9753$  &  $0.3454(-7)$  &  $2.9736$  \\ \hline
$360$ &  $0.2809(-7)$  &  $2.9714$  &  $0.2434(-7)$  &  $2.9705$  \\ \hline
$400$ &  $0.2055(-7)$  &  $2.9667$  &  $0.1781(-7)$  &  $2.9678$  \\ \hline
$440$ &  $0.1550(-7)$  &  $2.9637$  &  $0.1342(-7)$  &  $2.9646$  \\ \hline
$480$ &  $0.1198(-7)$  &  $2.9617$  &  $0.1037(-7)$  &  $2.9615$  \\ \hline
$520$ &  $0.9451(-8)$  &  $2.9580$  &  $0.8186(-8)$  &  $2.9589$  \\ \hline
$560$ &  $0.7594(-8)$  &  $2.9510$  &  $0.6576(-8)$  &  $2.9558$  \\ \hline
$600$ &  $0.6199(-8)$  &  $2.9418$  &  $0.5365(-8)$  &  $2.9502$  \\ \hline
$640$ &  $0.5123(-8)$  &  $2.9551$  &  $0.4433(-8)$  &  $2.9550$  \\ \hline
\end{tabular}
}
\caption{$L^{2}$-,$L^{\infty}$-, and $H^{1}$-errors and orders of convergence, cubic splines on a uniform mesh
with $h=1/N$ and Shu-Osher time stepping with $k=h/10$, \eqref{eqssw}}
\label{tbl52}
\end{table}
\normalsize
Table \ref{tbl52} shows the $L^{2}$-, $L^{\infty}$- and $H^{1}$-errors and associated rates of convergence for this problem at
$T=0.5$ as $N$ is increased. The rate of convergence in $L^{2}$ stabilizes to about $3$ for both components of the solution,
which is the expected temporal rate, as the experimental spatial rate is four in view of the numerical results in
Table \ref{tbl42}. The $L^{\infty}$-errors converge at a rate which appears to be equal to $3$ again
(we expect a $O(k^{3} + h^{4})$ behaviour), and so do the $H^{1}$-errors as well, for which the expected error is of
$O(k^{3} + h^{3})$.
\section{Remarks}
\subsection{Periodic boundary conditions}
In this section we consider the {\it{periodic}} initial-value problem for the usual and the symmetric shallow-water systems,
which we discretize using the standard Galerkin method with periodic splines of order $r\geq 2$ on a uniform mesh.
Using suitable {\it{quasiinterpolants}} of smooth periodic functions in the space of periodic splines, cf. \cite{twe},
we will prove optimal-order $L^{2}$-error estimates for the semidiscrete approximations of both systems. A similar
error analysis in the case of Boussinesq (i.e. dispersive) systems was done in \cite{adm}.
For the purposes of the present subsection we shall denote, for integer $k\geq 0$, by $H_{per}^{k}$ the usual,
$L^{2}$-based, real Sobolev space of periodic functions on $[0,1]$ with associated norm $\|\cdot\|_{k}$, and by
$C_{per}^{k}$ the space of periodic functions in $C^{k}[0,1]$. \par
We consider the periodic initial-value problem for the shallow-water systems. In the case of the usual system
we seek $\eta=\eta(x,t)$, $u=u(x,t)$, 1-periodic in $x$ for all $t\in [0,T]$, such that
\begin{align}
\begin{aligned}
\eta_{t} & + u_{x} + (\eta u)_{x} = 0, \notag\\
u_{t} & + \eta_{x} + uu_{x} = 0,\notag
\end{aligned}
\quad  x\in [0,1],\,\,\, t\in [0,T],
\tag{SW$_{per}$} \label{eqswper}
\\
\eta(x,0) =\eta_{0}(x), \quad u(x,0)=u_{0}(x), \quad x\in [0,1],\hspace{-9pt}&  \notag
\end{align}
where $\eta_{0}$, $u_{0}$ are given smooth 1-periodic functions. The analogous problem for the symmetric system is
\begin{align}
\begin{aligned}
\eta_{t} & + u_{x} + \tfrac{1}{2}(\eta u)_{x} = 0, \notag\\
u_{t} & + \eta_{x} + \tfrac{3}{2}uu_{x} + \tfrac{1}{2}\eta\eta_{x} = 0,\notag
\end{aligned}
\quad  x\in [0,1],\,\,\, t\in [0,T],
\tag{SSW$_{per}$} \label{eqsswper}
\\
\eta(x,0) =\eta_{0}(x), \quad u(x,0)=u_{0}(x), \quad x\in [0,1],\hspace{23pt}&  \notag
\end{align}
where again $\eta(\cdot,t)$, $u(\cdot,t)$ are 1-periodic for $0\leq t\leq T$ and $\eta_{0}$, $u_{0}$ given smooth
1-periodic functions. We shall assume that \eqref{eqswper} has a unique smooth enough solution on $[0,T]$ and that there
exists a positive $\alpha$ such that $1+\eta(x,t)\geq \alpha>0$ for $x\in [0,1]$, $t\in [0,T]$. Similarly, it will be
assumed that \eqref{eqsswper} has a unique smooth enough solution for $0\leq t\leq T$. For a theory of local existence-uniqueness
of solutions of \eqref{eqswper} we refer the reader to \cite{rtt}. \par
Let $N$ be a positive integer, $h=1/N$, and $x_{j}=jh$, $0\leq j\leq N$. For integer $r\geq 2$ let $S_{h}$ be the
$N$-dimensional space of smooth 1-periodic splines, i.e.
\[
S_{h}=\{\phi \in C_{per}^{r-2}[0,1] : \phi\big|_{[x_{j},x_{j+1}]} \in \mathbb{P}_{r-1}\,, 1\leq j\leq N-1\}.
\]
It is well known that $S_{h}$ has the approximation property that given $w\in H_{per}^{s}$, where $1\leq s\leq r$, there
exists a $\chi\in S_{h}$ such that
\begin{equation}
\sum_{j=0}^{s-1}h^{j}\|w-\chi\|_{j}\leq Ch^{s}\|w\|_{s}, \quad 1\leq s\leq r,
\label{eq61}
\end{equation}
where $C$ is a constant independent of $h$ and $w$. In addition, the inverse inequalities \eqref{eq23} and \eqref{eq24}
hold in the present framework as well. Following Thom\'{e}e and Wendroff, \cite{twe}, one may construct a basis $\{\phi\}_{j=1}^{N}$
of $S_{h}$, with supp$(\phi_{j})=O(h)$, such that for a sufficiently smooth 1-periodic function $w$, the associated
{\it{quasiinterpolant}}
\[
Q_{h}w = \sum_{j=1}^{N}w(x_{j})\phi_{j},
\]
satisfies
\begin{equation}
\|w-Q_{h}w\| \leq Ch^{r}\|w^{(r)}\|.
\label{eq62}
\end{equation}
In addition, it follows from \cite{twe} that the basis $\{\phi\}_{j=1}^{N}$ may be chosen so that the following properties hold:\\
(i)\,\, If $\psi\in S_{h}$, then
\begin{equation}
\|\psi\| \leq Ch^{-1}\max_{1\leq i\leq N}\abs{(\psi,\phi_{i})}.
\label{eq63}
\end{equation}
(ii)\,\, Let $w$ be a sufficiently smooth 1-periodic function and $\nu$, $\kp$ integers such that $0\leq \nu, \kp\leq r-1$. Then
\begin{equation}
\bigl((Q_{h}w)^{(\nu)},\phi_{i}^{(\kp)}\bigr) = (-1)^{\kp}hw^{(\nu+\kp)}(x_{i}) + O(h^{2r+j-\nu-\kp}), \quad 1\leq i\leq N,
\label{eq64}
\end{equation}
where $j=1$ if $\nu+\kp$ is even, and $j=2$ if $\nu+\kp$ is odd.\\
(iii)\,\, Let $f$, $g$ be sufficiently smooth 1-periodic functions and $\nu$ and $\kp$ as in (ii) above.
Let
\[
\beta_{i}=\bigl(f(Q_{h}g)^{(\nu)},\phi_{i}^{(\kp)}\bigr) - (-1)^{\kp}\bigl(Q_{h}[(fg^{(\nu)})^{(\kp)}],\phi_{i}\bigr),
\quad 1\leq i\leq N.
\]
Then
\begin{equation}
\max_{1\leq i\leq N}\abs{\beta_{i}} = O(h^{2r + j - \nu-\kp}),
\label{eq65}
\end{equation}
where $j$ as in (ii). \par
The semidiscretizations of the two systems are defined as follows. In the case of \eqref{eqswper} we seek $\eta_{h}$,
$u_{h} : [0,T]\to S_{h}$ satisfying
\begin{equation}
\begin{aligned}
(\eta_{ht},\phi) & + (u_{hx},\phi) + ((\eta_{h}u_{h})_{x},\phi) = 0, \quad \forall \phi \in S_{h},\,\, 0\leq t\leq T,\\
(u_{ht},\chi) & + (\eta_{hx},\chi) + (u_{h}u_{hx},\chi) = 0, \quad \forall \chi \in S_{h},\,\, 0\leq t\leq T,
\end{aligned}
\label{eq66}
\end{equation}
\[
\eta_{h}(0) = \eta_{0,h}, \quad u_{h}(0) = u_{0,h},
\]
where $\eta_{0,h}$, $u_{0,h}\in S_{h}$ are any approximations of $\eta_{0}$, $u_{0}$ in $S_{h}$ satisfying
$\|\eta_{0,h} - \eta_{0}\| + \|u_{0,h} - u_{0}\| = O(h^{r})$. The analogous semidiscrete i.v.p. for
\eqref{eqsswper} is
\begin{equation}
\begin{aligned}
(\eta_{ht},\phi) & + (u_{hx},\phi) + \tfrac{1}{2}((\eta_{h}u_{h})_{x},\phi) = 0, \quad \forall \phi \in S_{h},\,\, 0\leq t\leq T,\\
(u_{ht},\chi) & + (\eta_{hx},\chi) + \tfrac{1}{2}(\eta_{h}\eta_{hx},\chi) +
\tfrac{3}{2}(u_{h}u_{hx},\chi) = 0, \quad \forall \chi \in S_{h},\,\, 0\leq t\leq T,
\end{aligned}
\label{eq67}
\end{equation}
\[
\eta_{h}(0) = \eta_{0,h}, \quad u_{h}(0) = u_{0,h},
\]
with $\eta_{0,h}$, $u_{0,h}$ as above. It is clear that \eqref{eq66} has a unique solution locally in time and due to
the conservation property \eqref{eq210}, which holds for solutions of \eqref{eq67} as well, \eqref{eq67} has a unique solution
in any temporal interval $[0,T]$. \par
The error analysis in the case of \eqref{eqsswper} is straightforward due to the symmetry of the system. We first estimate
a truncation error for the system \eqref{eq67} defined for all $t\in [0,T]$ in terms of the quasiinterpolants of $\eta$ and $u$.
\begin{lemma} Let $(\eta,u)$ be the solution of \eqref{eqsswper} and $H=Q_{h}\eta$, $U=Q_{h}u$. Define
$\psi$ and $\zeta \in S_{h}$ so that for $0\leq t\leq T$
\begin{align}
(H_{t},\phi) & + (U_{t},\phi) + \tfrac{1}{2}((HU)_{x},\phi) = (\psi,\phi), \quad \forall \phi \in S_{h},
\label{eq68} \\
(U_{t},\chi) & + (H_{t},\chi) + \tfrac{1}{2}(HH_{x},\chi) + \tfrac{3}{2}(UU_{x},\chi) = (\zeta,\chi), \quad
\forall \chi \in S_{h}.
\label{eq69}
\end{align}
Then, there is a constant $C$ independent of $h$, such that
\begin{equation}
\|\psi(t)\| + \|\zeta(t)\| \leq Ch^{r}, \quad 0\leq t\leq T.
\label{eq610}
\end{equation}
\end{lemma}
\begin{proof} Applying \eqref{eq64} and \eqref{eq68} and using the first p.d.e. of \eqref{eqsswper} yields for $1\leq i\leq N$,
$t\in [0,T]$
\begin{align*}
(\psi,\phi_{i}) & = h(\eta_{t}+u_{x})(x_{i},t) + \tfrac{1}{2}((HU)_{x},\phi_{i}) + O(h^{2r+1})\\
& = -\tfrac{h}{2}(\eta u)_{x}(x_{i},t) + \tfrac{1}{2}((HU)_{x},\phi_{i}) + O(h^{2r+1})\\
& = \tfrac{1}{2}\bigl([(HU) - Q_{h}(\eta u)]_{x},\phi_{i}\bigr) + O(h^{2r+1}).
\end{align*}
Since
\[
HU - Q_{h}(\eta u) = \eta u - \ve u - e\eta + \ve e - Q_{h}(\eta u),
\]
where $\ve:=\eta -  H$, $e:=u-U$, we have, using \eqref{eq65}, for $1\leq i\leq N$
\begin{align*}
(\psi,\phi_{i}) = & \tfrac{1}{2}\bigl((\ve e)_{x},\phi_{i}\bigr) + \tfrac{1}{2}(\ve u,\phi_{i}')
 + \tfrac{1}{2}(e\eta,\phi_{i}') + \tfrac{1}{2}\bigl((\eta u)_{x}-Q_{h}[(\eta u)_{x}],\phi_{i}\bigr)\\
& + \tfrac{1}{2}\bigl(Q_{h}[(\eta u)_{x}] - [Q_{h}(\eta u)]_{x},\phi_{i}\bigr) + O(h^{2r+1}) \\
= & \tfrac{1}{2}\bigl((\ve e)_{x},\phi_{i}\bigr)
- \tfrac{1}{2}\bigl((\eta u)_{x} - Q_{h}[(\eta u)_{x}],\phi_{i}\bigr) + O(h^{2r+1}).
\end{align*}
Therefore, by \eqref{eq63} we obtain, using \eqref{eq61} and \eqref{eq62}
\[
\|\psi\| \leq C \|\ve\|_{1} \|e\|_{1} + O(h^{r}) \leq Ch^{r}.
\]
The analogous estimate for $\|\zeta\|$ follows along similar lines.
\end{proof}
We now proceed to prove an optimal-order $L^{2}$-error estimate for the solution of \eqref{eq67}.
\begin{proposition} Let $(\eta,u)$, $(\eta_{h},u_{h})$ be the solutions of \eqref{eqsswper}, \eqref{eq67}, respectively.
Then
\begin{equation}
\max_{0\leq t\leq T}(\|\eta-\eta_{h}\| + \|u-u_{h}\|) \leq Ch^{r}.
\label{eq611}
\end{equation}
\end{proposition}
\begin{proof}
Let $\theta:=H-\eta_{h}=Q_{h}\eta-\eta_{h}$ and $\xi:=U-u_{h}=Q_{h}u-u_{h}$. Then, from \eqref{eq67} and \eqref{eq68},
\eqref{eq69} we have for $t\in [0,T]$
\begin{align}
(\theta_{t},\phi) & + (\xi_{x},\phi) + \tfrac{1}{2}\bigl((H\xi + U\theta - \theta\xi)_{x},\phi\bigr) = (\psi,\phi),
\quad \forall \phi \in S_{h},
\label{eq612} \\
(\xi_{t},\chi) & + (\theta_{x},\chi) + \tfrac{1}{2}\bigl((H\theta)_{x} - \theta\theta_{x},\chi\bigr)
+ \tfrac{3}{2}\bigl((U\xi)_{x}-\xi\xi_{x},\chi\bigr) = (\zeta,\chi), \quad \forall \chi \in S_{h},
\label{eq613}
\end{align}
Taking $\phi=\theta$ in \eqref{eq612}, $\chi=\xi$ in \eqref{eq613}, adding the resulting equations, and using periodicity
we obtain for $0\leq t\leq T$
\begin{equation}
\tfrac{1}{2}\tfrac{d}{dt}(\|\theta\|^{2} + \|\xi\|^{2}) + \tfrac{1}{2}(H_{x}\theta,\xi) + 
\tfrac{1}{4}(U_{x}\theta,\theta) + \tfrac{3}{4}(U_{x}\xi,\xi)
\\
 = (\psi,\theta) + (\zeta,\chi).
\label{eq614}
\end{equation}
From \eqref{eq62} and the inverse inequalities we have for $0\leq t\leq T$ $\|H_{x}\|_{\infty}\leq C$,
$\|U_{x}\|_{\infty}\leq C$, where $C$ is independent of $h$. Therefore it follows from \eqref{eq610} and \eqref{eq614}
that for $0\leq t\leq T$
\[
\tfrac{1}{2}\tfrac{d}{dt}(\|\theta\|^{2} + \|\xi\|^{2}) \leq C(\|\theta\|^{2} + \|\xi\|^{2} + h^{2r}).
\]
An application of Gronwall's lemma, \eqref{eq62}, and our choice of $\eta_{0,h}$ and $u_{0,h}$ yield now the
desired estimate \eqref{eq611}.
\end{proof}
We now estimate the errors of the semidiscrerization of \eqref{eqswper}. As before we may prove
\begin{lemma} Let $(\eta,u)$ be the solution of \eqref{eqswper} and $H=Q_{h}\eta$, $U=Q_{h}u$. Define $\psi$, $\zeta \in S_{h}$
so that for $t\in [0,T]$
\begin{align}
(H_{t},\phi) & + (U_{x},\phi) + ((HU)_{x},\phi) = (\psi,\phi), \quad \forall \phi \in S_{h},
\label{eq615}\\
(U_{t},\chi) & + (H_{x},\chi) + (UU_{x},\chi) = (\zeta,\chi),\quad \forall \chi \in S_{h}.
\label{eq616}
\end{align}
Then, for some constant $C$ independent of $h$, we have
\begin{equation}
\|\psi(t)\| + \|\zeta(t)\| \leq Ch^{r}, \quad 0\leq t\leq T.
\label{eq617}
\end{equation}
\end{lemma}
The proof of the main error estimate for \eqref{eqswper} is not as straightforward as that of the symmetric system
but goes through if we use ideas from the proof of Proposition 2.2.
\begin{proposition} Let $(\eta,u)$ be the solution of \eqref{eqswper}. Then, for $h$ sufficiently small, \eqref{eq66}
has a unique solution $(\eta_{h},u_{h})$ for $0\leq t\leq T$, satisfying
\begin{equation}
\max_{0\leq t\leq T}(\|\eta - \eta_{h}\| + \|u - u_{h}\|) \leq Ch^{r}.
\label{eq618}
\end{equation}
\end{proposition}
\begin{proof}
We let again $\theta:=H-\eta_{h}=Q_{h}\eta - \eta_{h}$ and $\xi:=U-u_{h}=Q_{h}u - u_{h}$. Then, from \eqref{eq66}
and \eqref{eq615}-\eqref{eq616}, we have, while the solution of \eqref{eq66} exists,
\begin{align}
(\theta_{t},\phi) & + (\xi_{x},\phi) + ((H\xi + U\theta - \theta\xi)_{x},\phi) = (\psi,\phi),
\quad \forall \phi \in S_{h},
\label{eq619}\\
(\xi_{t},\chi) & + (\theta_{x},\chi) + ((U\xi)_{x} - \xi\xi_{x},\chi) = (\zeta,\chi), \quad \forall \chi \in S_{h}.
\label{eq620}
\end{align}
Putting $\phi=\theta$ in \eqref{eq619} and using periodicity we have
\begin{equation}
\tfrac{1}{2}\tfrac{d}{dt}\|\theta\|^{2} + (\xi_{x},\theta) + ((H\xi)_{x},\theta)
+ \tfrac{1}{2}(U_{x}\theta,\theta) - \tfrac{1}{2}(\xi_{x}\theta,\theta) = (\psi,\theta).
\label{eq621}
\end{equation}
Now, using the inverse inequalities and \eqref{eq62} we see that
\begin{equation}
((H\xi)_{x},\theta) = ((\eta\xi)_{x},\theta) + ((H-\eta)_{x}\xi,\theta) +
((H-\eta)\xi_{x},\theta) \leq ((\eta\xi)_{x},\theta) + C\|\xi\|\|\theta\|.
\label{eq622}
\end{equation}
Let $t_{h}>0$ denote a maximal value of $t$ such that $(\eta_{h},u_{h})$ exists and $\|\xi_{x}\|_{\infty}\leq 1$
for $0\leq t\leq t_{h}$, and suppose that $t_{h}<T$. From \eqref{eq621}, \eqref{eq622} and \eqref{eq617} we conclude
then that
\begin{equation}
\tfrac{1}{2}\tfrac{d}{dt}\|\theta\|^{2} - (\theta_{x},\gamma) \leq C(h^{r}\|\theta\| + \|\xi\|\|\theta\|), \quad
0\leq t\leq t_{h},
\label{eq623}
\end{equation}
where $\gamma:=(1+\eta)\xi$.\\
We put now in \eqref{eq620} $\chi=P\gamma=P[(1+\eta)\xi]$, where $P$ is the $L^{2}$-projection on $S_{h}$, and obtain
for $0\leq t\leq t_{h}$
\begin{equation}
(\xi_{t},(1+\eta)\xi) + (\theta_{x},P\gamma) = - ((U\xi)_{x}-\xi\xi_{x},P\gamma) + (\zeta,P\gamma).
\label{eq624}
\end{equation}
Now, using periodicity, we have
\begin{align*}
((U\xi)_{x},P\gamma) = & (U\xi_{x},P\gamma - \gamma) + (U_{x}\xi,P\gamma-\gamma) + (U_{x}\xi,(1+\eta)\xi)\\
& - \bigl((U(1+\eta))_{x},\xi^{2}\bigr).
\end{align*}
Using again the superapproximation property \eqref{eq223} which holds in the space of periodic splines as well,
the fact that $\|U\|_{1,\infty}\leq C$, and inverse properties we obtain from the above
\begin{equation}
((U\xi)_{x},P\gamma) \leq C\|\xi\|^{2}.
\label{eq625}
\end{equation}
Using, in addition, the fact that $\|\xi_{x}\|_{\infty}\leq 1$ in $[0,t_{h}]$, we also have
\begin{equation}
(\xi\xi_{x},P\gamma) = (\xi\xi_{x},P\gamma-\gamma) + (\xi\xi_{x},\gamma) \leq C\|\xi\|^{2}.
\label{eq626}
\end{equation}
Therefore, by \eqref{eq616}, \eqref{eq625}, \eqref{eq626}, and \eqref{eq624} we have
\begin{equation}
(\xi_{t},(1+\eta)\xi) + (\theta_{x},P\gamma) \leq C(h^{r}\|\xi\| + \|\xi\|^{2}), \quad 0\leq t\leq t_{h}.
\label{eq627}
\end{equation}
Adding \eqref{eq623} and \eqref{eq627} we see that
\begin{align*}
\tfrac{1}{2}\tfrac{d}{dt}\|\theta\|^{2} + (\xi_{t},(1+\eta)\xi) + (\theta_{x},P\gamma -\gamma)\leq
C\bigl(h^{r}(\|\xi\| + \|\theta\|) + \|\xi\|^{2} + \|\theta\|^{2}\bigr).
\end{align*}
As in the proof of Proposition 2.2 we have, mutatis mutandis, that
\[
\|\theta(t)\| + \|\xi(t)\| \leq Ch^{r}, \quad 0\leq t\leq t_{h},
\]
for a constant $C$ independent of $h$. It follows that $\|\xi_{x}\|_{\infty}\leq Ch^{r-3/2}$, i.e. that $t_{h}$ is not maximal
if $h$ is chosen sufficiently small. The result of the proposition now follows in the standard manner.
\end{proof}
\subsection{Existence-uniqueness of solutions of \eqref{eqssw}} As was already mentioned in the Introduction, Petcu and
Temam, \cite{pt}, proved existence and uniqueness of solutions of \eqref{eqsw}. In our notation their result is
\begin{theorem} Let $u_{0}$, $\eta_{0}\in H^{2}$ and $\alpha$ be a constant such that $1+\eta_{0}(x)\geq 2\alpha>0$,
$x\in [0,1]$. Then there exists a $T_{*}>0$ depending on $\|u_{0}\|_{2}$, $\|\eta_{0}\|_{2}$, and a unique solution
$(u,\eta)$ of \eqref{eqsw} for $0\leq t\leq T_{*}$ such that $(u,\eta)\in L^{\infty}(0,T_{*};H^{2})$. Moreover
$1+\eta_{0}(x)\geq \alpha$ $\forall t\in (0,T_{*})$.
\end{theorem}
In the course of the proof it is shown that $u\in L^{\infty}(0,T;H^{2}\cap H_{0}^{1})$, that $u_{t}$,
$\eta_{t}\in L^{\infty}(0,T_{*};H^{1})$ and $\eta_{x}(0,t)=\eta_{x}(1,t)=0$ for $0<t<T_{*}$; it is also assumed that
$u_{0}(0)=u_{0}(1)=0$ and $\eta_{0}'(0)=\eta_{0}'(1)=0$. \par
Following the steps of the proof of Theorem 6.1 in \cite{pt}, it is straightforward to show that a similar result holds
for \eqref{eqssw}. In fact the proof is simpler due to the symmetry of the system in \eqref{eqssw}; we just state the result
\begin{theorem} Let $u_{0}\in H^{2}\cap H_{0}^{1}$, $\eta_{0}\in H^{2}$ with $\eta_{0}'(0)=\eta_{0}'(1)=0$, and $\alpha$ be
a constant such that $1+\tfrac{1}{2}\eta_{0}(x)\geq 2\alpha>0$, $x\in [0,1]$. Then, there exists a $T_{*}>0$ depending on
$\|u_{0}\|_{2}$, $\|\eta_{0}\|_{2}$, and a unique solution $(u,\eta)$ of \eqref{eqssw} for $0\leq t<T_{*}$ such that
$u\in L^{\infty}(0,T_{*};H^{2}\cap H_{0}^{1})$, $\eta\in L^{\infty}(0,T_{*};H^{2})$, $u_{t}$, $\eta_{t}\in L^{\infty}(0,T_{*};H^{1})$.
Moreover $1+\tfrac{1}{2}\eta(x,t)\geq \alpha>0$ for $(x,t)\in [0,1]\times [0,T_{*})$ and $\eta_{x}(0,t)=\eta_{x}(1,t)=0$ for
$0\leq t<T_{*}$.
\end{theorem}
\subsection{Comparison of SW and SSW for small-amplitude solutions} As is well known, the system of shallow water equations
(which has been written thusfar in terms of nondimensional, unscaled variables) is derived from the 2D-Euler equations for
surface water waves in the long-wave regime, i.e. when $\sigma:=\tfrac{h_{0}}{\lambda}<<1$, where $h_{0}$ is the depth of the
horizontal channel and $\lambda$ is a typical wavelength. Under the additional assumption that the wave amplitude is small, i.e.
when $\ve:=\tfrac{\alpha}{h_{0}}<<1$, one may formally derive, cf. \cite{p}, \cite{bcs}, from the Euler equations one of the
original versions of a Boussinesq system written in nondimensional, scaled variables in the form
\begin{align*}
\eta_{t} & + u_{x} + \ve(\eta u)_{x} + \tfrac{\sigma^{2}}{3}u_{xxx} = O(\ve\sigma^{2},\sigma^{4}),\\
u_{t} & + \eta_{x} + \ve uu_{x} = O(\ve\sigma^{2},\sigma^{4}),
\end{align*}
where $u$ denotes the horizontal velocity at the free surface and $\eta$ is the displacement of the free surface from its rest
position. (Here $x\in\mathbb{R}$ is proportional to length along the channel and $t\geq 0$ is proportional to time.) If we
assume that the dispersive effects are small, in the sense that $\ve\sim\sigma$, we obtain
\begin{align*}
\eta_{t} & + u_{x} + \ve(\eta u)_{x} = O(\ve^{2}),\\
u_{t} & + \eta_{x} + \ve uu_{x} = O(\ve^{3}),
\end{align*}
from which, replacing the right-hand side be zero, we get the system
\begin{align}
\eta_{t} & + u_{x} + \ve (\eta u)_{x} = 0,
\label{eq628}\\
u_{t} & + \eta_{x} + \ve uu_{x} = 0,
\label{eq629}
\end{align}
a scaled version of the shallow water equations valid for small-amplitude waves in the regime $\ve\sim\sigma << 1$. \par
Making in \eqref{eq628}-\eqref{eq629} the nonlinear change of variable $v=u(1+\tfrac{\ve}{2}\eta)$, used in \cite{bcl} in the
context of dispersive waves, we obtain that
\begin{align*}
\eta_{t} & + v_{x} + \tfrac{\ve}{2}(\eta v)_{x} = O(\ve^{2}),\\
v_{t} & + \eta_{x} + \tfrac{\ve}{2}\eta\eta_{x} + \tfrac{3\ve}{2} vv_{x} = O(\ve^{2}),
\end{align*}
i.e. that $(\eta,v)$ satisfy a scaled version of the symmetric shallow water equations which is formally equivalent as a model
up to $O(\ve^{2})$ terms to the scaled shallow water system. \par
Let now $(\eta^{s},u^{s})$ denote the solution of the Cauchy problem for the symmetric system
\begin{align}
\eta_{t}^{s} & + u_{x}^{s} + \tfrac{\ve}{2}(\eta^{s}u^{s})_{x} = 0,
\label{eq630}\\
u_{t}^{s} & + \eta_{x}^{s} + \tfrac{\ve}{2}\eta^{s}\eta_{x}^{s} + \tfrac{3\ve}{2}u^{s}u_{x}^{s}=0,
\label{eq631}
\end{align}
for $x\in\mathbb{R}$, $t\geq 0$, with initial data
\begin{equation}
\eta^{s}(x,0)=\eta_{0}^{s}(x), \quad u^{s}(x,0)=u_{0}^{s}(x), \quad  x\in\mathbb{R},
\label{eq632}
\end{equation}
and consider the Cauchy problem for the system \eqref{eq628}-\eqref{eq629} with initial conditions
\begin{equation}
\eta(x,0)=\eta_{0}(x), \quad u(x,0)=u_{0}(x), \quad  x\in\mathbb{R}.
\label{eq633}
\end{equation}
Using the theory of local existence for initial-value problems for quasilinear hyperbolic systems, \cite{m}, \cite{t}, and
examining the proofs of Proposition 4 and Corollary 2 of \cite{bcl}, we may conclude that the results of \cite{bcl} hold also
in the non-dispersive case, and specifically for the initial-value problems \eqref{eq628}, \eqref{eq629}, \eqref{eq633} and
\eqref{eq630}-\eqref{eq632}. In particular, if $(\eta_{0}^{s},u_{0}^{s})\in \bigl(H^{\ell}(\mathbb{R})\bigr)^{2}$ for some
$\ell>3/2$, there exists $T_{0}>0$ independent of $\ve$ and a unique solution
$(\eta^{s},u^{s}) \in C\bigl( [0,\tfrac{T_{0}}{\ve}];(H^{\ell}(\mathbb{R}))^{2}\bigr)$ of \eqref{eq630}-\eqref{eq632}.
In addition, $\|(\eta^{s},u^{s})\|_{W^{k,\infty}\bigl(0,\tfrac{T_{0}}{\ve};(H^{\ell-k}(\mathbb{R}))^{2}\bigr)}\leq C_{0}$
for some constant $C_{0}$ independent of $\ve$ and for all $k$ such that $\ell-k>3/2$. An entirely analogous result (with different
constants $T_{0}'$ and $C_{0}'$) holds for the solutions $(\eta,u)$ of the initial-value problem for the shallow-water system
\eqref{eq628},\eqref{eq629},
\eqref{eq633}. Under these hypotheses and if
\begin{equation}
\eta_{0}^{s}=\eta_{0}, \quad u_{0}^{s}=u_{0}(1+\tfrac{\ve}{2}\eta_{0}),
\label{eq634}
\end{equation}
and $T=\min (T_{0},T_{0}')$, there exists $\ve_{0}>0$ such that for $0<\ve<\ve_{0}$
\begin{equation}
\|\eta-\eta^{s}\|_{L^{\infty}(0,t;H^{\ell}(\mathbb{R}))}
+ \|u-(1-\tfrac{\ve}{2}\eta^{s})u^{s}\|_{L^{\infty}(0,t;H^{\ell}(\mathbb{R}))} \leq C\ve^{2}t,
\label{eq635}
\end{equation}
for all $t\in [0,\tfrac{T}{\ve}]$ and some constant $C$ independent of $\ve$.
If therefore the initial data in \eqref{eq632} and \eqref{eq633} are related by \eqref{eq634}, the solutions $(\eta,u)$
and $(\eta^{s},u^{s})$ of the two systems (transformed as in \eqref{eq635}) differ by an amount of at
most $O(\ve^{2}t)$ for $t$ up to $O(T/\ve)$.
(Note that initially smooth solutions of both systems are expected in general to develop singularities after times of
$O(1/\ve)$.) \par
We will now investigate by computational means whether an estimate of the form \eqref{eq635} holds also in the case of
initial-boundary value problems for the two systems when they are posed on a finite interval, say on $[0,1]$, with the velocity
variable equal to zero at the endpoints. We consider therefore the ibvp's $(\text{SW}_{\ve})$ consisting of \eqref{eq628} and
\eqref{eq629} for $x\in [0,1]$, $t\geq 0$, initial conditions of the form \eqref{eq633} for $x\in [0,1]$ and boundary
conditions $u(0,t)=u(1,t)=0$ for $t\geq 0$, and the analogous problem $(\text{SSW}_{\ve})$ consisting of \eqref{eq630}-\eqref{eq632}
for $x\in [0,1]$, $t\geq 0$, and boundary conditions $u^{s}(0,t)=u^{s}(1,t)=0$, $t\geq 0$. (Note that the change of variables
$u^{s}=u(1+\tfrac{\ve}{2}\eta)$ preserves the homogeneous boundary conditions on the velocity.) We solve both problems numerically
using cubic splines on a uniform mesh in space coupled with the third-order Shu-Osher temporal discretization with
$h=10^{-3}$, $k=10^{-3}$, taking as initial conditions for $(\text{SW}_{\ve})$ the functions
\[
\eta_{0}(x)=1, \quad u_{0}(x)=x(x-1), \quad x\in [0,1],
\]
and for $(\text{SSW}_{\ve})$  $\eta_{0}^{s}=\eta_{0}$, $u_{0}^{s}=u_{0}(1+\tfrac{\ve}{2}\eta_{0})$. In Figure \ref{fig61} we plot
as functions of $t$ the quantities
\begin{align*}
L_{2}-error & :=\|\eta - \eta^{s}\| + \|u-u^{s}(1-\tfrac{\ve}{2}\eta^{s})\|,\\
H^{1}-error & :=\|\eta - \eta^{s}\|_{1} + \|u-u^{s}(1-\tfrac{\ve}{2}\eta^{s})\|_{1},
\end{align*}
where $(\eta,u)$ and $(\eta^{s},u^{s})$ are the numerical approximations of the solutions of $(\text{SW}_{\ve})$ and $(\text{SSW}_{\ve})$,
respectively, evolving from the stated initial conditions for various values of $\ve$. For values of $\ve$ up to $10^{-3}$ the
temporal profile is practically linear up to about $t=300$ and the same is observed for $\ve=10^{-2}$ up to about $t=100$ for
the $L^{2}$-error. In the case $\ve=10^{-2}$ - note the change of scale in the $t$-axis in the figure - a singularity starts
developing after about $t=120$ (when $t\ve=O(1)$). In Table \ref{tbl61} we present the values of the $L^{2}$- and $H^{1}$-errors
from the same computations at $t=50$, $100$, $200$, $300$ as functions of diminishing $\ve$ in the range where the models are valid,
i.e. before singularities emerge. The computed numerical orders of convergence in $\ve$ for each fixed $t$ are practically equal to 2.
\captionsetup[subfloat]{labelformat=empty,position=bottom,singlelinecheck=true}
\begin{figure}[h]
  \begin{center}
    \subfloat[$L^{2}-error$ vs time]{\includegraphics[scale=0.38]{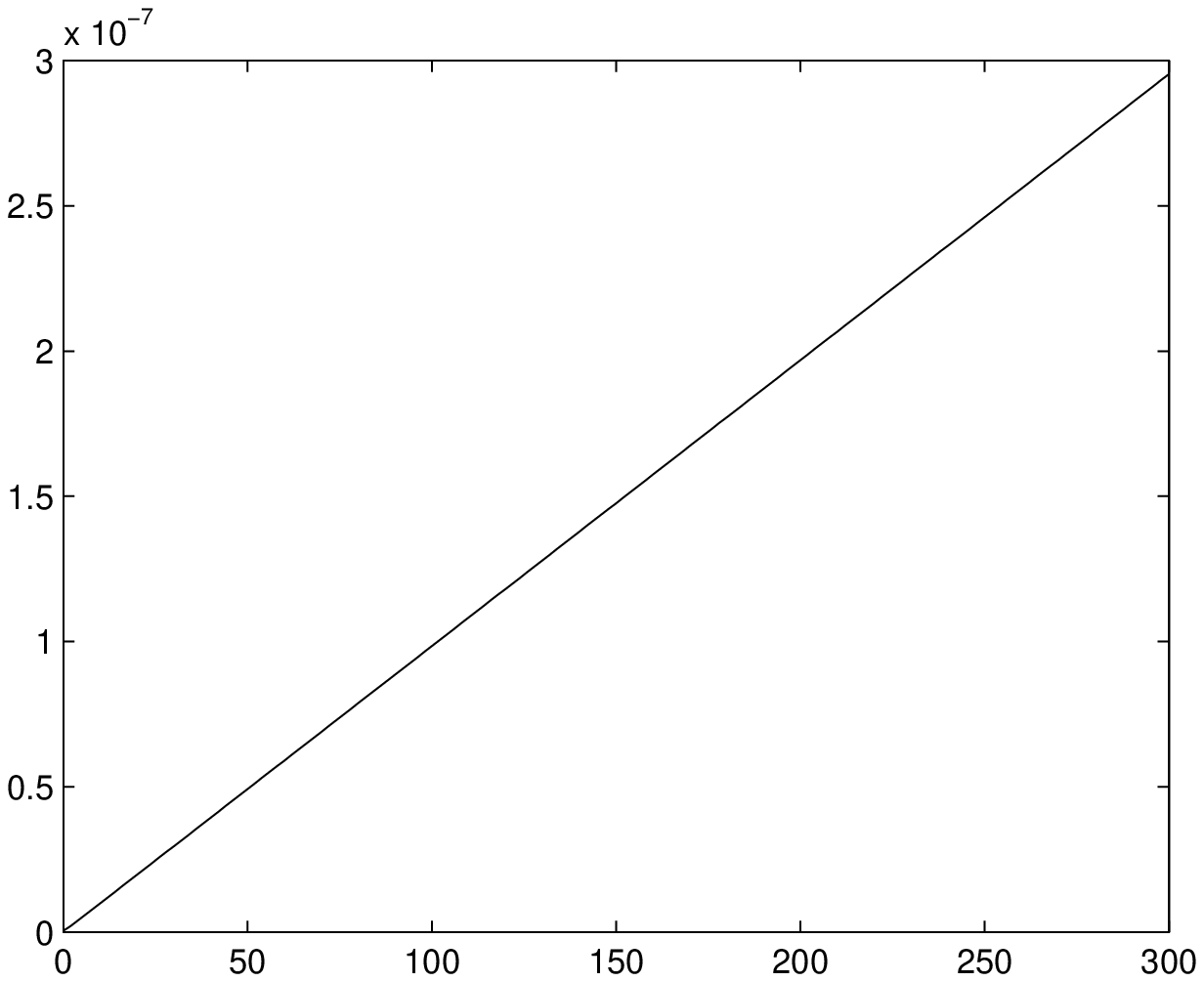}}\qquad
    \subfloat[$H^{1}-error$ vs time]{\includegraphics[scale=0.38]{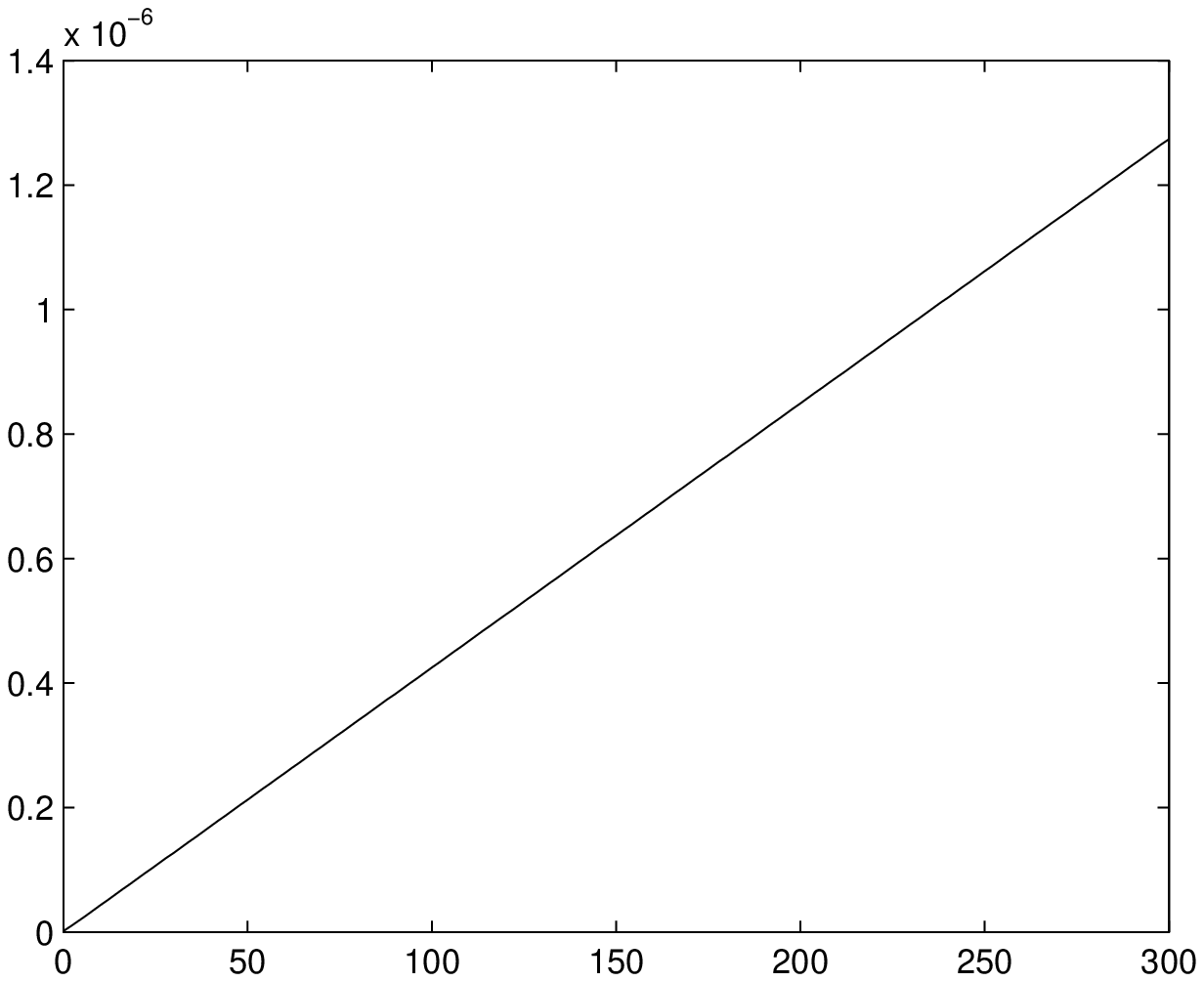}}\\
    $\ve = 10^{-4}$ \\
    \subfloat[$L^{2}-error$ vs time]{\includegraphics[scale=0.38]{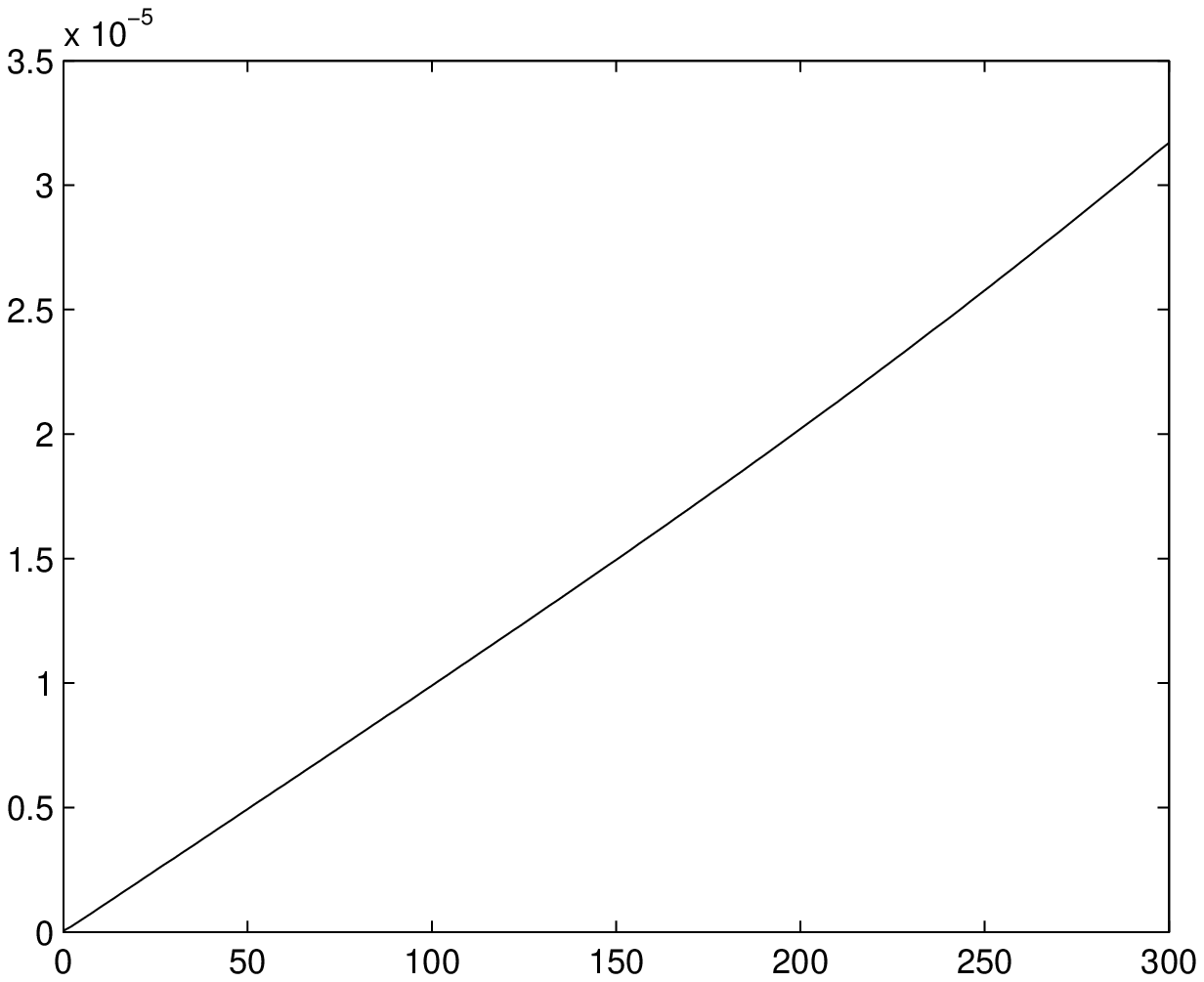}}\qquad
    \subfloat[$H^{1}-error$ vs time]{\includegraphics[scale=0.38]{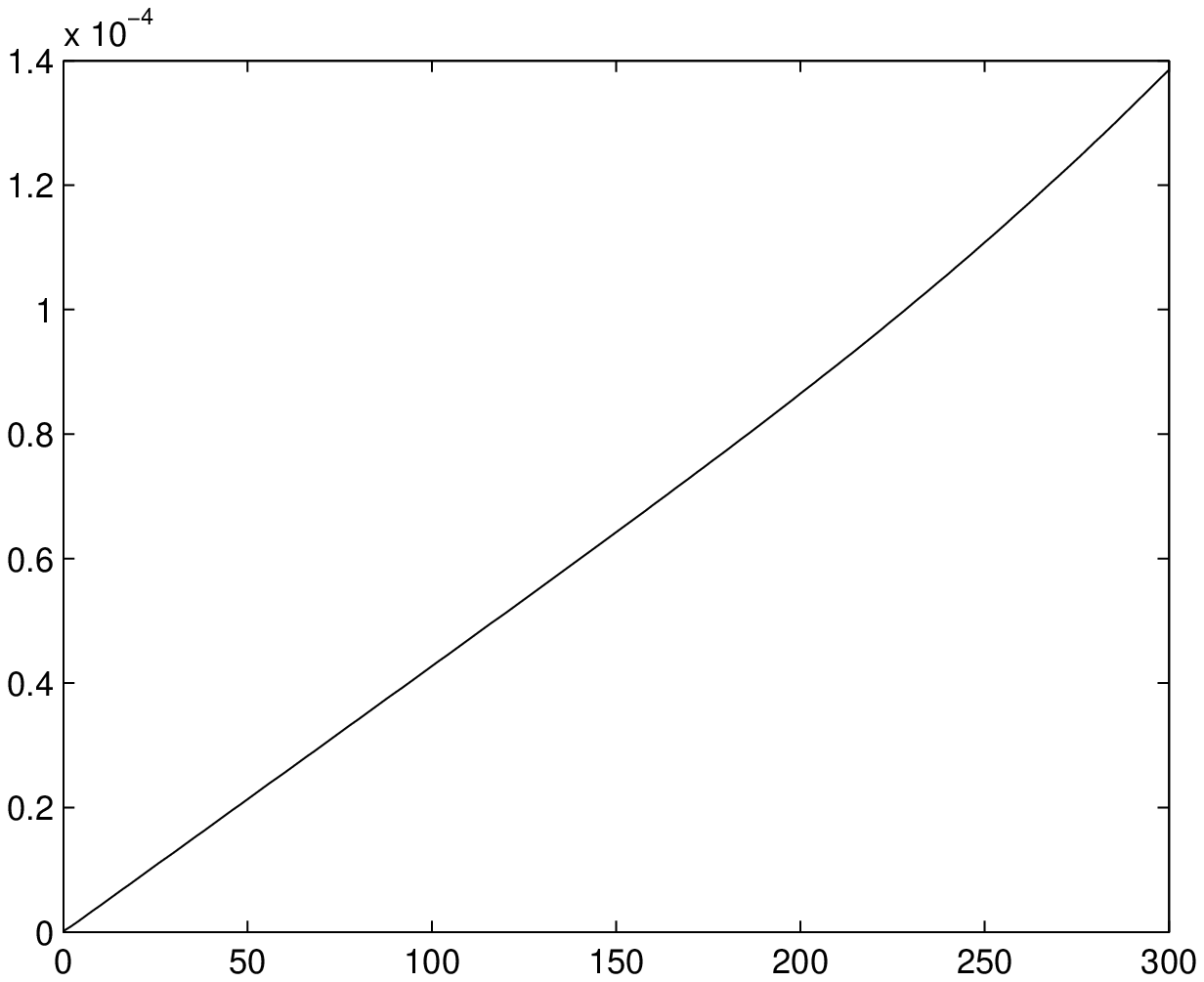}}\\
    $\ve = 10^{-3}$
   \end{center}
\end{figure}
\newpage
\captionsetup[subfloat]{labelformat=empty,position=bottom,singlelinecheck=true}
\begin{figure}[h]
  \begin{center}
    \subfloat[$L^{2}-error$ vs time]{\includegraphics[scale=0.38]{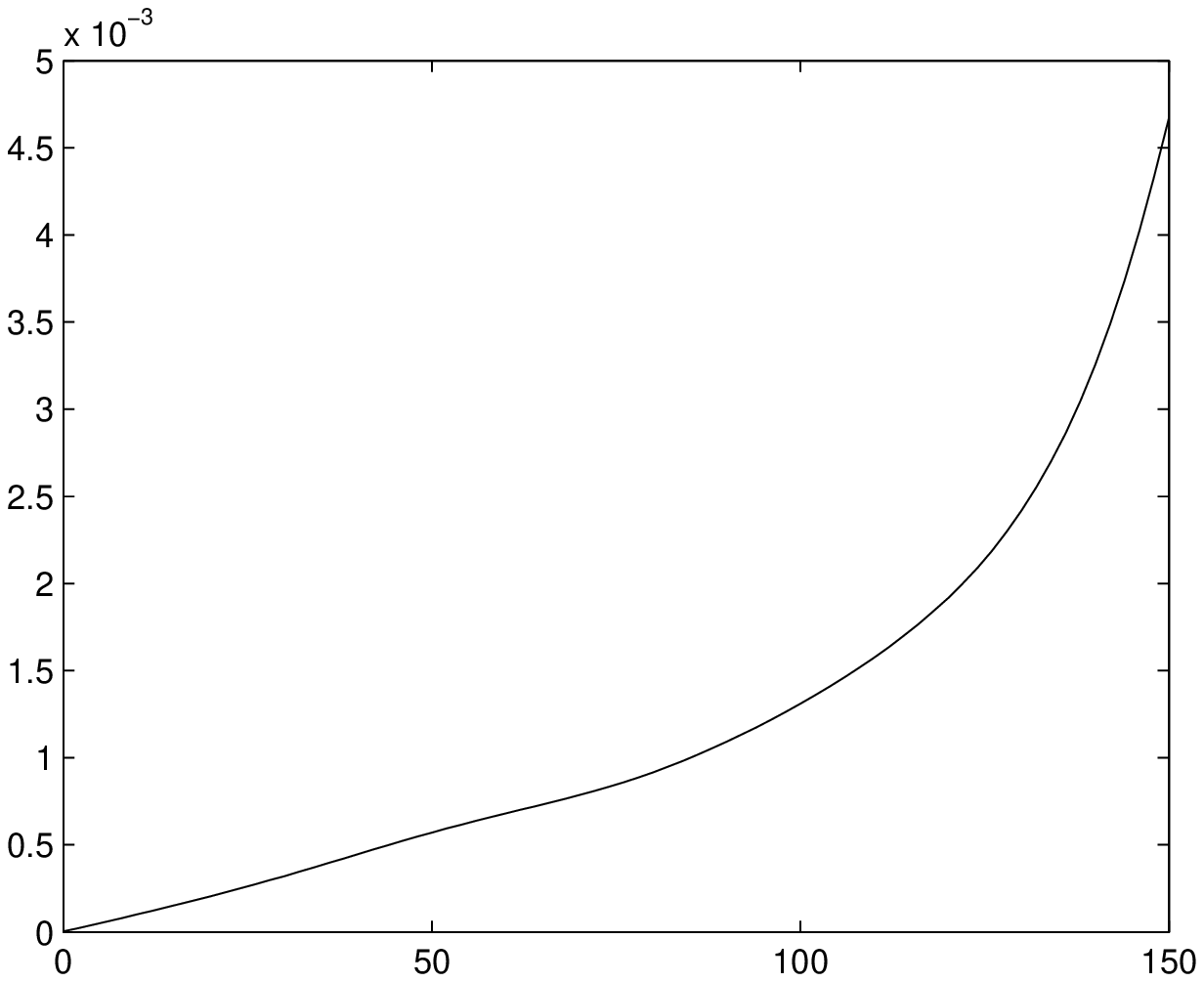}}\qquad
    \subfloat[$H^{1}-error$ vs time]{\includegraphics[scale=0.38]{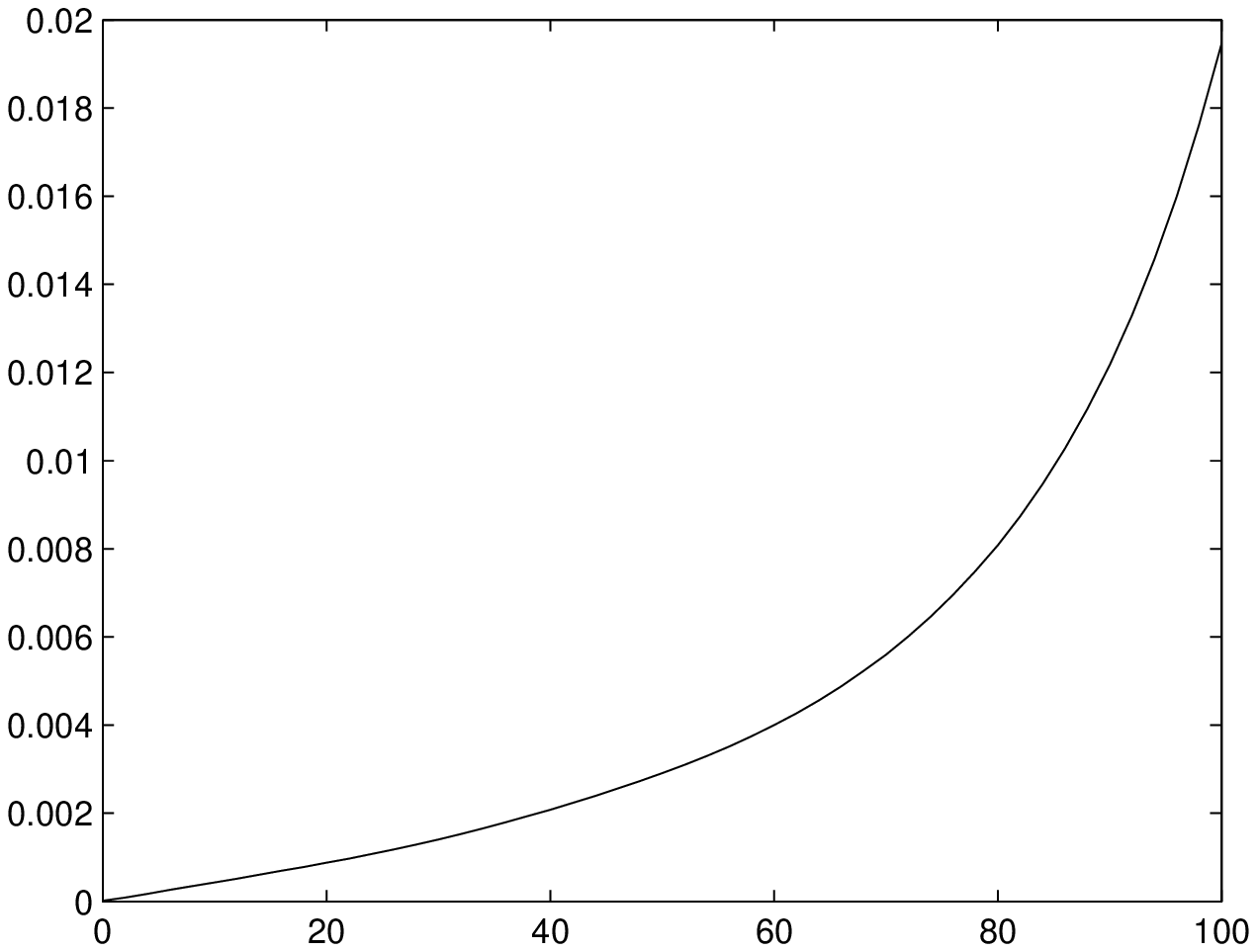}}\\
    $\ve = 10^{-2}$ \\
  \end{center}
  \caption{$L^{2}$ and $H^{1}$ norms of the differences $\bigl(\eta-\eta^{s},u-u^{s}(1-\tfrac{\ve}{2}\eta^{s})\bigr)$
  (``$L^{2}$,$H^{1}$-errors'') as functions of $t$ for $\ve=10^{-4}$, $10^{-3}$, $10^{-2}$.}
  \label{fig61}
\end{figure}
\scriptsize
\begin{table}[h]
\begin{center}
\begin{tabular}[h]{ | c | c | c | c | c | c | c | c | c |  }\hline
& \multicolumn{2}{c |}{$\text{time}=50$} &
\multicolumn{2}{c |}{$\text{time}=100$} & \multicolumn{2}{c |}{$\text{time}=200$} & \multicolumn{2}{c |}{$\text{time}=300$}
 \\ \hline
$\ve$ &  $L^{2}$-$error$  &  $order$ &  $L^{2}$-$error$ & $order$ &
$L^{2}$-$error$ & $order$ & $L^{2}$-$error$  &  $order$ \\ \hline
$10^{-2}$   &  $5.7064(-04)$  &          &  $1.3102(-03)$ &
            &                 &          &                &    \\ \hline
$10^{-3}$   &  $4.9362(-06)$  & $2.0630$ &  $9.8992(-06)$ & $2.1217$ & $2.0206(-05)$ &          & $3.1706(-05)$ &          \\ \hline
$10^{-4}$   &  $4.9224(-08)$  & $2.0012$ &  $9.8437(-08)$ & $2.0024$ & $1.9688(-07)$ & $2.0113$ & $2.9536(-07)$ & $2.0308$ \\ \hline
$10^{-5}$   &  $4.9214(-10)$  & $2.0001$ &  $9.8415(-10)$ & $2.0001$ & $1.9682(-09)$ & $2.0001$ & $2.9523(-09)$ & $2.0002$ \\ \hline
\end{tabular}\vspace{8pt} \\
\begin{tabular}[h]{ | c | c | c | c | c | c | c | c | c |  }\hline
& \multicolumn{2}{c |}{$\text{time}=50$} &
\multicolumn{2}{c |}{$\text{time}=100$} & \multicolumn{2}{c |}{$\text{time}=200$} & \multicolumn{2}{c |}{$\text{time}=300$}
 \\ \hline
$\ve$ &  $H^{1}$-$error$  &  $order$ &  $H^{1}$-$error$ & $order$ &
$H^{1}$-$error$ & $order$ & $H^{1}$-$error$  &  $order$ \\ \hline
$10^{-2}$   &  $2.9087(-03)$  &          &  $1.9463(-02)$ &
            &                 &          &                &    \\ \hline
$10^{-3}$   &  $2.1336(-05)$  & $2.1346$ &  $4.2709(-05)$ & $2.6587$ & $8.6497(-05)$ &          & $1.3853(-04)$ &          \\ \hline
$10^{-4}$   &  $2.1254(-07)$  & $2.0017$ &  $4.2482(-07)$ & $2.0023$ & $8.4936(-07)$ & $2.0079$ & $1.2740(-06)$ & $2.0364$ \\ \hline
$10^{-5}$   &  $2.1305(-09)$  & $2.0000$ &  $4.2485(-09)$ & $2.0000$ & $8.4888(-09)$ & $2.0003$ & $1.2725(-08)$ & $2.0005$ \\ \hline
\end{tabular}
\end{center}
\caption{Data of Figure 6.1. $L^{2}$- and $H^{1}$- errors at $t=50$, $100$, $200$, $300$ as functions of $\ve$, and
order of convergence as $\ve\to 0$.}
\label{tbl61}
\end{table}


\begin{thebibliography}{10000}  
\bibitem[AD]{ad} D.C. Antonopoulos and V.A. Dougalis, Error estimates for Galerkin approximations of the `classical' Boussinesq
system, Math. Comp. 82(2013), 689--717.
\bibitem[AD arXiv]{adarxiv} D.C. Antonopoulos and V.A. Dougalis, Notes on error estimates for Galerkin approximations of
the `classical' Boussinesq system and related hyperbolic problems, arXiv:1008.4248, 2010.
\bibitem[ADM]{adm} D.C. Antonopoulos, V.A. Dougalis and D.E. Mitsotakis, Galerkin approximations of periodic solutions of
Boussinesq systems, Bull. Greek Math. Soc. 57(2010), 13--30.
\bibitem[BCS]{bcs} J.L. Bona, M. Chen, and J.-C. Saut, Boussinesq equations and other systems for small-amplitude long waves in
nonlinear dispersive media: I. Derivation and linear theory, J. Nonlinear Sci. 12(2002), 283--318.
\bibitem[BCL]{bcl} J.L. Bona, T. Colin, and D. Lannes, Long wave aspproximations for water waves, Arch. Rational Mech. Anal.
178(2005), 373--410.
\bibitem[BEF]{bef} E. Burman, A. Ern, and M.A. Fern\'{a}ndez, Explicit Runge-Kutta schemes and finite elements with symmetric
stabilization for first-order linear PDE systems, SIAM J. Numer. Anal. 48(2010), 2019-2042.
\bibitem[DDW]{ddw} J. Douglas, Jr., T. Dupont, and L. Wahlbin, Optimal $L^{\infty}$ error estimates for
Galerkin approximations to solutions of two-point boundary value problems, Math. Comp.
29(1975), 475--483.
\bibitem[D1]{d1} T. Dupont, Galerkin methods for modelling gas pipelines, Lecture Notes in Math., v. 430(1974), 112--130,
Springer-Verlag, Berlin, 1974.
\bibitem[D2]{d2} T. Dupont, Galerkin methods for first order hyperbolics:
an example, SIAM J. Numer. Anal. 10(1973), 890--899.
\bibitem[L]{l} Q. Lin, Full convergence for hyperbolic finite elements, in \emph{Discontinuous Galerkin Methods: Theory,
Computation and Applications}, B. Cockburn, G.E. Karniadakis, C.-W. Shu, eds., Springer, Berlin, Heidelberg, 2000, pp. 167--177.
\bibitem[M]{m} A, Majda, \emph{Compressible Fluid Flow and Systems of Conservation Laws in Several Space Variables},
Springer-Verlag, New York 1984.
\bibitem[P]{p} D.H. Peregrine, Equations for water waves and the approximation behind them, in \emph{Waves on Beaches and Resulting
Sediment Transport}, R.E. Meyer ed., Academic Press, New York 1972, pp. 95--121.
\bibitem[PT]{pt} M. Petcu and R. Temam, The one dimensional Shallow Water equations with Dirichlet boundary conditions,
Discr. Cont. Dyn. Syst., Series S, 4(2011), 209--222.
\bibitem[RTT]{rtt} J.M. Rakotson, R. Temam, and J. Tribbia, Remarks on the nonviscous shallow water equations,
Indiana Univ. Math. J., 57(2008), 2969--2998.
\bibitem[Sch]{sch} R. Schreiber, Finite element methods of high-order accuracy for singular two-point boundary value
problems with non-smooth solutions, SIAM J. Numer. Anal. 17(1980), 547--566.
\bibitem[SO]{so} C.-W. Shu and S. Osher, Efficient implementation of essentially non-oscillating shock-capturing
schemes, J. Comp. Phys. 77(1988), 439--471.
\bibitem[T]{t} M. Taylor, \emph{Partial Differential Equations III: Nonlinear Equations}, $2^{\text{nd}}$ ed., Springer,
New York 2011.
\bibitem[TWa]{twa} V. Thom\'{e}e and L.B. Wahlbin, Maximum-norm stability and error estimates in Galerkin methods for
parabolic equations in one space variable, Numer. Math. 41(1983), 345--371.
\bibitem[TWe]{twe} V. Thom\'{e}e and B. Wendroff, Convergence estimates for Galerkin methods for variable coefficient initial
value problems, SIAM J. Numer. Anal. 11(1974), 1059--1068.
\bibitem[W]{w} G.B. Whitham, \emph{Linear and Nonlinear Waves}, Wiley, New York 1974.
\bibitem[Y]{y} L. Ying, A second order explicit finite element scheme to multi-dimensional conservation laws and its
convergence, Science in China (Ser. A), 43(2000), 945--957.
\bibitem[ZL]{zl} A. Zhou and Q. Lin, Optimal and superconvergence estimates of the finite element method for a scalar
hyperbolic equation, Acta Mathematica Scientia, 14 (1994), 90--94.
\bibitem[ZS1]{zs1} Q. Zhang and C.-W. Shu, Error estimates to smooth solutions of Runge-Kutta Discontinuous Galerkin methods
for scalar conservation laws, SIAM J. Numer. Anal. 42(2004), 641--666.
\bibitem[ZS2]{zs2} Q. Zhang and C.-W. Shu, Error estimates to smooth solutions of Runge-Kutta Discontinuous Galerkin method for
symmetrizable systems of conservation laws, SIAM J. Numer. Anal. 44(2006), 1703--1720.
\bibitem[ZS3]{zs3} Q. Zhang and C.-W. Shu, Stability analysis and a priori error estimates to the third order explicit
Runge-Kutta discontinuous Galerkin method for scalar conservation laws, SIAM J. Numer. Anal. 48(2010), 1038--1064.
\end{thebibliography}
\end{document}